\documentclass[11pt,reqno]{amsart}
\usepackage[margin=1.2in, footskip=.5in]{geometry}
\usepackage[T1]{fontenc}
\usepackage[american]{babel}
\usepackage{amsmath,amsfonts,amssymb,amsthm,wasysym,bm}
\usepackage{mathrsfs}
\usepackage{hyperref}
\usepackage{graphicx}
\usepackage[justification=centering]{caption}
\usepackage{subcaption}
\usepackage{multicol}
\usepackage[table,xcdraw]{xcolor}
\usepackage{float}
\usepackage{enumitem}
\usepackage[europeanresistors, cuteinductors]{circuitikz} 
\newtheorem{theorem}{Theorem}[section]
\newtheorem*{theorem*}{Theorem}
\newtheorem{proposition}[theorem]{Proposition}
\newtheorem{corollary}[theorem]{Corollary}
\newtheorem{lemma}[theorem]{Lemma}
\newtheorem*{conjecture*}{Conjecture}
\theoremstyle{definition}
\newtheorem{definition}{Definition}[section]
\newtheorem{remark}{Remark}[section]
\newtheorem{assumption}{Assumption}
\newtheorem*{notation}{Notation}
\numberwithin{equation}{section}
\theoremstyle{remark}
\newtheorem{example}{Example}[section]
\newcommand{\mbbR}{{\mathbb R}}

\def\mcA{\mathcal{A}}
\def\mcC{\mathcal{C}}
\newcommand{\E}{\mathcal{E}}
\newcommand{\F}{\mathcal{F}}

\def\mcH{\mathcal{H}}
\def\mcR{\mathcal{R}}
\def\mcS{\mathcal{S}}

\newcommand{\SG}{\operatorname{SG}}
\newcommand{\I}{\operatorname{I}}

\newcommand{\supp}{\operatorname{supp}}
\newcommand{\dom}{\operatorname{dom}}

\newcommand{\rH}{r_{\text{H}}}
\newcommand{\rI}{r_{\I}}
\newcommand{\rSG}{r_{\SG}}

\def\ss#1{\scriptscriptstyle{#1}}
\def\sss#1{\scriptscriptstyle{(#1)}}

\title{Analysis on hybrid fractals}

\author[P. Alonso Ruiz]{Patricia Alonso Ruiz}
\address[P. Alonso Ruiz]{University of Connecticut, 341 Mansfield Road, Storrs CT 06226.\newline \texttt{patricia.alonso-ruiz@uconn.edu}}
\author[Y. Chen]{Yuming Chen}
\address[Y. Chen]{University of Hong Kong, 9 Pokfulam Road, Hong Kong SAR.\newline \texttt{chenym@hku.hk}}
\author[H. Gu]{Haotian Gu}
\address[H. Gu]{University of Hong Kong, 4A No.85 Smithfield Road, Kennedy Town, Hong Kong SAR\newline
\texttt{ght1368@hku.hk}}
\author[R.S. Strichartz]{Robert S. Strichartz}
\address[R.S. Strichartz]{Cornell University, 310 Malott Hall, Ithaca NY 14853. \texttt{str@math.cornell.edu}}
\author[Z. Zhou]{Zirui Zhou}
\address[Z. Zhou]{ University of California Berkeley, 2301 Durant Avenue, Berkeley, CA 94704\newline
\texttt{zirui$\_$zhou@berkeley.edu}}


\begin{document}
\begin{abstract}
We introduce hybrid fractals as a class of fractals constructed by gluing several fractal pieces in a specific manner and study energy forms and Laplacians on them. We consider in particular a hybrid based on the $3$-level Sierpinski gasket, for which we construct explicitly an energy form with the property that it does not ``capture'' the $3$-level Sierpinski gasket structure. This characteristic type of energy forms that ``miss'' parts of the structure of the underlying space are investigated in the more general framework of finitely ramified cell structures. The spectrum of the associated Laplacian and its asymptotic behavior in two different hybrids is analyzed theoretically and numerically. A website with further numerical data analysis is available at \url{http://www.math.cornell.edu/~harry970804/}.
\end{abstract}
\maketitle

\section{Introduction}
%
The term \textit{hybrid} refers to something that is a mixture of several other things and hence a \textit{hybrid fractal} is a fractal which is a mixture of other fractals; see Section~\ref{section:HF} for its precise definition. The present paper investigates several questions concerning the analysis, in particular the energy and Laplacian, that can be constructed on this type of sets.

\medskip

One of the basic tools to develop analysis on arbitrary metric spaces is the theory of \textit{resistance forms} due to Kigami~\cite{Kig89,Kig01}. These forms thus provide an essential mathematical structure to model physical phenomena on rough spaces, as for instance heat or wave propagation, that are the object of many investigations in the fractal setting~\cite{B+08,ABD+12,ADL14,A++17,AR17}. Our aim here is to study these structures on hybrid fractals. On finitely ramified fractals, see e.g.~\cite{Kig93,Kig01,HMT06,Tep08}, resistance forms typically arise as the limit of a sequence of energies on finite weighted graph approximations that carry along the intrinsic structure of the fractal. An appropriate choice of the weights (resistances) is crucial in order to obtain a meaningful limit and it is remarkable how far beyond the existence of that limit the consequences of this choice actually go. 

\medskip

In this regard, this paper addresses the following question: Does a resistance form reflect the intrinsic structure of the underlying space completely? Previous investigations~\cite{ARKT16,ARFK17} revealed the possibility that a resistance form may ``miss'' essential properties of the space on which it is defined, yielding an incomplete framework to treat analytic questions. More precisely, in the mentioned works a resistance form was constructed on the fractal set displayed in Figure~\ref{fig:SSG} which did not produce ``fractal analysis''. To obtain this type of information about a resistance form requires a more explicit expression than the usual limit of graph energies and, ultimately, this is tantamount to a characterization of its domain.

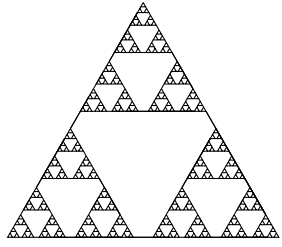
\begin{figure}[H]
\centering
\begin{tikzpicture}[scale=20/620]
\draw ($(90:2.75)$) --  ($(90+120:2.75)$) -- ($(90+120*2:2.75)$) -- ($(90:2.75)$) ($(90:2.75)+(120*2:2)$) to ++ ($(0:2)$) ($(90+120:2.75)+(0:2)$) to ++ ($(120:2)$)  ($(90+120*2:2.75)+(120:2)$) to ++ ($(120*2:2)$) ;
\draw ($(90:2.75)$) -- ($(90+120:2.75)+(60:6)$) ($(90+120*2:2.75)$) -- ($(90+120:2.75)+(0:6)$) ($(90:2.75)+(0:6)$) -- ($(90+120*2:2.75)+(60:6)$);
\foreach \a in {0,60} {
\draw ($(90:2.75)+(\a:6)$) --  ($(90+120:2.75)+(\a:6)$) -- ($(90+120*2:2.75)+(\a:6)$) -- ($(90:2.75)+(\a:6)$) ($(90:2.75)+(120*2:2)+(\a:6)$) to ++ ($(0:2)$) ($(90+120:2.75)+(0:2)+(\a:6)$) to ++ ($(120:2)$)  ($(90+120*2:2.75)+(120:2)+(\a:6)$) to ++ ($(120*2:2)$);
}
\draw ($(90:2.75)+(60:6)$) -- ($(90+120:2.75)+(60:6)+(60:13)$) ($(90+120*2:2.75)+(0:6)$) -- ($(90+120:2.75)+(0:6)+(0:13)$) ($(90:2.75)+(60:6)+(0:13)$) -- ($(90+120*2:2.75)+(0:6)+(60:13)$);
\foreach \b in {0,60} {
\draw ($(90:2.75)+(\b:13)$) --  ($(90+120:2.75)+(\b:13)$) -- ($(90+120*2:2.75)+(\b:13)$) -- ($(90:2.75)+(\b:13)$) ($(90:2.75)+(120*2:2)+(\b:13)$) to ++ ($(0:2)$) ($(90+120:2.75)+(0:2)+(\b:13)$) to ++ ($(120:2)$)  ($(90+120*2:2.75)+(120:2)+(\b:13)$) to ++ ($(120*2:2)$) ;
\draw ($(90:2.75)+(\b:13)$) -- ($(90+120:2.75)+(60:6)+(\b:13)$) ($(90+120*2:2.75)+(\b:13)$) -- ($(90+120:2.75)+(0:6)+(\b:13)$) ($(90:2.75)+(0:6)+(\b:13)$) -- ($(90+120*2:2.75)+(60:6)+(\b:13)$);
\foreach \a in {0,60} {
\draw ($(90:2.75)+(\a:6)+(\b:13)$) --  ($(90+120:2.75)+(\a:6)+(\b:13)$) -- ($(90+120*2:2.75)+(\a:6)+(\b:13)$) -- ($(90:2.75)+(\a:6)+(\b:13)$) ($(90:2.75)+(120*2:2)+(\a:6)+(\b:13)$) to ++ ($(0:2)$) ($(90+120:2.75)+(0:2)+(\a:6)+(\b:13)$) to ++ ($(120:2)$)  ($(90+120*2:2.75)+(120:2)+(\a:6)+(\b:13)$) to ++ ($(120*2:2)$) ;
}
}
\foreach \c in {0,60} {
\draw ($(90:2.75)+(\c:28)$) --  ($(90+120:2.75)+(\c:28)$) -- ($(90+120*2:2.75)+(\c:28)$) -- ($(90:2.75)+(\c:28)$) ($(90:2.75)+(120*2:2)+(\c:28)$) to ++ ($(0:2)$) ($(90+120:2.75)+(0:2)+(\c:28)$) to ++ ($(120:2)$)  ($(90+120*2:2.75)+(120:2)+(\c:28)$) to ++ ($(120*2:2)$) ;
\draw ($(90:2.75)+(\c:28)$) -- ($(90+120:2.75)+(60:6)+(\c:28)$) ($(90+120*2:2.75)+(\c:28)$) -- ($(90+120:2.75)+(0:6)+(\c:28)$) ($(90:2.75)+(0:6)+(\c:28)$) -- ($(90+120*2:2.75)+(60:6)+(\c:28)$);
\foreach \a in {0,60} {
\draw ($(90:2.75)+(\a:6)+(\c:28)$) --  ($(90+120:2.75)+(\a:6)+(\c:28)$) -- ($(90+120*2:2.75)+(\a:6)+(\c:28)$) -- ($(90:2.75)+(\a:6)+(\c:28)$) ($(90:2.75)+(120*2:2)+(\a:6)+(\c:28)$) to ++ ($(0:2)$) ($(90+120:2.75)+(0:2)+(\a:6)+(\c:28)$) to ++ ($(120:2)$)  ($(90+120*2:2.75)+(120:2)+(\a:6)+(\c:28)$) to ++ ($(120*2:2)$);
}
\draw ($(90:2.75)+(60:6)+(\c:28)$) -- ($(90+120:2.75)+(60:6)+(60:13)+(\c:28)$) ($(90+120*2:2.75)+(0:6)+(\c:28)$) -- ($(90+120:2.75)+(0:6)+(0:13)+(\c:28)$) ($(90:2.75)+(60:6)+(0:13)+(\c:28)$) -- ($(90+120*2:2.75)+(0:6)+(60:13)+(\c:28)$);
\foreach \b in {0,60} {
\draw ($(90:2.75)+(\b:13)+(\c:28)$) --  ($(90+120:2.75)+(\b:13)+(\c:28)$) -- ($(90+120*2:2.75)+(\b:13)+(\c:28)$) -- ($(90:2.75)+(\b:13)+(\c:28)$) ($(90:2.75)+(120*2:2)+(\b:13)+(\c:28)$) to ++ ($(0:2)$) ($(90+120:2.75)+(0:2)+(\b:13)+(\c:28)$) to ++ ($(120:2)$)  ($(90+120*2:2.75)+(120:2)+(\b:13)+(\c:28)$) to ++ ($(120*2:2)$) ;
\draw ($(90:2.75)+(\b:13)+(\c:28)$) -- ($(90+120:2.75)+(60:6)+(\b:13)+(\c:28)$) ($(90+120*2:2.75)+(\b:13)+(\c:28)$) -- ($(90+120:2.75)+(0:6)+(\b:13)+(\c:28)$) ($(90:2.75)+(0:6)+(\b:13)+(\c:28)$) -- ($(90+120*2:2.75)+(60:6)+(\b:13)+(\c:28)$);
\foreach \a in {0,60} {
\draw ($(90:2.75)+(\a:6)+(\b:13)+(\c:28)$) --  ($(90+120:2.75)+(\a:6)+(\b:13)+(\c:28)$) -- ($(90+120*2:2.75)+(\a:6)+(\b:13)+(\c:28)$) -- ($(90:2.75)+(\a:6)+(\b:13)+(\c:28)$) ($(90:2.75)+(120*2:2)+(\a:6)+(\b:13)+(\c:28)$) to ++ ($(0:2)$) ($(90+120:2.75)+(0:2)+(\a:6)+(\b:13)+(\c:28)$) to ++ ($(120:2)$)  ($(90+120*2:2.75)+(120:2)+(\a:6)+(\b:13)+(\c:28)$) to ++ ($(120*2:2)$) ;
}
}
\draw ($(90:2.75)+(60:6)+(60:13)$) -- ($(90+120:2.75)+(60:6)+(60:13)+(60:28)$) ($(90+120:2.75)+(0:6)+(0:13)$) -- ($(90+120:2.75)+(0:6)+(0:13)+(0:28)$) ($(90:2.75)+(60:6)+(0:13)+(0:28)$) -- ($(90+120*2:2.75)+(0:6)+(0:13)+(60:28)$);
}
\foreach \d in {0,60} {
\draw ($(90:2.75)+(\d:60)$) --  ($(90+120:2.75)+(\d:60)$) -- ($(90+120*2:2.75)+(\d:60)$) -- ($(90:2.75)+(\d:60)$) ($(90:2.75)+(120*2:2)+(\d:60)$) to ++ ($(0:2)$) ($(90+120:2.75)+(0:2)+(\d:60)$) to ++ ($(120:2)$)  ($(90+120*2:2.75)+(120:2)+(\d:60)$) to ++ ($(120*2:2)$) ;
\draw ($(90:2.75)+(\d:60)$) -- ($(90+120:2.75)+(60:6)+(\d:60)$) ($(90+120*2:2.75)+(\d:60)$) -- ($(90+120:2.75)+(0:6)+(\d:60)$) ($(90:2.75)+(0:6)+(\d:60)$) -- ($(90+120*2:2.75)+(60:6)+(\d:60)$);
\foreach \a in {0,60} {
\draw ($(90:2.75)+(\a:6)+(\d:60)$) --  ($(90+120:2.75)+(\a:6)+(\d:60)$) -- ($(90+120*2:2.75)+(\a:6)+(\d:60)$) -- ($(90:2.75)+(\a:6)+(\d:60)$) ($(90:2.75)+(120*2:2)+(\a:6)+(\d:60)$) to ++ ($(0:2)$) ($(90+120:2.75)+(0:2)+(\a:6)+(\d:60)$) to ++ ($(120:2)$)  ($(90+120*2:2.75)+(120:2)+(\a:6)+(\d:60)$) to ++ ($(120*2:2)$);
}
\draw ($(90:2.75)+(60:6)+(\d:60)$) -- ($(90+120:2.75)+(60:6)+(60:13)+(\d:60)$) ($(90+120*2:2.75)+(0:6)+(\d:60)$) -- ($(90+120:2.75)+(0:6)+(0:13)+(\d:60)$) ($(90:2.75)+(60:6)+(0:13)+(\d:60)$) -- ($(90+120*2:2.75)+(0:6)+(60:13)+(\d:60)$);
\foreach \b in {0,60} {
\draw ($(90:2.75)+(\b:13)+(\d:60)$) --  ($(90+120:2.75)+(\b:13)+(\d:60)$) -- ($(90+120*2:2.75)+(\b:13)+(\d:60)$) -- ($(90:2.75)+(\b:13)+(\d:60)$) ($(90:2.75)+(120*2:2)+(\b:13)+(\d:60)$) to ++ ($(0:2)$) ($(90+120:2.75)+(0:2)+(\b:13)+(\d:60)$) to ++ ($(120:2)$)  ($(90+120*2:2.75)+(120:2)+(\b:13)+(\d:60)$) to ++ ($(120*2:2)$) ;
\draw ($(90:2.75)+(\b:13)+(\d:60)$) -- ($(90+120:2.75)+(60:6)+(\b:13)+(\d:60)$) ($(90+120*2:2.75)+(\b:13)+(\d:60)$) -- ($(90+120:2.75)+(0:6)+(\b:13)+(\d:60)$) ($(90:2.75)+(0:6)+(\b:13)+(\d:60)$) -- ($(90+120*2:2.75)+(60:6)+(\b:13)+(\d:60)$);
\foreach \a in {0,60} {
\draw ($(90:2.75)+(\a:6)+(\b:13)+(\d:60)$) --  ($(90+120:2.75)+(\a:6)+(\b:13)+(\d:60)$) -- ($(90+120*2:2.75)+(\a:6)+(\b:13)+(\d:60)$) -- ($(90:2.75)+(\a:6)+(\b:13)+(\d:60)$) ($(90:2.75)+(120*2:2)+(\a:6)+(\b:13)+(\d:60)$) to ++ ($(0:2)$) ($(90+120:2.75)+(0:2)+(\a:6)+(\b:13)+(\d:60)$) to ++ ($(120:2)$)  ($(90+120*2:2.75)+(120:2)+(\a:6)+(\b:13)+(\d:60)$) to ++ ($(120*2:2)$) ;
}
}
\foreach \c in {0,60} {
\draw ($(90:2.75)+(\c:28)+(\d:60)$) --  ($(90+120:2.75)+(\c:28)+(\d:60)$) -- ($(90+120*2:2.75)+(\c:28)+(\d:60)$) -- ($(90:2.75)+(\c:28)+(\d:60)$) ($(90:2.75)+(120*2:2)+(\c:28)+(\d:60)$) to ++ ($(0:2)$) ($(90+120:2.75)+(0:2)+(\c:28)+(\d:60)$) to ++ ($(120:2)$)  ($(90+120*2:2.75)+(120:2)+(\c:28)+(\d:60)$) to ++ ($(120*2:2)$) ;
\draw ($(90:2.75)+(\c:28)+(\d:60)$) -- ($(90+120:2.75)+(60:6)+(\c:28)+(\d:60)$) ($(90+120*2:2.75)+(\c:28)+(\d:60)$) -- ($(90+120:2.75)+(0:6)+(\c:28)+(\d:60)$) ($(90:2.75)+(0:6)+(\c:28)+(\d:60)$) -- ($(90+120*2:2.75)+(60:6)+(\c:28)+(\d:60)$);
\foreach \a in {0,60} {
\draw ($(90:2.75)+(\a:6)+(\c:28)+(\d:60)$) --  ($(90+120:2.75)+(\a:6)+(\c:28)+(\d:60)$) -- ($(90+120*2:2.75)+(\a:6)+(\c:28)+(\d:60)$) -- ($(90:2.75)+(\a:6)+(\c:28)+(\d:60)$) ($(90:2.75)+(120*2:2)+(\a:6)+(\c:28)+(\d:60)$) to ++ ($(0:2)$) ($(90+120:2.75)+(0:2)+(\a:6)+(\c:28)+(\d:60)$) to ++ ($(120:2)$)  ($(90+120*2:2.75)+(120:2)+(\a:6)+(\c:28)+(\d:60)$) to ++ ($(120*2:2)$);
}
\draw ($(90:2.75)+(60:6)+(\c:28)+(\d:60)$) -- ($(90+120:2.75)+(60:6)+(60:13)+(\c:28)+(\d:60)$) ($(90+120*2:2.75)+(0:6)+(\c:28)+(\d:60)$) -- ($(90+120:2.75)+(0:6)+(0:13)+(\c:28)+(\d:60)$) ($(90:2.75)+(60:6)+(0:13)+(\c:28)+(\d:60)$) -- ($(90+120*2:2.75)+(0:6)+(60:13)+(\c:28)+(\d:60)$);
\foreach \b in {0,60} {
\draw ($(90:2.75)+(\b:13)+(\c:28)+(\d:60)$) --  ($(90+120:2.75)+(\b:13)+(\c:28)+(\d:60)$) -- ($(90+120*2:2.75)+(\b:13)+(\c:28)+(\d:60)$) -- ($(90:2.75)+(\b:13)+(\c:28)+(\d:60)$) ($(90:2.75)+(120*2:2)+(\b:13)+(\c:28)+(\d:60)$) to ++ ($(0:2)$) ($(90+120:2.75)+(0:2)+(\b:13)+(\c:28)+(\d:60)$) to ++ ($(120:2)$)  ($(90+120*2:2.75)+(120:2)+(\b:13)+(\c:28)+(\d:60)$) to ++ ($(120*2:2)$) ;
\draw ($(90:2.75)+(\b:13)+(\c:28)+(\d:60)$) -- ($(90+120:2.75)+(60:6)+(\b:13)+(\c:28)+(\d:60)$) ($(90+120*2:2.75)+(\b:13)+(\c:28)+(\d:60)$) -- ($(90+120:2.75)+(0:6)+(\b:13)+(\c:28)+(\d:60)$) ($(90:2.75)+(0:6)+(\b:13)+(\c:28)+(\d:60)$) -- ($(90+120*2:2.75)+(60:6)+(\b:13)+(\c:28)+(\d:60)$);
\foreach \a in {0,60} {
\draw ($(90:2.75)+(\a:6)+(\b:13)+(\c:28)+(\d:60)$) --  ($(90+120:2.75)+(\a:6)+(\b:13)+(\c:28)+(\d:60)$) -- ($(90+120*2:2.75)+(\a:6)+(\b:13)+(\c:28)+(\d:60)$) -- ($(90:2.75)+(\a:6)+(\b:13)+(\c:28)+(\d:60)$) ($(90:2.75)+(120*2:2)+(\a:6)+(\b:13)+(\c:28)+(\d:60)$) to ++ ($(0:2)$) ($(90+120:2.75)+(0:2)+(\a:6)+(\b:13)+(\c:28)+(\d:60)$) to ++ ($(120:2)$)  ($(90+120*2:2.75)+(120:2)+(\a:6)+(\b:13)+(\c:28)+(\d:60)$) to ++ ($(120*2:2)$) ;
}
}
\draw ($(90:2.75)+(60:6)+(60:13)+(\d:60)$) -- ($(90+120:2.75)+(60:6)+(60:13)+(60:28)+(\d:60)$) ($(90+120:2.75)+(0:6)+(0:13)+(\d:60)$) -- ($(90+120:2.75)+(0:6)+(0:13)+(0:28)+(\d:60)$) ($(90:2.75)+(60:6)+(0:13)+(0:28)+(\d:60)$) -- ($(90+120*2:2.75)+(0:6)+(0:13)+(60:28)+(\d:60)$);
}
\draw ($(90:2.75)+(60:6)+(60:13)+(60:28)$) -- ($(90+120:2.75)+(60:6)+(60:13)+(60:28)+(60:20)$);
\draw ($(90+120:2.75)+(0:6)+(0:13)+(0:30)$) -- ($(90+120:2.75)+(0:6)+(0:13)+(0:28)+(0:20)$);
\draw ($(90:2.75)+(60:6)+(0:13)+(0:28)+(0:60)$) -- ($(90+120*2:2.75)+(0:6)+(0:13)+(0:28)+(60:60)$);
}
\end{tikzpicture}
\caption{The Hanoi attractor or stretched Sierpinski gasket.}
\label{fig:SSG}
\end{figure}

The main contribution of this paper consists in showing that there are in fact many situations where a rather natural resistance form only mirrors part of the fractal upon which it is constructed. We study in detail the case of a finitely ramified fractal $\rm H$ whose intrinsic structure is based on the $3$-level Sierpinski gasket, the interval and the regular inverted Sierpinski gasket. A precise description of this set is given in Section~\ref{section:SG3}. 

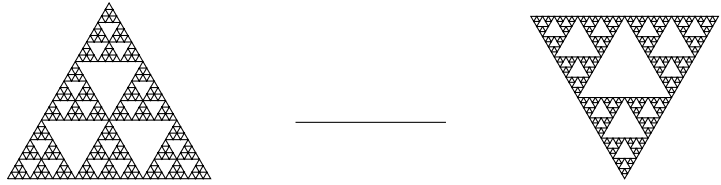
\begin{figure}[H]
\centering
\begin{tabular}{ccc}
\begin{tikzpicture}[scale=.1]
\foreach \d in {0,1}{
\foreach \e in {1,2}{

\foreach \a in {0,1} {
\foreach \b in {1,2} {
\draw ($(60*\d:9*\e)+(60*\a:3*\b)$) to ++ ($(0:3)$) to ++ ($(120:3)$) to ++ ($(240:1)$) to++ ($(0:1)$) to++ ($(240:1)$) to++ ($(0:1)$) to ++ ($(240:1)$) to++ ($(120:2)$) to++ ($(240:1)$) to++ ($(0:1)$) to++  ($(240:1)$) to ++($(120:1)$) to++ ($(240:1)$);
}
}
\foreach \c in {0,1} {
\draw ($(60*\d:9*\e)+(0:3*\c)+(60:3*\c)$) to ++ ($(0:3)$) to ++ ($(120:3)$) to ++ ($(240:1)$) to++ ($(0:1)$) to++ ($(240:1)$) to++ ($(0:1)$) to ++ ($(240:1)$) to++ ($(120:2)$) to++ ($(240:1)$) to++ ($(0:1)$) to++  ($(240:1)$) to ++($(120:1)$) to++ ($(240:1)$);
}
}
}
\foreach \f in {0,1} {
\foreach \a in {0,1} {
\foreach \b in {1,2} {
\draw ($(60*\a:3*\b)+(0:9*\f)+(60:9*\f)$) to ++ ($(0:3)$) to ++ ($(120:3)$) to ++ ($(240:1)$) to++ ($(0:1)$) to++ ($(240:1)$) to++ ($(0:1)$) to ++ ($(240:1)$) to++ ($(120:2)$) to++ ($(240:1)$) to++ ($(0:1)$) to++  ($(240:1)$) to ++($(120:1)$) to++ ($(240:1)$);
}
}
\foreach \c in {0,1} {
\draw ($(0:3*\c)+(60:3*\c)+(0:9*\f)+(60:9*\f)$) to ++ ($(0:3)$) to ++ ($(120:3)$) to ++ ($(240:1)$) to++ ($(0:1)$) to++ ($(240:1)$) to++ ($(0:1)$) to ++ ($(240:1)$) to++ ($(120:2)$) to++ ($(240:1)$) to++ ($(0:1)$) to++  ($(240:1)$) to ++($(120:1)$) to++ ($(240:1)$);
}
}
\end{tikzpicture}
&
\hspace*{.25in}
\begin{tikzpicture}[scale=.5]
\draw ($(0:2)$) --  ($(180:2)$);
\draw[white] ($(270:1.5)$) to++ ($(0:1)$);
\end{tikzpicture}
&
\hspace*{.25in}
\begin{tikzpicture}[scale=.045]
\foreach \c in {0,1,2} {
\foreach \b in {0,1,2} {
\foreach \a in {0,1,2} {
\foreach \e in {0,1,2} {
\foreach \f in {0,1,2} {
\draw[] ($(270+120*\f:16)+(270+120*\e:8) + (9270+120*\c:4)+(270+120*\b:2)+(270+120*\a:1)+(270:1)$) -- ($(270+120*\f:16)+(270+120*\e:8) + (270+120*\c:4)+(270+120*\b:2)+(270+120*\a:1)+(30:1)$) -- ($(270+120*\f:16)+(270+120*\e:8) + (270+120*\c:4)+(270+120*\b:2)+(270+120*\a:1)+(150: 1)$)--cycle; 
}
}
}
}
}
\end{tikzpicture}
\end{tabular}

\caption{\small $3$-level Sierpinski gasket, interval and regular inverted Sierpinski gasket.}
\label{fig:intro01}
\end{figure} 

On $\rm H$ we can construct a resistance form $(\E,\dom\E)$ with a certain self-similarity property, c.f. Lemma~\ref{lemma:SG3.01}, that is essentially a countable sum of energies (resistance forms) on intervals $\I_\alpha$ and inverted Sierpinski gaskets $\SG_\alpha$. More precisely, for any $u\in\dom\E$ we have that
\begin{equation}\label{eq:intro01}
\E(u,u)=\sum_{k=0}^\infty\sum_{\alpha\in\mcA_k}\Big(\sum_{\I\in J}\E_k^{\I_\alpha}(u,u)+\E_k^{\SG_\alpha}(u,u)\Big),
\end{equation}
where $J$ denotes the set of intervals that appear at the first level approximation of $\rm H$. This expression provides a resistance form that effectively overlooks the underlying $3$-level Sierpinski gasket structure of the set. Notice that it will not be enough that the sum~\eqref{eq:intro01} is finite for a function on $\rm H$ in order to be in the domain of the energy $\E$. Functions in the domain of $\E$ also need to be continuous. This continuity assumption is enforced by our definition of the domain, see Definition~\ref{def:SG3.03} and Remark~\ref{rem:SG3.04}.

\medskip

The same construction extends to a class of finitely ramified fractals that we introduce in Section~\ref{section:HD3S} and more generally in Section~\ref{section:HF}, which we call \textit{p.c.f.\ hybrids}. From the geometric point of view, these sets can be regarded as graph-directed fractals~\cite{MW88,HN03}. They need not be strictly self-similar but generated through a non-trivial mechanism that combines several post-critically finite (p.c.f.) self-similar pieces. Again, we can find resistance forms that do not reflect the whole intrinsic fractal nature of the set on which they are defined. 

\medskip

Besides, p.c.f.\ hybrids can be equipped with a \textit{finitely ramified cell structure}, a concept described in~\cite{Tep08} that appears in numerous applications, specially in the fractal setting~\cite{Str01,RT10,IRT12}. In Section~\ref{section:FRCS} we investigate the question about the characterization of the domain of general resistance forms within this abstract framework, setting aside matters of existence and uniqueness that are known to be rather delicate; see e.g.~\cite{HMT06} and references therein. For finitely ramified cell structures that support a fractal dust, Theorem~\ref{thm:MR02} provides necessary and sufficient conditions under which a given resistance form fails to ``detect'' the underlying fractal dust.

\medskip

From the theory of resistance forms~\cite{Kig01}, we know that a resistance form on a hybrid fractal equipped with a Borel regular measure gives rise to a Dirichlet form and hence to a Laplacian. Finding out properties of the spectrum of the Laplacian in the fractal setting is a problem with long history in the literature, see e.g.~\cite{FS92,KL93,BST99,ASST03,BKS13,SZ17} and the second part of this paper is devoted to the investigation of the spectrum of the Laplacian on the Hanoi attractor and on the hybrid $\rm H$ previously mentioned and studied in detail in Section~\ref{section:SG3}.

\medskip

To analyze the spectrum of the Laplacian on the Hanoi attractor we present two different approaches: the first one relies on approximation by discrete graphs and their associated discrete Laplacians. Numerical computations of the spectrum, the eigenvalue counting function and simulations of eigenfunctions are provided. The second approach is based in the approximation by quantum graphs introduced in~\cite{ARKT16}. We compare the numerical evidence obtained from both methods and refer to the reader to the website~\cite{CGSZ17} for more data. An analogous study is presented for the hybrid $\rm H$ based on the $3$-level Sierpinski gasket for which further data is also available at~\cite{CGSZ17}. In this case we also obtain the asymptotic behavior of the corresponding eigenvalue counting function.

\medskip

The outline of the paper is the following. Taking the hybrid fractal $\rm H$ as an example, we introduce in Section~\ref{section:SG3} the concept of p.c.f.\ hybrid fractals and work out the details that lead in Theorem~\ref{thm:SG3.02} to a characterizable resistance form on it. These ideas are further developed in Section~\ref{section:HD3S} for the case of self-similar energies on hybrids with an underlying dihedral-$3$ symmetry. Section~\ref{section:FRCS} treats the characterization problem in the abstract setting of finitely ramified cell structures, removing the self-similarity property of the resistance forms in consideration. A general definition of p.c.f.\ hybrid fractals along with an associated finitely ramified cell structure is established in Section~\ref{section:HF}, where Theorem~\ref{thm:MR02} is applied to the case of the Hanoi attractor. Section~\ref{section:SpectrumHanoi} provides a numerical study of the spectrum of the Laplacians induced by the self-similar energy on the Hanoi attractor, while Section~\ref{sec:SpcSG3} discusses the spectral properties of the corresponding Laplacian on the hybrid $\rm H$ based on the $3$-level Sierpinski gasket.
\subsection*{Acknowledgments} The first author thanks A.\ Teplyaev for many fruitful discussions and the Feodor Lynen fellowship from the Humboldt Foundation for partial support. We are also very thankful to the University of Hong Kong for the support provided to Y.\ Chen and H.\ Gu.
\section{Characterizable energy on a hybrid fractal}\label{section:SG3}
This section presents a motivating example of a hybrid fractal where we can establish the existence of an energy (resistance form) with the particular property that it ``captures'' only part of the intrinsic structure of the fractal. 

\medskip

The hybrid fractal that we discuss here is based on $3$-level Sierpinski gasket, a self-similar set introduced in~\cite{Str00} that we denote by $\SG_3$, see Figure~\ref{fig:SG3.01}. 

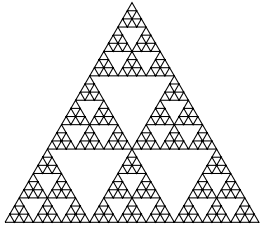
\begin{figure}[H]
\centering
\begin{tikzpicture}[scale=.125]
\foreach \d in {0,1}{
\foreach \e in {1,2}{

\foreach \a in {0,1} {
\foreach \b in {1,2} {
\draw ($(60*\d:9*\e)+(60*\a:3*\b)$) to ++ ($(0:3)$) to ++ ($(120:3)$) to ++ ($(240:1)$) to++ ($(0:1)$) to++ ($(240:1)$) to++ ($(0:1)$) to ++ ($(240:1)$) to++ ($(120:2)$) to++ ($(240:1)$) to++ ($(0:1)$) to++  ($(240:1)$) to ++($(120:1)$) to++ ($(240:1)$);
}
}
\foreach \c in {0,1} {
\draw ($(60*\d:9*\e)+(0:3*\c)+(60:3*\c)$) to ++ ($(0:3)$) to ++ ($(120:3)$) to ++ ($(240:1)$) to++ ($(0:1)$) to++ ($(240:1)$) to++ ($(0:1)$) to ++ ($(240:1)$) to++ ($(120:2)$) to++ ($(240:1)$) to++ ($(0:1)$) to++  ($(240:1)$) to ++($(120:1)$) to++ ($(240:1)$);
}
}
}
\foreach \f in {0,1} {
\foreach \a in {0,1} {
\foreach \b in {1,2} {
\draw ($(60*\a:3*\b)+(0:9*\f)+(60:9*\f)$) to ++ ($(0:3)$) to ++ ($(120:3)$) to ++ ($(240:1)$) to++ ($(0:1)$) to++ ($(240:1)$) to++ ($(0:1)$) to ++ ($(240:1)$) to++ ($(120:2)$) to++ ($(240:1)$) to++ ($(0:1)$) to++  ($(240:1)$) to ++($(120:1)$) to++ ($(240:1)$);
}
}
\foreach \c in {0,1} {
\draw ($(0:3*\c)+(60:3*\c)+(0:9*\f)+(60:9*\f)$) to ++ ($(0:3)$) to ++ ($(120:3)$) to ++ ($(240:1)$) to++ ($(0:1)$) to++ ($(240:1)$) to++ ($(0:1)$) to ++ ($(240:1)$) to++ ($(120:2)$) to++ ($(240:1)$) to++ ($(0:1)$) to++  ($(240:1)$) to ++($(120:1)$) to++ ($(240:1)$);
}
}
\end{tikzpicture}
\caption{The $3$-level Sierpinski gasket.}
\label{fig:SG3.01}
\end{figure} 

To construct the hybrid we proceed as follows: starting with the first level approximation of $\SG_3$, substitute each junction point by a line segment or an inverted Sierpinski gasket ($\SG$) according to Figure~\ref{fig:SG3.02}. At each new level, existing line segments and inverted $\SG$s remain, while triangular cells subdivide as before: each new junction point of the first level approximation of $\SG_3$ is substituted by (a smaller copy of) an interval or an inverted $\SG$. In general, we will loosely speak of a triangular $n$-cell as a triangular cell that appears in the $n$th construction level .
\begin{figure}[H]
\centering
\begin{tikzpicture}[scale=.8]
\draw ($(270:2.375)$) node[] {\includegraphics[angle=180,origin=c,scale=0.015]{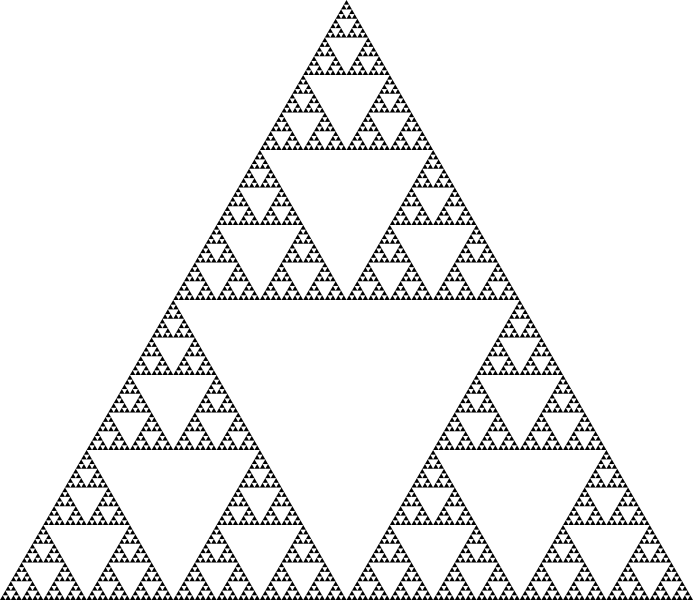}};
\filldraw ($(240:4)$) to ++ ($(0:1)$) to ++ ($(120:1)$) to ++ ($(240:1)$);
\filldraw ($(240:4)+(0:1)$) to ++ ($(0:.5)$) to ++ ($(0:1)$) to ++ ($(120:1)$) to ++ ($(240:1)$);
\filldraw ($(240:4)+(0:2.5)$) to ++ ($(0:.5)$) to ++ ($(0:1)$) to ++ ($(120:1)$) to ++ ($(240:1)$);
\filldraw ($(240:4)+(60:1.5)$) to ++ ($(0:1)$) to ++ ($(120:1)$) to ++ ($(240:1.5)$);
\filldraw ($(240:4)+(60:3)$) to ++ ($(0:1)$) to ++ ($(120:1)$)  to ++ ($(240:1.5)$);
\filldraw ($(240:4)+(60:1.5)+(0:1.5)$) to ++ ($(0:1)$) to ++ ($(120:1)$)  to ++ ($(240:1)$);
\filldraw ($(300:3)$) to ++ ($(120:.5)$) ($(300:1.5)$) to ++ ($(120:.5)$);
\filldraw ($(240:2.5)$) node[left] {\small $\I$} ($(240:1)$) node[left] {\small $\I$} ($(300:2.5)$) node[right] {\small $\I$} ($(300:1)$) node[right] {\small $\I$} ($(258:3.65)$) node[above] {\small $\I$} ($(283:3.65)$) node[above] {\small $\I$} ($(270:2.25)$) node[above] {$\scriptscriptstyle{\SG}$};
\end{tikzpicture}
\caption{\small Line segments and one inverted $\SG$ join the six triangular $1$-cells of the first level approximation of $\SG_3$.}
\label{fig:SG3.02}
\end{figure}
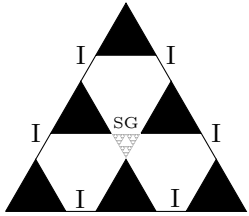 
\begin{notation}
Let $J$ be the set of all line segments and inverted Sierpinski gaskets ``born'' in the first generation.
\end{notation}

\begin{definition}\label{def:SG3.01}
Let $\{\phi_i\colon\mbbR^2\to\mbbR^2\}_{i=1}^6$ be the $s$-similitudes, $s\in (0,1)$, that map the unit equilateral triangle to the respective copies in Figure~\ref{fig:SG3.02}.
\begin{enumerate}[leftmargin=.25in]
\item[(i)] The hybrid fractal $\rm H\subseteq\mbbR^2$ is the unique non-empty compact set that satisfies
\begin{equation*}\label{eq:SG3.001}
{\rm H}=\bigcup_{i=1}^6\phi_i({\rm H})\cup\bigcup_{K\in J}K.
\end{equation*}
\item[(ii)]  The the \textit{fractal dust} associated with $\rm H$ is the self-similar set that satisfies
\begin{equation*}\label{eq:SG3.01}
\mcC_{\rm H}=\bigcup_{i=1}^6\phi_i(\mcC_{\rm H}).
\end{equation*} 
\item[(iii)] For each word of length $n\geq 1$, $\alpha=\alpha_1\cdots\alpha_n\in\mcA_n:=\{1,\ldots, 6\}^n$ define
\begin{equation*}\label{eq:SG3.phi_alpha}
\phi_\alpha=\phi_{\alpha_1}\circ\cdots\circ\phi_{\alpha_n}
\end{equation*}
and $\phi_{\text{\o}}={\rm id}$. A set $\phi_\alpha({\rm H})$ is called an $n$-cell of ${\rm H}$.
\end{enumerate}
\end{definition}

For existence and uniqueness of these sets we refer to~\cite{Hat85}.

\medskip

The hybrid $\rm H$ can also be described as a graph-directed fractal, see~\cite{MW88,HN03}. It arises by recursively replacing triangles and lines with the corresponding combinations depicted in Figure~\ref{fig:SG3.03}. In this case, the underlying division of the base generator, the upright triangle, relies on the transformation of Figure~\ref{fig:SG3.02}. The other fractals constituting the hybrid are the interval  and the (inverted) $\SG$. 

\begin{figure}[H]
\centering
\begin{tabular}{ccc}
\begin{tikzpicture}[scale=.5]
\filldraw ($(240:4)$) 
--  ($(240+60:4)$) 
 -- ($(240:4)+(60:4)$) 
-- cycle;
\end{tikzpicture}
\hspace*{.5in}
&
\begin{tikzpicture}[scale=.5]
\filldraw ($(60:4)$) 
--  ($(120:4)$) 
 -- ($(120:4)+(300:4)$) 
-- cycle;
\end{tikzpicture}
\hspace*{.5in}
&
\begin{tikzpicture}[scale=.5]
\draw[very thick] ($(0:2)$) 
--  ($(180:2)$) 
;
\draw[white] ($(270:1.5)$) to++ ($(0:1)$);
\end{tikzpicture}
\\
& &\\
$\rm H\downarrow$\hspace*{.7in} &$\SG\downarrow$\hspace*{.8in} &$\I\downarrow$\hspace*{.15in} \\
& \hspace*{.55in}&\hspace*{.55in}\\
\begin{tikzpicture}[scale=.5]
\filldraw ($(240:4)$) to ++ ($(0:1)$) to ++ ($(120:1)$);
\filldraw ($(240:4)+(0:1.5)$) to ++ ($(0:1)$) to ++ ($(120:1)$);
\filldraw ($(240:4)+(0:3)$) to ++ ($(0:1)$) to ++ ($(120:1)$);
\filldraw ($(240:4)+(60:1.5)$) to ++ ($(0:1)$) to ++ ($(120:1)$);
\filldraw ($(240:4)+(60:3)$) to ++ ($(0:1)$) to ++ ($(120:1)$);
\filldraw ($(240:4)+(60:1.5)+(0:1.5)$) to ++ ($(0:1)$) to ++ ($(120:1)$);
\filldraw ($(240:4)+(60:1.5)+(0:1)$) to ++ ($(0:.5)$) to ++ ($(240:.5)$);
\draw[thick] ($(240:4)$) to ++ ($(0:4)$) to ++ ($(120:4)$) -- cycle;
\end{tikzpicture}
\hspace*{.5in}
&
\begin{tikzpicture}[scale=.5]
\filldraw ($(60:4)$) to ++ ($(180:2)$) to ++ ($(300:2)$)-- cycle;
\filldraw ($(120:4)$) to ++ ($(0:2)$) to ++ ($(240:2)$)-- cycle;
\filldraw ($(0:0)$) to ++ ($(60:2)$) to ++ ($(180:2)$)-- cycle;
\end{tikzpicture}
\hspace*{.5in}
&
\begin{tikzpicture}[scale=.5]
\draw[very thick] ($(0:2)$) 
--  ($(180:2)$) 
;
\draw[white] ($(270:1.5)$) to++ ($(0:1)$);
\end{tikzpicture}
\end{tabular}
\caption{\small Generators of the hybrid fractal $\rm H$ based on the the upright triangle.}
\label{fig:SG3.03}
\end{figure}
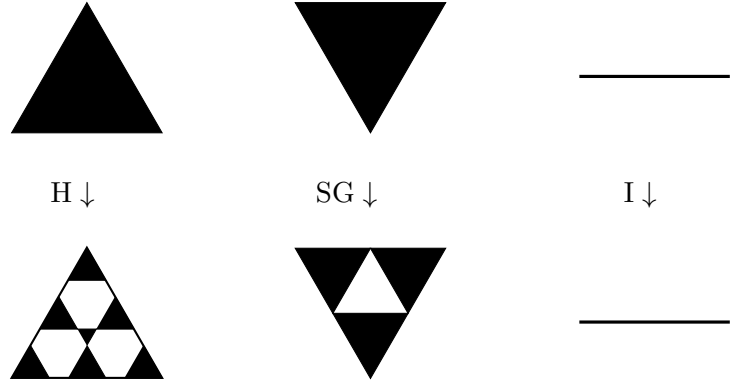
\begin{remark}\label{rem:SG3.01}
The hybrid fractal $\rm H$ is \textit{finitely ramified} because any subset of it can be disconnected by removing finitely many points. Its boundary $V_0$ consists of the three vertices of the base upright triangle.
\end{remark}
\subsection{Energy from approximating graphs}\label{section:SG3.1}
An energy (resistance) form on $\rm H$ can be obtained as the limit of a sequence of resistance forms associated with suitable weighted graphs $\{\Gamma_n\}_{n\geq 0}$ that approximate the hybrid. 

\medskip

The construction scheme for the hybrid yields a recursive procedure to define the graphs $\Gamma_n$: at each new level $n$, edges that built an upright triangle in level $n-1$ subdivides as $\rm H$ and edges that joined upright triangles subdivide either into two, or as an inverted Sierpinski gasket, see Figure~\ref{fig:SG3.04}. In addition, each type of subdivision carries specific \textit{resistance scaling factors}. The \textit{resistance} (i.e. the weight) of a $n$th-level edge is obtained by multiplying the resistance of the $(n-1)$th-level edge it comes from by the resistance scaling associated to the $n$th-level subdivision it performs. 

\medskip

For each $n\geq 0$, the graph $\Gamma_n=(V_n,E_n,r_n)$ induces a resistance form given by
\begin{equation}\label{eq:SG3.02}
\sum_{\{x,y\}\in E_n}\frac{1}{r_n(x,y)}(u(x)-u(y))(v(x)-v(y))
\end{equation}
for any $u,v\in\ell(V_n):=\{u\colon V_n\to\mbbR\}$. In this particular example, we will choose the same resistance scaling factors for every level and give each subdivision type the ones described in Figure~\ref{fig:SG3.03}. Based on~\cite[Lemma 5.1]{HN03}, we will call the resulting resistance form a \textit{graph directed self-similar} resistance form.

\begin{definition}\label{def:SG3.02}
Let $R,\rH,\rI,\rSG>0$. For each $n\geq 0$, define the weight function of the graph $\Gamma_n$ by 
\begin{equation}\label{eq:SG3.03}
r_n(x,y)=\begin{cases}
R\rH^n&\text{if }\{x,y\}\text{ build a triangle},\\
R\rH^{k-1}\rI\big(\frac{1}{2}\big)^{n-k}&\text{if }\{x,y\} \text{ build a segment ``born'' at level }k\leq n ,\\
R\rH^{k-1}\rSG\big(\frac{3}{5}\big)^{n-k}&\text{if }\{x,y\} \text{ build an inverted }\SG\text{ ``born'' at level }k\leq n,
\end{cases}
\end{equation}
and define $(\E_n,\ell(V_n))$ to be the associated resistance form given by~\eqref{eq:SG3.02}.
\end{definition}
\begin{figure}[h!tpb]
\centering
\begin{tabular}{ccc}
\begin{tikzpicture}[scale=.65]
\draw[thick] ($(240:4)$) circle (2pt) --  ($(240+60:4)$) circle (2pt) -- ($(240:4)+(60:4)$) circle (2pt) -- ($(240:4)$);
\end{tikzpicture}
\hspace*{.5in}
&
\begin{tikzpicture}[scale=.65]
\draw[thick] ($(60:4)$) circle (2pt) --  ($(120:4)$) circle (2pt) -- ($(120:4)+(300:4)$) circle (2pt) -- ($(60:4)$);
\end{tikzpicture}
\hspace*{.5in}
&
\begin{tikzpicture}[scale=.65]
\draw[thick] ($(0:2)$) circle (2pt) --  ($(180:2)$) circle (2pt);
\draw[white] ($(270:1.5)$) to++ ($(0:1)$);
\end{tikzpicture}
\\
\hspace*{.5in}& \hspace*{.5in}&\\
$\downarrow$\hspace*{.5in} &$\downarrow$ \hspace*{.5in}&$\downarrow$ \\
& &\\
\begin{tikzpicture}[scale=.65]
\draw[thick] ($(240:4)$) circle (2pt) to ++ ($(0:1)$) circle (2pt) to ++ ($(120:1)$) circle (2pt) to ++ ($(240:1)$);
\draw[thick] ($(240:4)+(0:1.5)$) circle (2pt) to ++ ($(0:1)$) circle (2pt) to ++ ($(120:1)$) circle (2pt) to ++ ($(240:1)$);
\draw[thick] ($(240:4)+(0:3)$)  circle (2pt) to ++ ($(0:1)$) circle (2pt) to ++ ($(120:1)$) circle (2pt) to ++ ($(240:1)$);
\draw[thick] ($(240:4)+(60:1.5)$) circle (2pt) to ++ ($(0:1)$) circle (2pt) to ++ ($(120:1)$) circle (2pt) to ++ ($(240:1)$);
\draw[thick] ($(240:4)+(60:3)$) circle (2pt) to ++ ($(0:1)$) circle (2pt) to ++ ($(120:1)$) circle (2pt) to ++ ($(240:1)$);
\draw[thick] ($(240:4)+(60:1.5)+(0:1.5)$) circle (2pt) to ++ ($(0:1)$) circle (2pt) to ++ ($(120:1)$) circle (2pt) to ++ ($(240:1)$);
\draw[thick] ($(240:4)+(60:1.5)+(0:1)$) circle (2pt) to ++ ($(0:.5)$) circle (2pt) to ++ ($(240:.5)$) circle (2pt) to ++ ($(120:.5)$) ;
\draw[thick] ($(240:4)$) to ++ ($(0:4)$) to ++ ($(120:4)$) -- cycle;
\draw  ($(240:3)$) node[below left] {\small $\rH\;$}
 ($(240:2.25)$) node[below left] {\small $\rI\;\;$}
 ($(240:1.5)$) node[below left] {\small $\rH\;$}
 ($(240:.7)$) node[below left] {\small $\rI\;\;$}
 ($(240:.05)$) node[below left] {\small $\rH\;$}
 ($(240:4)$) node[below right] {\small $\rH\;$}
 ($(260:3.5)$) node[below] {\small $\rI\;\;$}
 ($(280:3.5)$) node[below left] {\small $\rH\;$}
 ($(285:3.55)$) node[below] {\small $\rI\;\;$}
 ($(295:3.75)$) node[below] {\small $\rH\;$}
 ($(270:2.25)$) node[below right] {$\scriptstyle \rSG\;$}
 ($(270:2.25)$) node[above] {$\scriptstyle\rSG\;$}
 ($(270:2.25)$) node[below left] {$\scriptstyle\rSG$};
\end{tikzpicture}
\hspace*{.5in}
&
\begin{tikzpicture}[scale=.65]
\draw[thick] ($(60:4)$) circle (2pt) to ++ ($(180:2)$) circle (2pt) to ++ ($(300:2)$) to ++ ($(60:2)$);
\draw[thick] ($(120:4)$) circle (2pt) to ++ ($(0:2)$) circle (2pt) to ++ ($(240:2)$) circle (2pt) to ++ ($(120:2)$);
\draw[thick] ($(0:0)$) circle (2pt) to ++ ($(60:2)$) circle (2pt) to ++ ($(180:2)$) circle (2pt)to ++ ($(300:2)$);
\draw ($(120:2.75)$) node[left] {\small $\frac{3}{5}$}
($(120:.75)$) node[left] {\small $\frac{3}{5}$}
($(110:3.5)$) node[above] {\small $\frac{3}{5}$}
($(70:3.5)$) node[above] {\small $\frac{3}{5}$}
($(60:2.75)$) node[right] {\small $\frac{3}{5}$}
($(60:.75)$) node[right] {\small $\frac{3}{5}$};
\end{tikzpicture}
\hspace*{.5in}
&
\begin{tikzpicture}[scale=.65]
\draw[thick] ($(180:2)$) circle (2pt) to ++ ($(0:2)$) circle (2pt);
\draw[thick] ($(180:2)+(0:2)$) to ++ ($(0:2)$) circle (2pt);
\draw[white] ($(270:1.5)$) to++ ($(0:1)$);
\draw ($(180:1)$) node[above] {\small $\frac{1}{2}$};
\draw ($(0:1)$) node[above] {\small $\frac{1}{2}$};
\end{tikzpicture}
\end{tabular}
\caption{\small Graph subdivision and corresponding resistance scaling factors.}
\label{fig:SG3.04}
\end{figure}
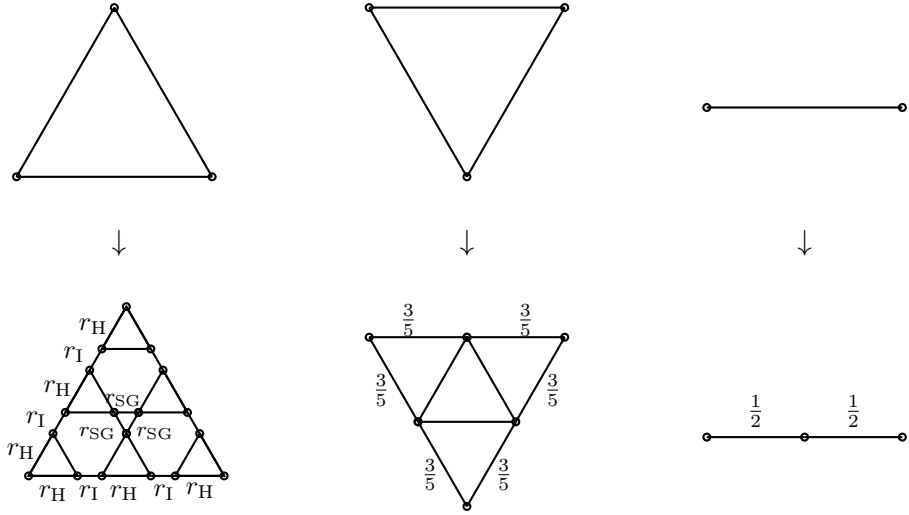
It is a well-known result from the theory of resistance forms, see e.g.~\cite[Theorem 2.2.6]{Kig01}, that the limit
\begin{equation}\label{eq:SG3.04}
\E(u,v)=\lim_{n\to\infty}\E_n(u_{|_{V_n}},v_{|_{V_n}})
\end{equation}
leads to a meaningful resistance form if and only if at each level the resistances $r_n$ satisfy a certain \textit{compatibility condition}. Roughly speaking, this condition says that if we think of the graphs in Figure~\ref{fig:SG3.04} as electric networks with resistors instead of edges, at every level $n$ each subnetwork of $\Gamma_n$ of the type displayed in the first row has to be electrically equivalent to the network below it. 
%
Due to our ``self-similar'' choice of the resistances, it suffices that the networks in Figure~\ref{fig:SG3.05} are equivalent. 
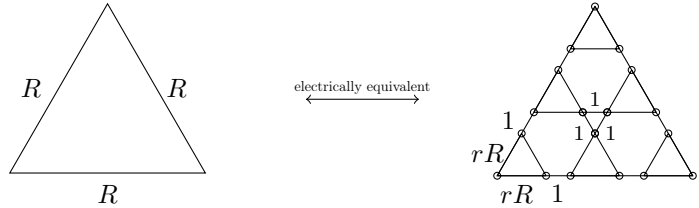
\begin{figure}[H]
\centering
\begin{tabular}{ccc}
\begin{tikzpicture}[scale=.65]
\draw ($(240:4)$) 
--  ($(240+60:4)$) 
 -- ($(240:4)+(60:4)$) 
-- cycle;
\draw ($(240:2)$) node[left] {\small $R\;$} ($(300:2)$) node[right] {\small $R$} ($(270:3.5)$) node[below] {\small $R$};
\end{tikzpicture}
\hspace*{.25in}&
\begin{tikzpicture}[scale=.5]
\draw[<->] ($(180:1.5)$) to++  ($(0:3)$) node[above, near start] {electrically equivalent};
\draw[white] ($(270:3)$) to++ ($(0:1)$);
\end{tikzpicture}
&
\begin{tikzpicture}[scale=.65]
\draw ($(240:4)$) circle (2pt) to ++ ($(0:1)$) circle (2pt) to ++ ($(120:1)$) circle (2pt) to ++ ($(240:1)$);
\draw ($(240:4)+(0:1.5)$) circle (2pt) to ++ ($(0:1)$) circle (2pt) to ++ ($(120:1)$) circle (2pt) to ++ ($(240:1)$);
\draw ($(240:4)+(0:3)$)  circle (2pt) to ++ ($(0:1)$) circle (2pt) to ++ ($(120:1)$) circle (2pt) to ++ ($(240:1)$);
\draw ($(240:4)+(60:1.5)$) circle (2pt) to ++ ($(0:1)$) circle (2pt) to ++ ($(120:1)$) circle (2pt) to ++ ($(240:1)$);
\draw ($(240:4)+(60:3)$) circle (2pt) to ++ ($(0:1)$) circle (2pt) to ++ ($(120:1)$) circle (2pt) to ++ ($(240:1)$);
\draw ($(240:4)+(60:1.5)+(0:1.5)$) circle (2pt) to ++ ($(0:1)$) circle (2pt) to ++ ($(120:1)$) circle (2pt) to ++ ($(240:1)$);
\draw ($(240:4)+(60:1.5)+(0:1)$) circle (2pt) to ++ ($(0:.5)$) circle (2pt) to ++ ($(240:.5)$) circle (2pt) to ++ ($(120:.5)$) ;
\draw ($(240:4)$) to ++ ($(0:4)$) to ++ ($(120:4)$) -- cycle;
\draw  ($(240:3)$) node[below left] {\small $rR\;$}
 ($(240:2.25)$) node[below left] {\small $1\;\;$}
 ($(245:3.8)$) node[below] {\small $rR$}
 ($(260:3.5)$) node[below] {\small $1\;\;$}
 ($(270:2.25)$) node[below right] {$\scriptstyle 1$}
 ($(270:2.25)$) node[above] {$\scriptstyle 1$}
 ($(270:2.25)$) node[below left] {$\scriptstyle 1$};
\end{tikzpicture}
\end{tabular}
\caption{\small As networks, each triangular cell must be equivalent to its $(n+1)$th-level subdivision.}
\label{fig:SG3.05}
\end{figure}
\begin{theorem}\label{thm:SG3.01}
For each $n\geq 0$, let $(\E_n,\ell(V_n))$ be the resistance form given by~\eqref{eq:SG3.03} with $\rI=\rSG=1$, and $\rH=r$. 
\begin{itemize}[leftmargin=.3in]
\item[(i)~] The sequence $\{(\E_n,\ell(V_n))\}_{n\geq 0}$ is compatible if and only if
\begin{equation}\label{eq:SG3.05}
R=\frac{30}{9-31r-\sqrt{81+r(61r-138)}}\qquad\text{where}\qquad 0<r<\frac{7}{15}.
\end{equation}
\item[(ii)] The latter sequence yields a unique resistance form $(\E,\dom\E)$ given by~\eqref{eq:SG3.04}, where $\dom\E=\{u\colon V_*\to\mbbR~|~\E(u,u)<\infty\}$ and $V_*=\cup_{n\geq 0}V_n$.
\item[(iii)] Any function $u\in\dom\E$ is H\"older continuous with respect to the Euclidean metric.
\item[(iv)] The effective resistance metric associated with $(\E,\dom\E)$ induces the same topology as the Euclidean metric.
\item[(v)] The resistance form $(\E,\dom\E)$ extends to a local and regular resistance form on $\rm H$.
\end{itemize}
\end{theorem}
For simplicity, we also denote by $(\E,\dom\E)$ the extended resistance form on $\rm H$.
\begin{proof}
Let us prove (i). Then, statements (ii) through (v) will follow from standard results of resistance forms, see e.g.~\cite[Section 3]{Kig01}. The sequence $\{(\E_n,\ell(V_n))\}_{n\geq 0}$ is compatible if and only if the two networks in Figure~\ref{fig:SG3.05} are equivalent. On the one hand, applying the $\Delta$-Y transform, see e.g.~\cite[Lemma 2.1.15]{Kig01} or~\cite[Lemma 1.5.2]{Str01}, the resistance at level one between two vertices of a $1$-cell $x,y\in V_0$ is given by
\begin{equation*}\label{eq:SG3.01.01}
r_{1}(x,y)=rR+\frac{\big(\frac{2}{3}rR+1\big)(4rR+5)}{\frac{7}{3}rR+3}.
\end{equation*}
On the other hand, the equivalence of the networks means
\begin{equation*}\label{eq:SG3.01.02}
r_{1}(x,y)=r_0(x,y)=R.
\end{equation*}
After some algebraic manipulations, the only positive solution of this equation is~\eqref{eq:SG3.05}.
\end{proof}

\begin{remark}\label{rem:SG3.02}
Notice that $R$ diverges as $r$ approaches the critical value $\frac{7}{15}$, which is precisely the resistance scaling factor of $\SG_3$, the base fractal of $\rm H$. The fact that $r$ should remain below this critical value can be explained by observing that if $r=\frac{7}{15}$, then the network on the left hand side of Figure~\ref{fig:SG3.05} would be electrically equivalent to the first-level approximation of the $\SG_3$. This would require $\rI=\rSG=0$, which we ruled out by assuming them (equal, but most importantly) strictly positive. 
\end{remark}

\begin{remark}\label{rem:SG3.03}
In terms of the \textit{effective resistance} metric associated with the energy $(\E,\dom\E)$ from~\eqref{eq:SG3.03},
\begin{equation*}\label{eq:SG3.06}
R_{\E}(x,y)=\sup\Big\{\frac{|u(x)-u(y)|^2}{\E(u,u)}~|~\E(u,u)\neq 0\Big\}.
\end{equation*}
Applying the $\Delta$-Y transform to the triangular network on the right hand side of Figure~\ref{fig:SG3.05}, we see that the effective resistance between any two points $x,y\in V_0$ is given by $\frac{2}{3}R$. Thus, a scaling factor $r$ that nears $\frac{7}{15}$ makes these two points \textit{effectively} further apart. When the scaling factor actually equals the critical value, the network effectively falls apart because the points become infinitely far away from each other. In other words, the network ends up being a set of three isolated nodes. 
\end{remark}

The general theory of resistance forms provides the existence of the energy $(\E,\dom\E)$ as the limit in~\eqref{eq:SG3.03}, where $\dom\E=\{u\colon{\rm H}\to\mbbR~|~\E(u,u)<\infty\}$. 
In order to obtain more information about the analytic structure of the form it is desirable to characterize the domain more precisely.

\subsection{Characterizable energy}
Adapting several arguments from~\cite{ARKT16} it is possible to construct an energy $(\widetilde{\E},\dom\widetilde{\E})$ on $\rm H$ with a more explicit expression. Namely, this energy will consist of a countable sum of energies on line segments and Sierpinski gaskets, hence not ``capturing'' the structure of the base fractal $\SG_3$. In the following we outline the main ideas of this procedure.

\medskip

Although $\rm H$ is not strictly speaking a \textit{fractal quantum graph} as defined in~\cite[Definition 8.1]{ARKT16}, this concept may be seen as a predecessor of hybrid fractals. Many ideas in~\cite[Section 8]{ARKT16} will thus extend to hybrids by considering the standard energy forms on intervals and (inverted) Sierpinski gaskets instead of only integrals.

\begin{definition}\label{def:SG3.03}
Let $(\E,\dom\E)$ be the resistance form from Theorem~\ref{thm:SG3.01}. For each $n\geq 0$, define
\begin{equation*}\label{eq:SG3.07}
\widetilde{\F}_n
=\{u\in\dom\E~|~u_{|_{\triangle}}\text{ is constant on every }n\text{-cell }\triangle\}.
\end{equation*}  
Furthermore, define the bilinear form $\widetilde{\E}_n=\E_{|_{\widetilde{\F}_n\times \widetilde{\F}_n}}$ and 
\begin{equation*}\label{eq:SG3.09}
\widetilde{R}_n(x,y)=\sup\bigg\{\frac{|u(x)-u(y)|^2}{\widetilde{\E}_n(u,u)}~|~u\in \widetilde{\F}_n,~\widetilde{\E}_n(u,u)\neq 0\bigg\}.
\end{equation*}
\end{definition}
The form $\widetilde{\E}_n$ can be understood as the energy associated with the graphs $\widetilde{\Gamma}_n=(V_n,E_n,\tilde{r}_n)$, constructed only through division of upright triangular cells as the first column of Figure~\ref{fig:SG3.02} and weight function given by
\begin{equation*}\label{eq:SG3.10}
\tilde{r}_n(x,y)=\begin{cases}
0&\text{if }\{x,y\}\text{ build a triangle},\\
R_0r^{k-1}&\text{if }\{x,y\} \text{ build a segment or inverted triangle ``born'' at level }k\leq n.
\end{cases}
\end{equation*}
Following~\cite[Theorem 5.2]{ARKT16}, the sequence of pseudo-metrics $\{\widetilde{R}_n\}_{n\geq 0}$ converges to a metric $\widetilde{R}$ on the hybrid $\rm H$ and for any $x,y\in {\rm H}$
\begin{equation*}\label{eq:SG3.11}
\widetilde{R}(x,y)\leq R_{\E}(x,y).
\end{equation*}
Moreover, the metric $\widetilde{R}$ induces the same topology as the Euclidean metric on $\rm H$. In order to describe its associated energy $(\widetilde{\E},\dom\widetilde{\E})$, let 
$(\E^{\I},\dom\E^{\I})$ and $(\E^{\SG},\dom\E^{\SG})$ denote the standard resistance forms on the interval and the (inverted) Sierpinski gasket, see e.g.~\cite[Chapter 2]{Str01}, and define for each $k\geq 0$, $\alpha\in\mcA_k$ and $u\in\dom\E$,
\begin{equation*}\label{eq:SG3.12}
\E_k^{\I_{\alpha}}(u,u)=\frac{1}{r^{k}}\E^{\I}((u\circ\phi_\alpha){}_{|_{\I}},(u\circ\phi_\alpha){}_{|_{\I}})
\end{equation*}
as well as
\begin{equation*}\label{eq:SG3.13}
\E_k^{\SG_\alpha}(u,u)=\frac{1}{r^{k}}\E^{\SG}((u\circ\phi_\alpha){}_{|_{\SG}},(u\circ\phi_\alpha){}_{|_{\SG}}).
\end{equation*}
Proposition~\ref{prop:MT01} and Theorem~\ref{thm:MT01} yield the existence of an energy on $\rm H$ that can be expressed as a countable sum of energies, ``ignoring'' the underlying $\SG_3$ structure of the base generator of $\rm H$.
\begin{proposition}\label{prop:SG3.01}
The pair $(\widetilde{\E},\dom\widetilde{\E})$ given by
\begin{equation}\label{eq:SG3.14}
\widetilde{\E}(u,v)=\sum_{k=0}^\infty\sum_{\alpha\in\mcA_k}\Big(\sum_{\I\in J}\E_k^{\I_\alpha}(u,v)+\E_k^{\SG_\alpha}(u,v)\Big)
\end{equation}
and 
\begin{multline*}\label{eq:SG3.15}
\dom\widetilde{\E}=\{u\in\dom\E~|~\text{there exists}~\{u_n\}_{n\geq 1}\subseteq\cup_{n\geq 1}\widetilde{\F}_n\text{ such that}\\
\lim_{n\to\infty}\E(u-u_{n},u-u_{n})=0\text{ and }\lim_{n\to\infty}u_{n}(x)=u(x)~\forall~x\in {\rm H}\}
\end{multline*}
is a resistance form on $\rm H$.
\end{proposition}
\begin{remark}\label{rem:SG3.04}
Any function in the domain of $\widetilde{\E}$ is continuous. By definition, c.f.\ Definition~\ref{def:SG3.03}, $\dom\E$ is the domain of the energy form provided by Theorem~\ref{thm:SG3.01} and any function in $\dom\E$ is (H\"older) continuous with respect to the Euclidean metric, c.f.\ Theorem~\ref{thm:SG3.01}(iii). In fact, in view of Theorem~\ref{thm:SG3.01}(ii) and Theorem~\ref{thm:SG3.02} we can also write
\begin{equation*}
\dom\widetilde{\E}=\{u\colon {\rm H}\to\mbbR~|~u\text{ is continuous and~\eqref{eq:SG3.14} is finite}\}.
\end{equation*}
\end{remark}
Moreover, this energy satisfies the following crucial property.
\begin{lemma}\label{lemma:SG3.01}
The resistance form $(\widetilde{\E},\dom\widetilde{\E})$ is graph-directed self-similar, i.e.
\begin{equation}\label{eq:SG3.16}
\widetilde{\E}(u,v)=\sum_{i=1}^6r\widetilde{\E}(u\circ\phi_i,v\circ\phi_i)+\sum_{\I\in J}\E^{\I}(u_{|_{\I}},v_{|_{\I}})+\E^{\SG}(u_{|_{\SG_\alpha}},v_{|_{\SG_\alpha}})
\end{equation}
for any $u,v\in\dom\widetilde{\E}$.
\end{lemma}
\begin{proof}
By~\cite[Lemma 5.1]{HN03}, the equality holds for $(\E,\dom\E)$ and since $\dom\widetilde{\E}\subseteq\dom\E$, 
\begin{equation*}\label{eq:SG3.17}
\widetilde{\E}(u,v)=\sum_{i=1}^6r\E(u\circ\phi_i,v\circ\phi_i)++\sum_{\I\in J}\E^{\I}(u_{|_{\I}},v_{|_{\I}})+\E^{\SG}(u_{|_{\SG_\alpha}},v_{|_{\SG_\alpha}}),
\end{equation*}
for any $u,v\in\dom\widetilde{\E}$. Thus, it remains to show that for any $u\in\dom\widetilde{\E}$ and $i\in\mcA_1$, $u\circ{\phi}_i\in\dom\widetilde{\E}$. Let $\{u_n\}_{n\geq 1}\subseteq \cup_{n\geq 1}\widetilde{\F}_n$ be the sequence that approximates $u$. Since $u_n$ is constant on $n$-cells, it follows that $u_n\circ\phi_i\in\widetilde{\F}_{n+1}$ for each $i\in\mcA_1$. Moreover, for all $x\in{\rm H}$, $\phi_i(x)\in {\rm H}$ and hence $\lim_{n\to\infty}u_n\circ{\phi}_i(x)=u\circ{\phi_i}(x)$. Finally, in view of~\eqref{eq:SG3.16} we get
\begin{equation*}\label{eq:SG3.18}
\widetilde{\E}(u\circ{\phi_i}-u_n\circ{\phi_i},u\circ{\phi_i}-u_n\circ{\phi_i})\leq \frac{1}{r}\E(u-u_n,u-u_n)
\end{equation*}
which tends to zero as $n\to\infty$ because $u\in\dom\widetilde{\E}$.
\end{proof}

Recall that the resistance scaling factors have been uniquely determined by the condition imposed in~\eqref{eq:SG3.05}. The associated graph directed self-similar resistance form is thus unique for these resistances, see~\cite[Section 7]{HMT06}, and therefore Lemma~\ref{lemma:SG3.01} yields the main result of this section.

\begin{theorem}\label{thm:SG3.02}
The resistance forms $(\E,\dom\E)$ and $(\widetilde{\E},\dom\widetilde{\E})$ coincide.
\end{theorem}
Finally, equipping the hybrid $\rm H$ with a Radon measure $\mu$, the resistance form $(\E,\dom\E)$ induces a local and regular Dirichlet form on $L^2({\rm H},\mu)$, see e.g.~\cite[Theorem 9.4]{Kig12}. In view of the characterization of the Dirichlet form given in~\eqref{eq:SG3.14} we deduce from Theorem~\ref{thm:SG3.02} the following result concerning the fractal dust associated with the hybrid.
\begin{corollary}
The energy measure associated with $(\E,\dom\E)$ does not charge the fractal dust $\mcC_{\rm H}$.
\end{corollary}

\section{Hybrids with dihedral-3 symmetric base}\label{section:HD3S}
The first straightforward extension of the results presented in the previous section concerns hybrids whose base $K_\ast$ is a \textit{Sierpinski-like fractal}, i.e. a p.c.f. self-similar set with diahedral-3 symmetry. These sets were introduced and studied in~\cite{Str00}. From the geometric point of view, a set of this type is generated by an iterated function system (i.f.s.)\ 
of similarities in $\mbbR^2$ 
with contraction ratio $c_*\in(0,1)$. Its boundary $V_0$ consists of three points and there exists a group of homeomorphisms of $K_\ast$ isomorphic to $D_3$ that acts as permutations on $V_0$ and preserves that self-similar structure of $K_\ast$.

\begin{definition}\label{def:HD3S.01}
Let $K_*$ be a Sierpinski-like fractal in $\mbbR^2$ with associated i.f.s.\ $\{\phi_{*,i}\}_{i\in\mcS}$ and let $V_0$ denote its boundary. 
\begin{itemize}[leftmargin=2em]
\item[(i)] Let $\mcC_*$ be the self-similar Cantor set generated by the i.f.s.\ $\{\phi_{i}\}_{i\in\mcS}$, where each $\phi_i$ is the similitude in $\mbbR^2$ defined by substituting in $\phi_{*,i}$ the contraction ratio $c_*$ by a smaller one $0<c<c_*$.
\item[(ii)] Let $B:=\{(i,j)\in\mcS^2~|~i\neq j\}$ and let $\{K_{(i,j)}\}_{(i,j)\in B}$ be a family of p.c.f.\ self-similar sets connected one another through their boundary points in such a way that $\partial K_{(i,j)}=K_{(i,j)}\cap(\phi_i(V_0)\cup\phi_j(V_0))$. The unique non-empty compact set satisfying
\begin{equation*}\label{eq:HD3S.01}
{\rm H}_{K_*}=\bigcup_{i\in\mcS}\phi_i({\rm H}_{K_*})\cup\bigcup_{j\in B}K_j
\end{equation*}
is called the hybrid fractal of base $K_*$ and bonds $\{K_{(i,j)}\}_{(i,j)\in B}$.
\item[(iii)] The set $\mcC_*$ is called the fractal dust associated with ${\rm H}_{K_*}$.
\end{itemize}
\end{definition}

From the analytic point of view, see~\cite[Section 5]{Str00}, and because of the dihedral-3 symmetry, 
the base fractal $K_*$ has an associated harmonic structure $(D,\bm{r})$, where 
\begin{equation*}\label{eq:HD3S.02}
D=\lambda\begin{pmatrix}
-2&1&1\\
1&-2&1\\
1&1&-2
\end{pmatrix}
\end{equation*}
for some $\lambda >0$ and $\bm{r}=(r_1,r_2,r_3)$ is the vector of weights/resistances. We will assume that $r_1=r_2=r_3=r$ and without loss of generality take $\lambda=1$. 

\medskip

Furthermore, each ``bond'' $K_{(i,j)}$ is equipped with a harmonic structure $(D_{(i,j)},\bm{r}_{(i,j)})$ that leads to a resistance form $(\E^{K_{(i,j)}},\dom\E^{K_{(i,j)}})$ on $K_{(i,j)}$. Again, write $\phi_\alpha=\phi_{\alpha_1}\circ\cdots\circ\phi_{\alpha_n}$ for each word $\alpha\in\mcS^n$ with $n\geq 1$ and $\phi_{\text{\o}}={\rm id}$. Following mutatis mutandis the previous section we obtain a characterizable graph-directed self-similar local and regular resistance form on ${\rm H}_{K_*}$.

\begin{theorem}\label{thm:HD3S.01}
\begin{itemize}[leftmargin=.25in]
\item[(i)]There exists a unique local and regular resistance form $(\E,\dom\E)$ on ${\rm H}_{K_*}$ such that
\begin{equation*}\label{eq:HD3S.03}
\E(u,v)=\sum_{i\in\mcS}r\E(u\circ\phi_i,v\circ\phi_i)+\sum_{(i,j)\in B}\E^{K_{(i,j)}}(u_{|_{K_{(i,j)}}},v_{|_{K_{(i,j)}}})
\end{equation*}
for any $u,v\in\dom\E$.
\item[(ii)] The resistance form $(\E,\dom\E)$ is given by
\begin{multline}\label{eq:HD3S.04}
\dom\E=\{u\in C({\rm H}_{K_*})~|~\text{there exists}~\{u_n\}_{n\geq 1}\subseteq\cup_{n\geq 1}\widetilde{\F}_n\text{ such that}\\
\lim_{n\to\infty}\E(u-u_{n},u-u_{n})=0\text{ and }\lim_{n\to\infty}u_{n}(x)=u(x)~\forall~x\in {\rm H}\}
\end{multline}
and 
\begin{equation*}\label{eq:HD3S.05}
\E(u,v)=\sum_{k=0}^\infty\sum_{(\alpha,(i,j))\in \mcS^k\times B}\E_k^{K_{\alpha,ij}}(u,v),
\end{equation*}
where for each $k\geq 0$ and $u,v\in\dom\E$,
\begin{equation*}\label{eq:HD3S.06}
\E_k^{K_{\alpha,ij}}(u,v)=\frac{1}{r^{k}}\E^{K_{(i,j)}}(u\circ\phi_\alpha{}_{|_{K_{(i,j)}}},v\circ\phi_\alpha{}_{|_{K_{(i,j)}}}).
\end{equation*}
\end{itemize}
\end{theorem}

\begin{remark}\label{rem:HD3S.01}
Once again, the self-similar choice of the resistances ($r$ for the base fractal and $1$ for the bonds) as well as the self-similar structure imposed by the Sierpinski-like fractal guarantee the uniqueness of a self-similar graph-directed energy for these fixed parameters.
The key fact that leads to the characterization~\eqref{eq:HD3S.04} is the analogue of Lemma~\ref{lemma:SG3.01} and the uniqueness of the resistance form $(\E,\dom\E)$. These hold for any graph-directed graph equipped with a self-similar energy.
\end{remark}


%

\section{Characterization condition of energy on finitely ramified cell structures}\label{section:FRCS}
So far we have discussed (graph-directed) self-similar energies on different types of fractals and it has turned out that they consist of a countable sum of energies on copies of the building blocks. All these sets fall under a larger class of spaces called \textit{finitely ramified cell structures} that were introduced in~\cite{Tep08}. In this section we investigate the possibility of characterizing non self-similar energies within this more abstract setting and give a condition under which they admit the same characterization as in the self-similar case.
\subsection{Resistance forms on finitely ramified cell structures}
We start by introducing the concept of a finitely ramified cell structure from~\cite[Definition 2.1]{Tep08} for an arbitrary countable set. Basic definitions and standard notation for resistance forms are reviewed in the Appendix.
\begin{definition}\label{def:FRCS01}
A countable set $V_*$ is said to support a finitely ramified cell structure if there exist an index set $\mcA$, a cell structure $\{K_\alpha\}_{\alpha\in\mcA}$ and a family of weighted graphs $\{(V_\alpha,E_\alpha,r_\alpha)\}_{\alpha\in\mcA}$ 
that satisfy the following properties.
\begin{enumerate}[leftmargin=.3in,label=(\alph*)]
\item $\mcA$ is a countable set,
\item each $K_\alpha$ is a distinct countable subset of $V_*$,
\item each $V_\alpha\subsetneq K_\alpha$ is finite and has at least two elements,
\item for each $x,y\in V_\alpha$, $\{x,y\}\in E_\alpha$ if and only if $0<r_\alpha(x,y)<\infty$,
\item if $K_\alpha=\bigcup_{i=1}^kK_{\alpha_i}$, then $V_{\alpha}\subseteq \bigcup_{i=1}^k V_{\alpha_i}$,
\item there exists a filtration $\{\mcA_n\}_{n\geq 0}$ such that
\begin{enumerate}[leftmargin=*]
\item[(f1)] $\mcA_n$ are finite subsets of $\mcA$, $\mcA_0=\{\text{\o}\}$ and $K_{\text{\o}}=V_*$,
\item[(f2)] $\mcA_m\cap\mcA_n=\emptyset$ if $m\neq n$,
\item[(f3)] for any $\alpha\in\mcA_n$ there exist $\alpha_1,\ldots,\alpha_k\in\mcA_{n+1}$ such that $K_\alpha=\bigcup_{i=1}^kK_{\alpha_i}$, 
\end{enumerate}
\item $K_\alpha\cap K_{\alpha'}=V_\alpha\cap V_{\alpha'}$ for any $\alpha\neq\alpha'$ in $\mcA$,
\item for any strictly decreasing infinite cell sequence $K_{\alpha_1}\supsetneq K_{\alpha_2}\supsetneq\ldots$ there exists $x\in V_*$ such that $\bigcap_{n\geq 1}K_{\alpha_n}=\{x\}$.
\end{enumerate}
\end{definition}
Any triple $(V_*,\{K_\alpha\}_{\alpha\in\mcA},\{(V_\alpha,E_\alpha,r_\alpha)\}_{\alpha\in\mcA})$ is called a \textit{finitely ramified cell structure}.

\begin{notation}
For each $n\geq 0$, set $V_n=\bigcup_{\alpha\in\mcA_n}V_\alpha$ and $V_*=\bigcup_{n\geq 0}V_n$. For each $\alpha\in\mcA$, we write $V_{\alpha,n}=K_\alpha\cap V_n$. Notice that $K_\alpha=\bigcup_{n\geq 0} V_{\alpha,n}$.
\end{notation}
\begin{definition}\label{def:FRCS02}
For any $\alpha\in\mcA$, define the bilinear form $\E_\alpha\colon\ell(V_\alpha)\times\ell(V_\alpha)\to\mbbR$ as the graph energy associated with $(V_\alpha,E_\alpha,r_\alpha)$, i.e.
\begin{equation*}\label{eq:FRCS01}
\E_\alpha(u,v)=\sum_{\{x,y\}\in E_\alpha}\frac{1}{r_\alpha(x,y)}(u(x)-u(y))(v(x)-v(y))
\end{equation*}
for any $u,v\in\ell(V_\alpha)$. For each $n\geq 0$, define $\E_n\colon\ell(V_n)\times\ell(V_n)\to\mbbR$ as
\begin{equation}\label{eq:FRCS02}
\E_n(u,v)=\sum_{\alpha\in\mcA_n}\E_\alpha(u_{|_{V_\alpha}},v_{|_{V_\alpha}})
\end{equation}
for any $u,v\in\ell(V_n)$.
\end{definition}
The following lemma due to Kigami gives a necessary and sufficient condition for the weight functions $r_\alpha$, $\alpha\in\mcA_n$ to determine a resistance form on $V_n$. 
\begin{lemma}{\cite[Lemma 10.3]{ARFK17}}\label{lemma:FRCS01}
Let $V$ be a finite set. A pair $(\E,\ell(V))$ is a resistance form on $V$ if and only if there exists a weight function $c_V\colon V\times V\to [0,\infty)$, such that for any $x\neq y$, there exist $m\geq 0$ and $x_0,\ldots ,x_m\in V$ with the property that $x_0=x$, $x_m=y$ and $c_V(x_i,x_{i+1})>0$ for any $i=0,\ldots, m-1$, and
\[
\E(u,v)=\frac{1}{2}\sum_{x,y\in V}c_V(x,y)(u(x)-u(y))(v(x)-v(y))
\]
for all $u,v\in\ell(V)$.
\end{lemma}
The definition of harmonic structures in the p.c.f. self-similar setting, c.f.~\cite[Definition 3.1.2]{Kig01}, can be carried to finitely ramified cell structures.

\begin{definition}\label{def:FRCS03}
A finitely ramified cell structure $(V_*,\{K_\alpha\}_{\alpha\in\mcA},\{(V_\alpha,E_\alpha,r_\alpha)\}_{\alpha\in\mcA})$ is said to be harmonic if there exists a filtration $\{\mcA_n\}_{n\geq 0}$ such that the sequence of resistance forms $\{(\E_n,\ell(V_n))\}_{n\geq 0}$ given by~\eqref{eq:FRCS02} is compatible in the sense of Kigami~\cite[Definition 3.12]{Kig12}.
\end{definition}
In view of~\cite[Theorem 3.13 and Theorem 3.14]{Kig12}, a  harmonic finitely ramified cell structure leads to a resistance form $(\E,\F)$ on the countable set $V_*$. From now on and throughout the paper we will only consider this type of finitely ramified cell structures.
\begin{remark}\label{rem:FRCS01}
By construction, the resistance form $(\E,\dom\E)$ is local: if $\supp u\,\cap\,\supp v=\emptyset$, then  we find disjoint $n$-cells $K_\alpha$, $K_{\alpha'}$ such that $\supp u\subseteq K_\alpha$ and $\supp v\subseteq K_{\alpha'}$. In virtue of~\eqref{eq:FRCS02}, $\E_n(u,v)=0$ and hence $\E(u,v)=0$.
\end{remark}
The next proposition employs the concept of harmonic extension and trace of a resistance form, which are recalled in Definition~\ref{def:DB.RF03} and Definition~\ref{def:DB.RF04} respectively.
\begin{proposition}\label{prop:DB.FRCS01}
Let $\alpha\in\mcA$. For each $n\geq 0$, define the resistance form $(\E_{\alpha,n},\ell(V_{\alpha,n}))$ as the restriction of $(\E|_{V_n},\ell(V_n))$ to $V_{\alpha,n}$. The bilinear form $(\E_{K_\alpha},\F_{K_\alpha})$ given by
\begin{equation}\label{eq:DB.FRCS01}
\E_{K_\alpha}(u,u)=\lim_{n\to\infty}\E_{\alpha,n}(u_{|_{V_{\alpha,n}}}u_{|_{V_{\alpha,n}}})
\end{equation}
for any $u\in\F_\alpha=\{u\colon K_\alpha\to\mbbR~|~\E_\alpha(u,u)<\infty\}$, is a resistance form on $K_\alpha$.
\end{proposition}
\begin{proof}
Let $n\geq 0$ and $u\in\ell(V_{\alpha,n})$. Its harmonic extension $\tilde{h}_{n+1}(u)\in\ell(V_{\alpha,n+1})$ is defined as
\begin{equation*}\label{eq:DB.FRCS02}
\tilde{h}_{n+1}(u)=h_{V_{n+1}}(\tilde{u})_{|_{V_{\alpha,n+1}}},
\end{equation*}
where $\tilde{u}\in\ell(V_{n+1})$ is a function such that $\tilde{u}_{|_{V_{\alpha,n}}}=u$. Notice that this definition is independent of the choice of $\tilde{u}$. 
Then,
\begin{align*}\label{eq:DB.FRCS03}
\E_{\alpha,n+1}(\tilde{h}_{n+1}(u),\tilde{h}_{n+1}(u))&=\E(h_{V_{n+1}}(\tilde{u}),h_{V_{n+1}}(\tilde{u}))
=\E|_{V_n}(\tilde{u}_{|_{V_n}},\tilde{u}_{|_{V_n}})=\E_{\alpha,n}(u,u)
\end{align*}
and hence $\{(\E_{\alpha,n},\ell(V_{\alpha,n}))\}_{n\geq 0}$ is a compatible sequence of resistance forms. By~\cite[Theorem 3.13]{Kig12} it follows that the limit~\eqref{eq:DB.FRCS01} is a resistance form on $K_\alpha$. 
\end{proof}

\begin{definition}
The resistance form $(\E_{K_\alpha},\F_{K_\alpha})$ is called the restriction of $(\E,\F)$ to $K_\alpha$.
\end{definition}
\begin{remark}\label{rem:DB.FRCS01}
\begin{itemize}[leftmargin=.25in]
\item[(i)] The latter resistance form is different than the trace $(\E|_{K_\alpha},\F|_{K_\alpha})$.
\item[(ii)] For any $u\in\F_{K_\alpha}$, $\E_{K_\alpha}|_{V_\alpha}(u,u)=\E_\alpha(u_{|_{V_\alpha}},u_{|_{V_\alpha}})$.
\end{itemize}
\end{remark}
\begin{proposition}\label{prop:DB.FRCS02}
Let $m\geq 1$ be fixed. Then,
\begin{equation*}\label{eq:DB.FRCS04}
\E(u,u)=\sum_{\alpha\in\mcA_m}\E_{K_\alpha}(u_{|_{K_\alpha}},u_{|_{K_\alpha}})
\end{equation*}
for any $u\in\F$. Consequently, for any $u\in\F$ and $\alpha\in\mcA_m$, $u_{|_{K_\alpha}}\in\F_{K_\alpha}$.
\end{proposition}
\begin{proof}
Let $n\geq 1$ and $u_n:=u_{|_{V_n}}\in\ell(V_n)$. By definition, $V_n=\cup_{\alpha\in\mcA_m} V_{\alpha,n}$ and hence
\begin{equation*}\label{eq:DB.FRCS05}
\E_n(u_n,u_n)=\sum_{\alpha\in\mcA_m}\E_{\alpha,n}(u_n{}_{|_{V_{\alpha,n}}},u_n{}_{|_{V_{\alpha,n}}})=\sum_{\alpha\in\mcA_m}\E_{\alpha,n}(u_{|_{V_{\alpha,n}}},u_{|_{V_{\alpha,n}}}).
\end{equation*}
Letting $n\to\infty$ in both sides of the equality yields the result.
\end{proof}
\begin{remark}\label{rem:DB.FRCS02}
One might be tempted to write something like $\F=\bigoplus_{\alpha\in\mcA_n}\F_{K_\alpha}$. However this is not true.
\end{remark}

\subsection{Characterizable energy condition}
Let $(V_*,\{V_\alpha\}_{\alpha\in\mcA},\{(V_\alpha,E_\alpha,r_\alpha)\}_{\alpha\in\mcA})$ be a finitely ramified cell structure with filtration $\{\mcA_n\}_{n\geq 0}$ and assume that the following property is satisfied.

\begin{assumption}\label{ass:MR01}
For each $n\geq 0$, there exists a subset $\widetilde{\mcA}_n\subseteq \mcA_n$ such that 
for any $\alpha\in\widetilde{\mcA}_{n+1}$, $K_{\alpha}\subseteq K_{\alpha'}$ for some $\alpha'\in\widetilde{\mcA}_n$, and the set
\begin{equation}\label{eq:MR01.A1}
\mcC:=\bigcap_{n\geq 0}\bigcup_{\alpha\in\widetilde{\mcA}_n}K_\alpha
\end{equation}
is non-empty.
\end{assumption}

\begin{definition}\label{def:MR01}
\begin{enumerate}[leftmargin=*]
\item For each $n\geq 1$, define the bilinear form $(\widetilde{\E}_n,\widetilde{\F}_n)$ by
\begin{equation*}\label{eq:MR01}
\widetilde{\F}_n=\{u\in\F~|~u_{|_{K_\alpha}}\text{ is constant for all }\alpha\in \widetilde{\mcA}_n\}.
\end{equation*}
and $\widetilde{\E}_n=\E_{|_{\widetilde{\F}_n\times\widetilde{\F}_n}}$.

\item Define the bilinear form $(\widetilde{\E},\widetilde{\F})$ by
\begin{multline*}\label{eq:MR02}
\widetilde{\F}=\{u\in\F~|~\text{there exists}~\{u_{n}\}_{n\geq 1}\subseteq\cup_{n\geq 1}\widetilde{\F}_n\text{ such that}\\
\lim_{n\to\infty}\E(u-u_{n},u-u_{n})=0\text{ and }\lim_{n\to\infty}u_{n}(x)=u(x)~\forall~x\in K\}.
\end{multline*}
and $\widetilde{\E}=\E_{|_{\widetilde{\F}\times\widetilde{\F}}}$.
\end{enumerate}
\end{definition}
The next proposition tells us that the resistance form $(\widetilde{\E},\widetilde{\F})$ does not ``see'' the set $\mcC$.
\begin{proposition}\label{prop:MT01}
For any $u\in\widetilde{\F}$
\begin{equation*}\label{eq:MT03}
\widetilde{\E}(u,u)=\lim_{n\to\infty}\sum_{\alpha\in\mcA_n\setminus\widetilde{\mcA}_n}\E_\alpha(u_{|_{K_\alpha}},u_{|_{K_\alpha}}).
\end{equation*}
\end{proposition}

\begin{lemma}\label{lemma:MR01}
For any distinct $x,y\in V_*$, there exists $u\in\widetilde{\F}$ such that $u(x)\neq u(y)$.
\end{lemma}
\begin{proof}
Let $n\geq 1$ be large enough so that $x,y\in V_n$ and define the set of $n$-cells $\mcA_n(x,y)=\{K_\alpha,\alpha\in\mcA_n~|~x\in K_\alpha,\text{ or }y\in K_\alpha,\text{ or } \mcC\cap K_\alpha\neq \emptyset\}$. By~\cite[Proposition 2.9]{Tep08}, $x$ and $y$ belong to different connected components of $\mcA_n(x,y)$. Denote by $C_n(x)$ the connected component of $x$. For each $K_\alpha$ in this component that intersects $\mcC$ and has $\alpha\notin\widetilde{\mcA}_n$, define $u_{\alpha}\in\F_{\alpha}$ to be the harmonic function in $K_{\alpha}$ with boundary values one at the intersection with $\mcC$ and zero otherwise. If $K_\alpha\in C_n(x)$ with $\alpha\notin\widetilde{\mcA}_n$ does not intersect $\mcC$, set $u_\alpha\equiv 0$. Finally, for any $K_\alpha\in C_n(x)$ with $\alpha\in\widetilde{\mcA}_n$, set $u_\alpha\in\F_\alpha$ to be constant one. Then, the function
\begin{equation*}\label{eq:MR03}
u(z)=\begin{cases}
u_{\alpha}(z)&\text{if }z\in K_{\alpha}\in C_n(x),\\
0&\text{otherwise},
\end{cases}
\end{equation*}
is $V_n$-harmonic and belongs to $\widetilde{\F}_n$. Hence $u\in\widetilde{\F}$ and it separates $x$ and $y$ as desired.
\end{proof}
For simplicity of the proofs, we make the following assumption, that may be removed in the future based on the treatment of general resistance forms in~\cite{HT15,Hin16}.
\begin{assumption}\label{ass:MR02}
The closure of $V_*$ with respect to the effective resistance metric $R_\E$ is compact.
\end{assumption}
\begin{theorem}\label{thm:MR02}
$(\widetilde{\E},\widetilde{\F})$ is a resistance form on $V_*$.
\end{theorem}
\begin{proof}
We prove first that $(\widetilde{\E},\widetilde{\F})$ is a resistance form. The properties (RF1) and (RF4) follow immediately from the fact that $\widetilde{\F}\subseteq\F$. The definition of $\widetilde{\F}$ implies (RF2) and Lemma~\ref{lemma:MR01} yields (RF3). 
Finally, since $V_*$ is in particular $R_\E$-bounded, we can adapt~\cite[Theorem 11.6]{ARFK17} to get (RF5).
\end{proof}

\begin{definition}\label{def:MT02}
For each $n\geq 0$ define the mapping $\widetilde{R}_{n}\colon V_*\times V_*\to [0,\infty)$ by
\begin{equation*}\label{eq:MT07}
\widetilde{R}_{n}(x,y)=\sup\bigg\{\frac{|u(x)-u(y)|^2}{\E(u,u)}~|~u\in \widetilde{\F}_{n},~\E(u,u)\neq 0\bigg\}
\end{equation*}
and
\begin{equation*}\label{eq:MT08}
\widetilde{R}(x,y):=\lim_{n\to\infty}\widetilde{R}_{n}(x,y)
\end{equation*}
for any $x,y\in V_*$.
\end{definition}
\begin{lemma}\label{lemma:MR02}
The mapping $\widetilde{R}$ is the resistance metric associated with $(\widetilde{\E},\widetilde{\F})$.
\end{lemma}
\begin{proof}
From the proof of Lemma~\ref{lemma:MR01}, for any $x,y\in V_*$ we can choose $n\geq 1$ large enough so that there is a $V_n$-harmonic function $h\in\widetilde{\F}$ with $h(x)=1$ and $h(y)=1$.
\end{proof}

\begin{notation}
For any $\alpha\in\mcA$, denote by $(\widetilde{\E}_\alpha,\widetilde{\F}_\alpha)$ the restriction of $(\widetilde{\E},\widetilde{\F})$ to $K_\alpha$, and let $R_\alpha$ denote its associated resistance metric.
\end{notation}
\begin{theorem}\label{thm:MT01}
The following statements are equivalent:
\begin{enumerate}[leftmargin=.3in]
\item $(\E,\F)=(\widetilde{\E},\widetilde{\F})$,
\item $(\E|_{V_n},\ell(V_n))=(\widetilde{\E}|_{V_n},\ell(V_n))$ for any $n\geq 0$,
\item $(\E_{\alpha},\F_{\alpha})=(\widetilde{\E}_{\alpha},\widetilde{\F}_{\alpha})$ for all $\alpha\in\mcA$,
\item $R_\alpha(x,y)=\widetilde{R}_\alpha(x,y)$ for any $x,y\in V_\alpha$,
\item $R_\E(x,y)=\widetilde{R}(x,y)$ for any $x,y\in V_*$.
\end{enumerate}
\end{theorem}
\begin{proof}
By~\cite[Theorem 3.13]{Kig12}, (1) and (2) are equivalent. Moreover, by Proposition~\ref{prop:DB.FRCS02} and since $V_n=\cup_{\alpha\in\mcA_n}V_\alpha$ for any $n\geq 1$, we have
\begin{equation}
\E|_{V_n}(u_{|_{V_n}},u_{|_{V_n}})=
\sum_{\alpha\in\mcA_n}\E_{\alpha}(u_{|_{V_{\alpha}}},u_{|_{V_{\alpha}}})
\end{equation}
and therefore (2) and (3) are equivalent. The equivalence of (5) and (1) as well as (3) and (4) follows as a consequence of~\cite[Lemma 2.1.12]{Kig01}.
\end{proof}

\section{p.c.f. hybrid fractals}\label{section:HF}
In this section is devoted to presenting a general notion of hybrid fractals. We explain how to equip them with a natural finitely ramified cell structure and work out an example where Theorem~\ref{thm:MT01} can be applied to obtain a characterized energy that is not self-similar. Although hybrid fractals can be viewed as graph-directed fractals, the type of resistance forms that we consider here is not captured by the setting in~\cite{HN03} and we will thus focus on their finitely ramified cell structure.

\medskip

A \textit{post critically finite hybrid} (p.c.f. hybrid) is based on a p.c.f. self-similar set $K_*$ with self-similar structure $(K_*,\mcS,\{\phi_{c,i}\}_{i\in\mcS})$ and boundary $V_0$, where the contraction mappings have ratio $c\in(0,1)$. We refer to~\cite[Chapter 1]{Kig01} for further details about p.c.f. sets.

\begin{notation}
We denote by $B$ the index set of the ``bonds'' or ``blocks'' that determine the hybrid.
\end{notation}

\begin{definition}\label{def:HF01}
Let $0<\tilde{c}<c<1$. For each $i\in\mcS$, define $\phi_i:=\phi_{\tilde{c},i}$.
\begin{itemize}[leftmargin=.25in]
\item[(i)] Define $\mcC_*\subseteq\mbbR^2$ to be self-similar Cantor set generated by $\{\phi_i\}_{i\in\mcS}$.
\item[(ii)] Let $\{K_j\}_{j\in B}$ be a family of p.c.f. self-similar sets such that $\partial K_j=K_j\cap (\cup_{i\in\mcS'}\phi_i(V_0))$ for some $\mcS'\subseteq \mcS$. The unique non-empty compact subset of $\mbbR^2$ such that
\begin{equation*}\label{eq:HF.01}
{\rm H}_{K_*}=\bigcup_{i\in\mcS}\phi_i({\rm H}_{K_*})\cup\bigcup_{j\in B}K_j
\end{equation*}
is called the hybrid fractal of base $K_*$ and bonds $\{K_j\}_{j\in B}$. 
\item[(iii)] The set $\mcC_*$ is called the fractal dust associated with ${\rm H}_{K_*}$.
\end{itemize}
\end{definition}
Although $\phi_i$ depends on $\tilde{c}$, a different contraction ratio will change the geometry but not the topology of ${\rm H}_{K_*}$. In order to set up a finitely ramified cell structure we introduce first some standard notation.

\begin{definition}\label{def:HF02}
For each $n\geq 0$ define
\begin{equation*}\label{eq:HF02}
W^{\ss{\mcS}}_n=\{\text{words of length }n\text{ in the alphabet }\mcS\},
\end{equation*}
and set $W^{\ss{\mcS}}_*=\cup_{n\geq 0}W^{\ss{\mcS}}_n$. For any $w\in W^{\ss{\mcS}}_*$ the mapping $\phi_w\colon\mbbR^2\to\mbbR^2$ is defined as
\begin{equation*}\label{eq:HF03}
\phi_w=\phi_{w_1}\circ \phi_{w_2}\circ\cdots\circ \phi_{w_n},
\end{equation*}
and $\phi_{\text{\o}}$ is the identity map on $\mbbR^2$. 
\end{definition}

In addition, each bond has an associated self-similar structure $(K_j,\mcS_j,\{\phi_{(j,k)}\}_{k\in\mcS_j})$ with boundary $V_{j,0}$. The sets $W_n^{\ss{\mcS_j}}$, $W_*^{\ss{\mcS_j}}$ and $\phi_{(j,w)}$ are defined entirely analogous using the corresponding alphabet $\mcS_j$.

\begin{definition}\label{def:HF03}
For each $0\leq k\leq n-1$ and $j\in B$, define 
\begin{equation*}\label{eq:HF04}
\mcA^{\sss{j}}_{n,k}:=W^{\ss{\mcS}}_k\times W^{\ss{\mcS_j}}_{n-k}\qquad\text{and}\qquad \mcA^{\sss{j}}_n=\bigcup_{k=1}^{n-1}\mcA^{\sss{j}}_{n-1,k-1}.
\end{equation*}
Further, set 
\begin{equation*}\label{eq:HF05}
\mcA_n:=W^{\ss{\mcS}}_n\cup\bigcup_{j\in B}\mcA_n^{\sss{j}}
\end{equation*} 
and for each $\alpha\in \mcA_n$ define
\begin{equation*}\label{eq:HF06}
K_\alpha=\begin{cases}
\phi_w(V_*)&\text{if }\alpha=w\in W^{\ss{\mcS}}_n,\\
\phi_w\phi_{(j,v)}(V_{j,*})&\text{if }\alpha=(w,v)\in\mcA_n^{\sss{j}},
\end{cases}
\end{equation*}
where $V_*$ and $V_{j,*}$ are countable dense subsets of ${\rm H}_{K_*}$ and $K_j$ respectively. The sets $V_\alpha$ are defined analogously substituting ${\rm H}_{K_*}$ and $K_j$ by their respective boundaries.
\end{definition}
Equipped with suitable weight functions, the weighted graphs $(V_\alpha,E_\alpha,r_\alpha)$ yield a harmonic finitely ramified cell structure. Notice that Assumption~\ref{ass:MR01} is satisfied with 
\begin{equation*}\label{eq:HF07}
\widetilde{\mcA}_n=W^{\ss{\mcS}}_n
\end{equation*}
for each $n\geq 1$. The closure with respect to the Euclidean metric of the set $\mcC$ from~\eqref{eq:MR01.A1} is the fractal dust $\mcC_*$.

\begin{example}[Hanoi attractor]\label{Ex:HF1}
This space, also called Stretched Sierpinski gasket~\cite{ARFK17}, falls into the class of hybrids treated in Section~\ref{section:HD3S}. Its base is the Sierpinski gasket and the set of bonds consists of the three line segments that join each triangle in Figure~\ref{fig:HF01}. 

\begin{figure}[H]
\centering
\input{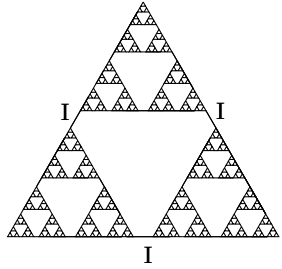}
%
\caption{\small{The bonds $\I\in\{K_j\}_{j\in B}$.}}
\label{fig:HF01}
\end{figure}

In contrast to Section~\ref{section:HD3S}, we choose now a family of weighted graphs that does not lead to a graph-directed self-similar energy. Nevertheless, the resulting resistance form 
enjoys a strong symmetry property and the associated resistance metric induces the same topology as the Euclidean metric. We refer to~\cite{ARFK17} for a thorough study of the resistance forms in this space and setting. 
\begin{figure}[H]
\centering
\begin{tabular}{ccc}
\begin{tikzpicture}[scale=.6]
\draw[thick] ($(240:4)$) circle (2pt) --  ($(240+60:4)$) circle (2pt) -- ($(240:4)+(60:4)$) circle (2pt) -- ($(240:4)$);
\end{tikzpicture}
\hspace*{.5in}
&
\begin{tikzpicture}[scale=.65]
\draw[thick] ($(0:2)$) circle (2pt) --  ($(180:2)$) circle (2pt);
\draw[white] ($(270:1.5)$) to++ ($(0:1)$);
\end{tikzpicture}
\\
\hspace*{.5in}& \hspace*{.5in}&\\
$\downarrow$\hspace*{.5in} &\hspace*{.5in}$\downarrow$ \hspace*{.5in} \\
& &\\
\begin{tikzpicture}[scale=.65]
\draw[thick] ($(240:4)$) circle (2pt) to++ ($(0:1.5)$) circle (2pt) to++ ($(0:1)$) circle (2pt) -- ($(240+60:4)$) circle (2pt) to++ ($(120:1.5)$) circle (2pt) to++ ($(120:1)$) circle (2pt) -- ($(240:4)+(60:4)$) circle (2pt) to++ ($(240:1.5)$) circle (2pt) to++ ($(240:1)$) circle (2pt) -- ($(240:4)$);
\draw ($(240:.75)$) node[left] {\small $r_n$};
\draw ($(300:.75)$) node[right] {\small $r_n$};
\draw ($(240:2)$) node[left] {\small $\rho_n$};
\draw ($(300:2)$) node[right] {\small $\rho_n$};
\draw ($(240:3.25)$) node[left] {\small $r_n$};
\draw ($(300:3.25)$) node[right] {\small $r_n$};
\draw ($(252:4.1)$) node[] {\small $r_n$};
\draw ($(288:4.1)$) node[] {\small $r_n$};
\draw ($(270:3.9)$) node[] {\small $\rho_n$};
\draw ($(240:1.5)$) -- ($(300:1.5)$) node[below, midway] {\small $r_n$};
\draw ($(240:2.5)$) -- ($(240:4)+(0:1.5)$)node[right, midway] {\small $r_n$} ($(300:2.5)$) -- ($(300:4)+(180:1.5)$)node[left, midway] {\small $r_n$};
\end{tikzpicture}
\hspace*{.5in}
&
\begin{tikzpicture}[scale=.65]
\draw[thick] ($(180:2)$) circle (2pt) to ++ ($(0:2)$) circle (2pt);
\draw[thick] ($(180:2)+(0:2)$) to ++ ($(0:2)$) circle (2pt);
\draw[white] ($(270:1.5)$) to++ ($(0:1)$);
\draw ($(180:1)$) node[above] {\small $\frac{1}{2}$};
\draw ($(0:1)$) node[above] {\small $\frac{1}{2}$};
\end{tikzpicture}
\end{tabular}
\caption{\small Graph subdivision and corresponding resistance scaling factors at level $n$.}
\label{fig:HF02}
\end{figure}
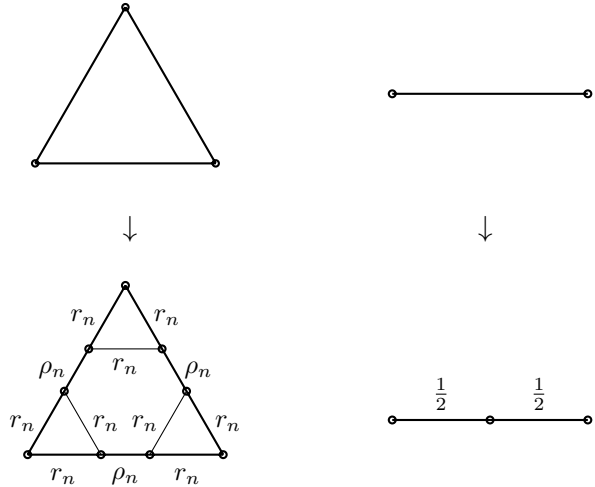
The weight functions $r_\alpha\colon V_\alpha\times V_\alpha\to [0,\infty)$ that arise from the construction in~\cite[Section 5]{ARFK17} are given by
\begin{equation*}\label{eq:HF08}
r_\alpha(x,y)=\begin{cases}
r_1\cdots r_n&\text{if }\alpha\in \widetilde{\mcA}_n,\\
\big(\frac{1}{2}\big)^{n-k}r_1\cdots r_{k-1}\rho_k&\text{if }\alpha\in \mcA^{\sss{j}}_{n-1,k-1}, j\in B,
\end{cases}
\end{equation*}
where the pairs $(r_k,\rho_k)$ satisfy
\begin{equation}\label{eq:HF09}
\frac{5}{3}r_k+\rho_k=1
\end{equation}
for each $k\geq 1$.
The resistance form $(\E,\F)$ on $V_*$ induced by this finitely ramified cell structure is
\begin{equation}\label{eq:HF10}
\E(u,u)=\lim_{n\to\infty}\sum_{\alpha\in\mcA_n}\E_\alpha(u_{|_{V_\alpha}},u_{|_{V_\alpha}})
\end{equation}
for any $u\in\F=\{u\colon V_*\to\mbbR~|~\E_\mcR(u,u)<\infty\}$. Due to~\cite[Theorem 5.16]{ARFK17}, $(\E,\F)$ naturally extends to the whole $\rm H$. In virtue of Theorem~\ref{thm:MR02}, this resistance form admits the expression 
\begin{equation*}
\E_\mcR(u,u)
=\sum_{k=0}^\infty\sum_{(\alpha,j)\in\widetilde{\mcA}_k\times B}\frac{1}{r_1\dots r_{r-1}\rho_k}\int_{\I_{j}}|(u\circ\phi_\alpha)'(x)|^2dx
\end{equation*}
if and only if the resistance metric $R_\E$ and the metric $\widetilde{R}$ from Definition~\ref{def:MT02} coincide. Let us see when this happens.

\medskip

On the one hand, by applying the $\Delta$-Y transform recursively, we obtain that the resistance distance between two boundary points $x,y\in V_0$ is given by
\begin{equation*}\label{eq:HF11}
R_\E(x,y)=\frac{2}{3}\lim_{n\to\infty}\Big(\frac{5}{3}\Big)^n\prod_{i=1}^n{r_i}+\frac{2}{3}\sum_{k=1}^\infty\rho_k\Big(\frac{5}{3}\Big)^{k-1}\Big(\prod_{i=1}^{k-1}{r_i}\Big).
\end{equation*}
On the other hand, the distance between these two points with respect to the metric $\widetilde{R}$ is
\begin{equation*}\label{eq:HF12}
\sum_{k=1}^\infty\rho_k\Big(\frac{5}{3}\Big)^{k-1}\Big(\prod_{i=1}^{k-1}{r_i}\Big).
\end{equation*}
These two quantities coincide if and only if $\lim\limits_{n\to\infty}\big(\frac{5}{3}\big)^n\prod\limits_{i=1}^nr_i=0$, which by~\cite[Lemma 7.1]{ARFK17} is equivalent to 
\begin{equation}\label{eq:HF13}
\sum_{k=1}^\infty\rho_k=\infty.
\end{equation}
Thus, the resistance form $(\E,\F)$ admits the expression~\eqref{eq:HF09} if and only if~\eqref{eq:HF13} holds. In this way we have recovered the necessary and sufficient condition that appeared in~\cite[Theorem 9.1]{ARFK17}.
\end{example}
The next sections are devoted to the study of spectral properties of hybrid fractals, starting with the latter example and afterwards moving to the hybrid $\SG_3$ from Section~\ref{section:SG3}.

\section{Spectrum of the Hanoi attractor}\label{section:SpectrumHanoi}
This section is devoted to the investigation of the spectrum of the Laplacian for the Hanoi attractor discussed in Example~\ref{Ex:HF1}. We denote this hybrid fractal by $\rm H$ and consider the resistance form $(\mathcal E,\mathcal F)$ obtained in~\ref{eq:HF10}. 

\medskip

We will restrict ourselves here to the case treated in~\cite{ARKT16}, where the resistances are given by $r_k=r$, $\rho_k=\rho$ for any $k \in \mathbb{N}$ and some $r,\rho>0$ satisfying~\eqref{eq:HF09}. Thus, $(\mathcal E,\mathcal F)$ is a graph-directed self-similar energy. As in~\cite{ARKT16}, $\rm H$ is equipped with a weakly self-similar measure $\mu$ that depends on a parameter $a\in (0,1/3)$. We refer to~\cite[Section 6]{ARKT16} for more details about this measure.
The Laplace operator $\Delta_\mu$ is defined through the weak formulation
\begin{equation}\label{E:DefDeltaMu}
\mathcal E(u,v)=-\int_H\Delta_{\mu}u v\mathrm{d\mu}
\end{equation}
for any $u\in \dom\Delta_{\mu}\subset\F$ and all $v\in \mathcal F$.

\subsection{Discrete graph approach}
Following Section~\ref{section:SG3}, let $\Gamma_m=(V_m,E_m)$ denote the $m$th level graph approximation of $\rm H$. For computational purposes, the graph subdivision starts after level one, so that $V_1$ consists only of the 9 vertices displayed in Figure~\ref{F:HanoiApprox}. Recall that $V_* = \cup_{m\geq 0} V_m$ is a dense set in $\rm H$.

\begin{figure}[H]
\centering
\begin{tabular}{ccc}
\begin{tikzpicture}[scale=0.375]
\coordinate (p_1) at (0,0);
\fill (p_1) circle (3.5pt);
\coordinate (p_2) at (3,5.1961);
\fill (p_2) circle (3.5pt);
\coordinate (p_3) at (6,0);
\fill (p_3) circle (3.5pt);
\draw (p_1) -- (p_2) -- (p_3)  -- (p_1);
\end{tikzpicture}
\hspace*{2em}
&
\begin{tikzpicture}[scale=0.375]
\coordinate (p_1) at (0,0);
\fill (p_1) circle (3.5pt);
\coordinate (p_2) at (3,5.1961);
\fill (p_2) circle (3.5pt);
\coordinate (p_3) at (6,0);
\fill (p_3) circle (3.5pt);
\coordinate (p_12) at (1,1.732);
\fill (p_12) circle (3.5pt);
\coordinate (p_21) at (2,3.4641);
\fill (p_21) circle (3.5pt);
\coordinate (p_13) at (2,0);
\fill (p_13) circle (3.5pt);
\coordinate (p_31) at (4,0);
\fill (p_31) circle (3.5pt);
\coordinate (p_23) at (4,3.4641);
\fill (p_23) circle (3.5pt);
\coordinate (p_32) at (5,1.732);
\fill (p_32) circle (3.5pt);

\draw (p_12) -- (p_21)  (p_13) -- (p_31)   (p_23) -- (p_32);
\draw (p_1) -- (p_12) -- (p_13) -- (p_1) (p_21) -- (p_2) -- (p_23) -- (p_21) (p_31) -- (p_32) -- (p_3) -- (p_31); 
\end{tikzpicture}
\hspace*{2em}
&
\begin{tikzpicture}[scale= 0.375]
\coordinate (p_1) at (0,0);
\fill (p_1) circle (3.5pt);
\coordinate (p_2) at (3,5.1961);
\fill (p_2) circle (3.5pt);
\coordinate (p_3) at (6,0);
\fill (p_3) circle (3.5pt);
\coordinate (p_12) at (1,1.732);
\fill (p_12) circle (3.5pt);
\coordinate (12_21) at (1.5,2.65);
\fill (12_21) circle (3.5pt);
\coordinate (p_21) at (2,3.4641);
\fill (p_21) circle (3.5pt);
\coordinate (p_13) at (2,0);
\fill (p_13) circle (3.5pt);
\coordinate (p_31) at (4,0);
\fill (p_31) circle (3.5pt);
\coordinate (13_31) at (3,0);
\fill (13_31) circle (3.5pt);
\coordinate (p_23) at (4,3.4641);
\fill (p_23) circle (3.5pt);
\coordinate (p_32) at (5,1.732);
\fill (p_32) circle (3.5pt);
\coordinate (23_32) at (4.5,2.65);
\fill (23_32) circle (3.5pt);
\coordinate (p_112) at (1/3,1.732/3);
\fill (p_112) circle (3.5pt);
\coordinate (p_121) at (2/3,3.4641/3);
\fill (p_121) circle (3.5pt);
\coordinate (p_113) at (2/3,0);
\fill (p_113) circle (3.5pt);
\coordinate (p_131) at (4/3,0);
\fill (p_131) circle (3.5pt);
\coordinate (p_123) at (4/3,3.4641/3);
\fill (p_123) circle (3.5pt);
\coordinate (p_132) at (5/3,1.732/3);
\fill (p_132) circle (3.5pt);
\coordinate (p_212) at (1/3+2,1.732/3+3.4641);
\fill (p_212) circle (3.5pt);
\coordinate (p_221) at (2/3+2,3.4641/3+3.4641);
\fill (p_221) circle (3.5pt);
\coordinate (p_213) at (2/3+2,3.4641);
\fill (p_213) circle (3.5pt);
\coordinate (p_231) at (4/3+2,3.4641);
\fill (p_231) circle (3.5pt);
\coordinate (p_223) at (4/3+2,3.4641/3+3.4641);
\fill (p_223) circle (3.5pt);
\coordinate (p_232) at (5/3+2,1.732/3+3.4641);
\fill (p_232) circle (3.5pt);
\coordinate (p_312) at (1/3+4,1.732/3);
\fill (p_312) circle (3.5pt);
\coordinate (p_321) at (2/3+4,3.4641/3);
\fill (p_321) circle (3.5pt);
\coordinate (p_313) at (2/3+4,0);
\fill (p_313) circle (3.5pt);
\coordinate (p_331) at (4/3+4,0);
\fill (p_331) circle (3.5pt);
\coordinate (p_323) at (4/3+4,3.4641/3);
\fill (p_323) circle (3.5pt);
\coordinate (p_332) at (5/3+4,1.732/3);
\fill (p_332) circle (3.5pt);

\draw (p_112) -- (p_113) (p_121) -- (p_123) (p_131) -- (p_132)
(p_212) -- (p_213) (p_221) -- (p_223) (p_231) -- (p_232) 
(p_312) -- (p_313) (p_321) -- (p_323) (p_331) -- (p_332)
(p_12) -- (p_21)  (p_13) -- (p_31)   (p_23) -- (p_32)
(p_1) -- (p_12) -- (p_13) -- (p_1) (p_21) -- (p_2) -- (p_23) -- (p_21) (p_31) -- (p_32) -- (p_3) -- (p_31); 
\end{tikzpicture}
\end{tabular}
\caption{Approximating graphs $\Gamma_0$, $\Gamma_1$ and $\Gamma_2$.}
\label{F:HanoiApprox}
\end{figure}
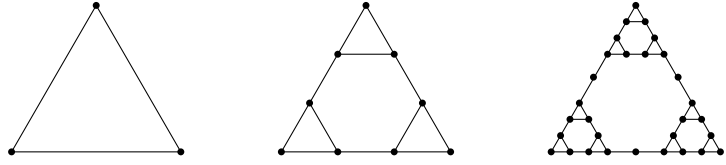

\begin{definition}
For each $m\geq 1$ and $x\in V_m$, the piecewise harmonic function $\psi_{x}^{(m)}\in C(H)$ is defined to be the unique continuous function satisfying the following conditions:
\begin{itemize}[leftmargin=.25in]
\item[(i)] For any $y\in V_m$ and $y\neq x$, $\psi_{x}^{(m)}(y)=0$, while $\psi_{x}^{(m)}(x)=1$.
\item[(ii)] For $n\geq m$, we have that
\begin{equation*}
\sum_{\alpha\in\mcA_n}\E_\alpha(\psi_{x}^{(m)}{}_{|_{V_\alpha}},u_{|_{V_\alpha}})=\sum_{\alpha\in\mcA_m}\E_\alpha(\psi_{x}^{(m)}{}_{|_{V_\alpha}},u_{|_{V_\alpha}}).
\end{equation*}
\end{itemize}
\end{definition}

\begin{definition}
For each $m\geq 1$, the discrete Laplacian on $V_m$ is defined by
\begin{equation}\label{E:DefLaplacian_m}
\Delta_m u(x)=\left(\int_H\psi_{x}^{(m)}\mathrm{d\mu}\right)^{-1}\sum_{\{x, y\} \in E_m}\frac{1}{r_{m}(x,y)}(u(y)-u(x))
\end{equation}
for any $u \in\ell(V_m)$ 
and $x\in V_m$, where $r_m(x,y)$ denotes the resistance between $x$ and $y$, see e.g.~\eqref{eq:SG3.02}.
\end{definition}
Equipped with suitable Dirichlet or Neumann boundary conditions, $-\Delta_m$ is a positive semidefinite self-adjoint operator on $\ell(V_m,\mu_m)$, where the measure $\mu_m$ assigned to each vertex is exactly $\int_{\rm H}\psi_{x}^{(m)}\mathrm{d\mu}$.

\medskip

For any $x\in V_m \setminus V_0$, $x$ is either an endpoint or an internal point of a unique line segment, first born in $V_k$ for some $k\leq m$, denoted by $I^{(m,k)}_{x}$. It can be shown that
\begin{equation*}
\int_H\psi_{x}^{(m)\,d\mu}=\begin{cases}
\frac{1}{3}a^{m}+(\frac{1}{2})^{m-k+1}a^{k-1}(\frac{1}{3}-a)&\text{if $x$ is an endpoint of } I^{(m,k)}_{x},\\
(\frac{1}{2})^{m-k}a^{k-1}(\frac{1}{3}-a)&\text{if $x$ is in the interior of } I^{(m,k)}_{x},
\end{cases}
\end{equation*}
where $a$ is a measure parameter. 
Moreover, if $x$ is an interior point of $I^{(m,k)}_{x}$, then
\begin{equation*}
\Delta_m u(x)=\frac{u(y_0)+u(y_1)-2u(x)}{\left(a^{k-1}(\frac{1}{3}-a)(\frac{1}{2})^{m-k}\right)\left((\frac{1}{2})^{m-k}r^{k-1}\rho\right)}
\end{equation*}
where $y_1$ and $y_2$ are adjacent to $x$ in $I^{(m,k)}_{x}$. Note that by~\eqref{eq:HF09} we have $(5/3) r+\rho=1$. Thus,
\begin{equation}\label{E:RelationLaplacian}
u^{\prime\prime}(x)=\lim_{m\to\infty}\frac{u(y_0)+u(y_1)-2u(x)}{\left((\frac{1}{2})^{m-k}r^{k-1}\rho\right)^{2}}
=\Big(\frac{a}{r}\Big)^{k-1}\frac{1-3a}{3-5r}\Delta_{\mu}u(x).
\end{equation}
Therefore, solving $-\Delta_{\mu}u=\lambda u$ on $\rm H$ will yield a trigonometric function on each interval. The relation between the Laplacian $\Delta_\mu$ introduced in~\eqref{E:DefDeltaMu} and $\Delta_m$ is given in the following theorem.
\begin{theorem}[Pointwise Formula]\label{T:PF}
Let $u \in \dom\Delta_\mu$, then
\begin{equation*}
\Delta_{\mu}u(x)=\lim_{m\rightarrow\infty}\Delta_m u(x)\qquad\text{for all}~x\in V_*\setminus V_0.
\end{equation*}
\end{theorem}
\begin{proof}
Analogous to~\cite[Theorem 2.2.1]{Str01}.
\end{proof}

\subsubsection{Numerical computation Method for the spectrum of $\Delta_m$}
In view of Theorem~\ref{T:PF} and as a first step to understand the behavior of the spectrum of $\Delta_\mu$, the primary goal of this paragraph is to explore the patterns that the spectrum of $\Delta_m$ shows. 
Recall that $\lambda_m$ is called a \textit{Dirichlet} (respectively) \textit{Neumann eigenvalue} of $\Delta_m$ with corresponding eigenfunction $u_m$ when
\begin{equation*}
\begin{cases} 
-\Delta_mu_m(x)=\lambda_m u_m(x) & \text{for } x\in V_m\setminus V_0,\\
u_m(x)=0, & \text{for } x\in V_0,
\end{cases}
\end{equation*}
respectively
\begin{equation*}
\begin{cases} 
-\Delta_mu_m(x)=\lambda_m u_m(x), & \mbox{for } x\in V_m\setminus V_0\\
(u_m(x)-u_m(y_1))+(u_m(x)-u_m(y_2))=0, & \mbox{for } x\in V_0,
\end{cases}
\end{equation*}
where $y_1$ and $y_2$ are adjacent to $x$ in the line segment $I_x^{(m,k)}$. A simple calculation shows that the number of vertices in $V_m \setminus V_0$ is
\begin{equation*}
N_m:=\frac{3}{2}(3^{m+1}-1)-3\cdot 2^{m},
\end{equation*}
hence, by definition, finding Dirichlet or Neumann eigenvalues and eigenfunctions of $\Delta_m$ is tantamount to finding the eigenvalues and eigenfunctions of an $N_m\times N_m$ matrix. In practice, we use the command 'eig' in MATLAB to obtain a numerical solution. Note that 'eig' will give a list of $N_m$ eigenvalues without indicating any information about their multiplicity. Due to numerical errors, for eigenvalues with high multiplicity, elements in the list that actually correspond to the same eigenvalue slightly differ from each other. At the same time, there exist very close but still different eigenvalues in the spectrum. Since it is not obvious how to estimate numerical errors occurring at different levels, or how to detect gaps between distinct eigenvalues, the only way we have been able to obtain each eigenvalue's multiplicity is by manually checking the spectrum.

\subsubsection{Spectrum: data and patterns}
In this paragraph we give several numerical computations of the spectrum and analyze some properties that can be derived from them. A collection of graphics, numerical tools, results and code can be found in the website~\cite{CGSZ17}.

\medskip

\textbf{Existence of D-N eigenvalues. }
On each level, both Dirichlet and Neumann eigenvalues are computed. Among them, the existence of D-N eigenvalues is trivial. For example, at the first level, consider the function with 0 at 3 boundary points and at the middle hexagon, it takes $1$ and $-1$ consecutively. It is a D-N eigenfunction. Meanwhile, if $u_m$ is a D-N eigenfunction on $\Gamma_m$ with eigenvalue $\lambda_m$, then we can contract $u_m$ into one of the second largest triangles on $\Gamma_{m+1}$ and assign other nodes with value 0, see e.g. Figure~\ref{F:Hanoi_mEvs} 
\begin{figure}[H]
\begin{subfigure}{.475\textwidth}
\includegraphics[width=\linewidth]{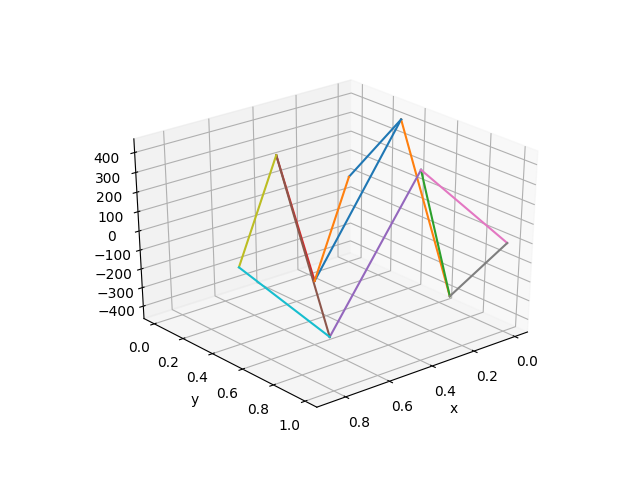}\par 
\caption{{\small $a=r=\frac{1}{6}$, level=1, $6$th eigenfunction=149.54}}
\end{subfigure}
\begin{subfigure}{.475\textwidth}
\includegraphics[width=\linewidth]{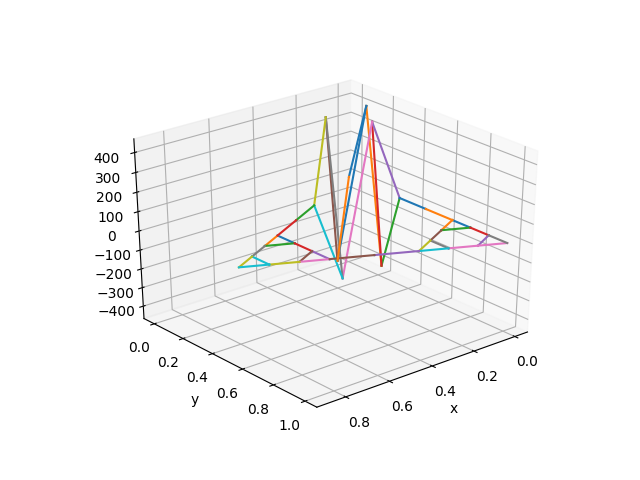}\par 
\caption{{\small $a=r=\frac{1}{6}$, level=2, $25$th eigenfunction=5383.385}}
\end{subfigure}
\caption{Examples of Localized D-N eigenfunctions}
\end{figure}
\begin{figure}[H]\ContinuedFloat
\par\bigskip
\begin{subfigure}[b]{0.45\textwidth}
\includegraphics[width=\linewidth]{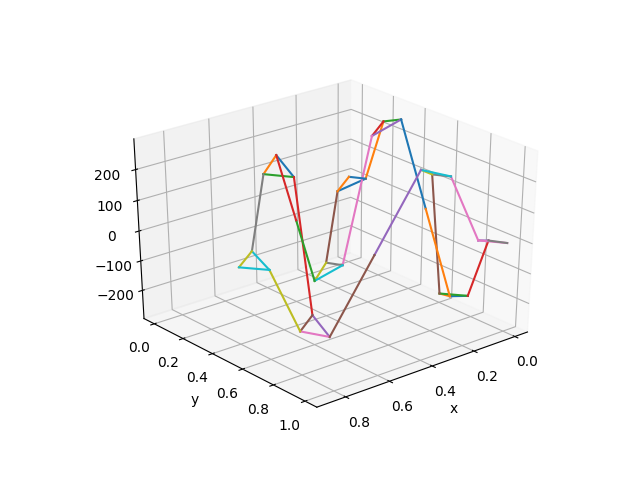}\par
\caption{{\small $a=r=\frac{1}{6}$, level=2, $7$th eigenfunction=257.045}}
\end{subfigure}
\begin{subfigure}[b]{0.45\textwidth}
\includegraphics[width=\linewidth]{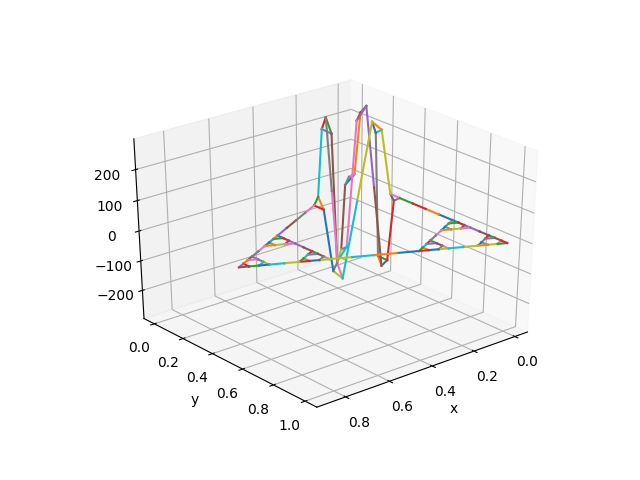}\par
\caption{{\small $a=r=\frac{1}{6}$, level=3, $36$th eigenfunction=9253.609}}
\end{subfigure}
\caption{Localized D-N eigenfunctions, see~\cite{CGSZ17} for further examples.}
\label{F:Hanoi_mEvs}
\end{figure}

It can be verified that such a function generates a localized eigenfunction on $V_{m+1}$. Let us denote it by $u_{m+1}$ and by $\lambda_{m+1}$ its corresponding eigenvalue. Then,
\begin{equation}\label{E:EigvRelation}
\lambda_m = \frac{\lambda_{m+1}}{ra}.
\end{equation}
Hence, all Dirichlet eigenvalues of $\Delta_m$ can be classified into 3 groups: Dirichlet but not Neumann eigenvalues, D-N eigenvalues that satisfy relation~\eqref{E:EigvRelation} with some D-N eigenvalue of $\Delta_{m-1}$, or D-N eigenvalue without the latter property. The table below shows the eigenvalues of $\Delta_m$ for $m=1,\ldots, 7$. Eigenvalues of the second type are colored yellow and eigenvalues of the third type are colored in green.
\begin{table}[H]
\centering
\begin{tabular}{ | c | c | c | c | c | c | c | }
\hline
\textbf{\large{level 1}} & \textbf{\large{level 2}} & \textbf{\large{level 3}} & \textbf{\large{level 4}} & \textbf{\large{level 5}} & \textbf{\large{level 6}} & \textbf{\large{level 7}} \\ \hline
43.20, 1 & 33.98, 1 & 33.74, 1 & 33.69, 1 & 33.68, 1 & 33.68, 1 & 33.68, 1 \\ \hline
57.19, 2 & 47.86, 2 & 49.44, 2 & 49.79, 2 & 49.88, 2 & 49.90, 2 & 49.91, 2 \\ \hline
135.55, 2 & 104.23, 2 & 117.75, 2 & 121.25, 2 & 122.12, 2 & 122.34, 2 & 122.40, 2 \\ \hline
\cellcolor{green!25!}149.54, 1 & 126.28, 1 & 184.79, 1 & 202.29, 1 & 206.86, 1 & 208.02, 1 & 208.31, 1 \\ \hline
 & \cellcolor{green!25!}257.04, 1 & \cellcolor{green!25!}207.30, 1 & \cellcolor{green!25!}217.81, 1 & \cellcolor{green!25!}220.21, 1 & \cellcolor{green!25!}220.80, 1 & \cellcolor{green!25!}220.95, 1 \\ \hline
 & 265.30, 2 & 289.42, 2 & 331.89, 2 & 341.94, 2 & 344.43, 2 & 345.05, 2 \\ \hline
 & 1881.37, 2 & 397.92, 2 & 474.25, 2 & 495.84, 2 & 501.41, 2 & 502.81, 2 \\ \hline
 & 1881.41, 1 & 466.78, 1 & \cellcolor{green!25!}584.16, 1 & \cellcolor{green!25!}615.23, 1 & \cellcolor{green!25!}623.24, 1 & \cellcolor{green!25!}625.26, 1 \\ \hline
 & 3103.35, 2 & \cellcolor{green!25!}{476.77, 1 }& 700.76, 1 & 768.64, 1 & 786.02, 1 & 790.39, 1 \\ \hline
 & 3103.74, 1 & 514.87, 2 & 787.80, 2 & 869.46, 2 & 890.54, 2 & 895.85, 2 \\ \hline
 & \cellcolor{green!25!}3259.84, 1 & 1607.46, 2 & 1089.04, 2 & 1264.93, 2 & 1309.12, 2 & 1320.11, 2 \\ \hline
 & \cellcolor{green!25!}3260.17, 2 & 1607.46, 1 & \cellcolor{green!25!}{ 1130.98, 1} & \cellcolor{green!25!}{ 1328.95, 1} & \cellcolor{green!25!}{ 1379.76, 1} & \cellcolor{green!25!}{1392.49, 1} \\ \hline
 & 4883.03, 2 & \cellcolor{green!25!}2150.98, 1 & 1455.32, 1 & 1653.59, 1 & 1662.94, 1 & 1664.78, 1 \\ \hline
 & 4883.03, 1 & \cellcolor{green!25!}2150.99, 2 & 1481.31, 2 & 1679.55, 2 & 1686.86, 2 & 1688.66, 2 \\ \hline
 & \cellcolor{green!25!}4889.90, 1 & 2314.95, 2 & 1666.05, 2 & 1861.84, 1 & 1948.37, 1 & 1970.24, 1 \\ \hline
 & \cellcolor{green!25!}4889.90, 2 & 2314.96, 1 & 1675.70, 1 & 1870.76, 2 & 1961.04, 2 & 1983.08, 2 \\ \hline
 & \cellcolor{yellow!50!}5383.38, 3 & 4046.57, 2 & 1812.72, 2 & \cellcolor{green!25!}2405.81, 1 & \cellcolor{green!25!}2492.48, 1 & \cellcolor{green!25!}2510.96, 1 \\ \hline
 &  & 4046.57, 1 & \cellcolor{green!25!}1814.01, 1 & 2447.98, 2 & 2565.41, 2 & 2591.25, 2 \\ \hline
 &  & \cellcolor{green!25!}6047.56, 1 & 2042.00, 1 & 2740.18, 2 & 2847.39, 2 & 2876.06, 2 \\ \hline
 &  & \cellcolor{green!25!}6047.56, 2 & 2042.44, 2 & \cellcolor{green!25!}2949.41, 1 & 3148.07, 1 & 3178.77, 1 \\ \hline
 &  & 6264.63, 2 & \cellcolor{green!25!}2557.42, 1 & 3012.62, 1 & \cellcolor{green!25!}3148.76, 1 & \cellcolor{green!25!}3203.72, 1 \\ \hline
 &  & 6264.63, 1 & \cellcolor{green!25!}2557.63, 2 & 3256.00, 2 & 3438.70, 2 & 3485.04, 2 \\ \hline
 &  & \cellcolor{yellow!50!}{ 9253.61, 3} & 2967.71, 2 & 3571.97, 2 & 3954.97, 2 & 4055.68, 2 \\ \hline
 &  & 9551.29, 3 & 2967.75, 1 & 3674.09, 1 & 4019.24, 1 & 4108.50, 1 \\ \hline
 &  & \cellcolor{green!25!}9552.10, 3 & 5053.44, 3 & \cellcolor{green!25!}4260.92, 1 & 4843.37, 2 & 4972.61, 2 \\ \hline
 &  & 67729.76, 3 & \cellcolor{yellow!50!}{7462.73, 3 }& 4270.27, 2 & \cellcolor{green!25!}4923.03, 1 & \cellcolor{green!25!}5092.29, 1 \\ \hline
 &  & \cellcolor{green!25!}83789.93, 3 & \cellcolor{green!25!}7514.45, 3 & 4962.44, 1 & 5503.63, 1 & 5592.27, 1 \\ \hline
 &  & \cellcolor{green!25!}83790.87, 3 & 9156.79, 3 & 5022.27, 2 & 5738.82, 2 & 5866.46, 2 \\ \hline
 &  & 111725.04, 3 & 11142.55, 3 & 5527.03, 2 & 6202.97, 2 & 6437.89, 2 \\ \hline
\end{tabular}
\caption{Bottom part of the Dirichlet spectrum, $r=a=1/6$, levels $1-7$. Eigenvalues are shown with their multiplicity and listed in increasing order.}
\label{Table:val_Dir_spct}
\end{table}
\textbf{Convergence of the spectrum.} 
Another clear pattern is what we call convergence of the spectrum: if $\lambda_m^{(k)}$ denotes the $k$th smallest eigenvalue of $\Delta_m$, we say the spectrum 
of $\Delta_m$ is convergent if the limit
\begin{equation*}
\lim_{m\to\infty} \lambda_m^{(k)}
\end{equation*}
exists for all $k$. From the data 
displayed in Table~\ref{Table:val_Dir_spct}, 
eigenvalues at the bottom of the spectrum seem to converge 
with convergence order 1, 
because 
for any fixed $k$, the ratio 
$\frac{|\lambda_{m+1}^{(k)}-\lambda_m^{(k)}|}{|\lambda_m^{(k)}-\lambda_{m-1}^{(k)}|}$ is roughly a constant. 
The reason for this pattern still remains unclear.
\begin{table}[H]
\centering
\begin{tabular}{|c|c|c|c|c|c|c|c|}
\hline
 & \textbf{\large{level 1}} & \textbf{\large{level 2}} & \textbf{\large{level 3}} & \textbf{\large{level 4}} & \textbf{\large{level 5}} & \textbf{\large{level 6}} & \textbf{\large{level 7}}  \\ \hline
Eigenvalue & 43.2000  & 33.9771  & 33.7412  & 33.6913  & 33.6794  & 33.6764  & 33.6756  \\ \hline
Difference &          &          & 0.2359   & 0.0499   & 0.0119   & 0.0030   & 0.0008   \\ \hline
Eigenvalue & 57.1907  & 47.8622  & 49.4401  & 49.7929  & 49.8791  & 49.9005  & 49.9058  \\ \hline
Difference &          &          & 1.5779   & 0.3528   & 0.0862   & 0.0214   & 0.0053   \\ \hline
Eigenvalue & 135.5477 & 104.2339 & 117.748  & 121.248  & 122.1248 & 122.3441 & 122.3989 \\ \hline
Difference &          &          & 13.5141  & 3.5000   & 0.8768   & 0.2193   & 0.0548   \\ \hline
Eigenvalue & 149.5385 & 126.2839 & 184.7948 & 202.2888 & 206.8625 & 208.0185 & 208.3083 \\ \hline
Difference &          &          & 58.5109  & 17.4940  & 4.5737   & 1.156    & 0.2853   \\ \hline
Eigenvalue &          & 257.0447 & 207.2981 & 217.8053 & 220.2108 & 220.8016 & 220.9487 \\ \hline
Difference &          &          & 49.7466  & 10.5072  & 2.4055   & 0.5908   & 0.1471   \\ \hline
Eigenvalue &          & 265.302  & 289.4171 & 331.8853 & 341.9429 & 344.4302 & 345.0504 \\ \hline
Difference &          &          &          & 42.4682  & 10.0576  & 2.4873   & 0.6202   \\ \hline
\end{tabular}
\caption{Convergence of the Dirichlet spectrum, with the estimated convergence order 1.}
\label{my-label}
\end{table}

\textbf{Top and bottom of the spectrum.} 
For all the measure and resistance parameters tried in the simulations, the multiplicities of the first few eigenvalues are always `1-2-2-1'. For the smallest eigenvalue, we can prove that it has multiplicity 1, and its corresponding eigenfunction is invariant under the dihedral symmetry group $D_3$ however further reasons for this `1-2-2-1' pattern remain unknown. In addition, the top of the spectrum mostly consists of D-N eigenvalues with high multiplicities. The code developed for the simulations is available at~\cite{CGSZ17}.

\begin{table}[H]
\centering
\begin{tabular}{|c|c|}
\hline
\multicolumn{2}{|c|}{level 7} \\ \hline
Dirichlet & Neumann \\ \hline
113759984105.32153, 3 & 25411878638.230247, 3 \\ \hline
\rowcolor{green!25!} 
114102841006, 6 & 114102841006, 6 \\ \hline
\rowcolor{yellow!40!} 
114788072698, 18 & 114788072698, 18 \\ \hline
\rowcolor{yellow!40!} 
116839632421, 54 & 116839632421, 54 \\ \hline
\rowcolor{yellow!40!} 
122950383156, 162 & 122950383156, 162 \\ \hline
\rowcolor{yellow!40!}  
140734901210, 243 & 140734901210, 243 \\ \hline
\rowcolor{yellow!40!}  
140736477695, 243 & 140736477626, 243 \\ \hline
187655171187, 3 & 186882957201, 3 \\ \hline
\rowcolor{green!25!} 
187659146198, 6 & 187659146198, 6 \\ \hline
\rowcolor{yellow!40!}  
187667196655, 18 & 187667196655, 18 \\ \hline
\rowcolor{yellow!40!}  
187692182839, 54 & 187692182839, 54 \\ \hline
\rowcolor{yellow!40!}  
187775569901, 162 & 187775569901, 162 \\ \hline
\rowcolor{yellow!40!}  
188131568230, 243 & 188131568230, 243 \\ \hline
\end{tabular}
\caption{Top of the spectrum, eigenvalues in increasing order, $r=a=1/6$, level 7.}
\end{table}
\begin{table}[H]\ContinuedFloat
\begin{tabular}{|c|c|}
\hline
\multicolumn{2}{|c|}{level 7} \\ \hline
Dirichlet & Neumann \\ \hline
188147174195, 3 & 188147026223, 3 \\ \hline
\rowcolor{green!25!} 
188147176156, 6 & 188147176156, 6 \\ \hline
\rowcolor{yellow!40!}  
188147180228, 18 & 188147180228, 18 \\ \hline
\rowcolor{yellow!40!}  
188147193778, 54 & 188147193778, 54 \\ \hline
\rowcolor{yellow!40!}  
188147252150, 162 & 188147252150, 162 \\ \hline
\rowcolor{yellow!40!}  
197109853834, 243 & 197109853834, 243 \\ \hline
\rowcolor{yellow!40!}  
197130271198, 243 & 197130271198, 243 \\ \hline
\rowcolor{yellow!40!}  
197130271465, 243 & 197130271465, 243 \\ \hline
295258340787, 3 & 294738682923, 3 \\ \hline
\rowcolor{green!25!} 
295259492723, 6 & 295259492723, 6 \\ \hline
\rowcolor{yellow!40!}  
295261805573, 18 & 295261805573, 18 \\ \hline
\rowcolor{yellow!40!}  
295268816917, 54 & 295268816917, 54 \\ \hline
\rowcolor{yellow!40!}  
295290531946, 162 & 295290531946, 162 \\ \hline
\rowcolor{yellow!40!}  
295362581347, 243 & 295362581347, 243 \\ \hline
\rowcolor{yellow!40!}  
295362609295, 243 & 295362609295, 243 \\ \hline
\rowcolor{yellow!40!}  
295673500703, 243 & 295673500703, 243 \\ \hline
\rowcolor{yellow!40!}  
295673630427, 486 & 295673630427, 486 \\ \hline
\rowcolor{yellow!40!}  
325512681630, 729 & 325512681630, 729 \\ \hline
\end{tabular}
\caption{(cont.) Top of the spectrum, eigenvalues in increasing order, $r=a=1/6$, level 7.}
\end{table}


\subsubsection{Eigenvalue counting function}
A simple modification of the p.c.f.\ case yields that the self-adjoint operator $-\Delta_m$ with either Dirichlet of Neumann boundary conditions has a compact resolvent. Therefore, its spectrum is pure point, eigenvalues all have finite multiplicity and the only accumulation point is $\infty$. In this paragraph we study the corresponding \textit{eigenvalue counting function}
\begin{equation}\label{E:Def_Ecf}
N^{(m)}_{N/D}(x)=\#\{\lambda_m\text{ (D/N)eigenvalue of }\Delta_m~|~\lambda_m\leq x\}
\end{equation}
for several choices of the measure and resistance parameters.

\medskip

The horizontal lines in the graphs of $N^{(m)}_{D/N}(x)$ displayed in Figure~\ref{F:Ecf_discrete_plot} represent gaps in the spectrum whereas vertical lines account the multiplicities. Plotting the eigenvalue counting function in a `log-log' scale, see Figure~\ref{F:Ecf_discrete_plot}, a special bifurcation pattern is revealed that hints to different slopes in the beginning part and the later part of the function. 
Since the convergence of the spectrum of $\Delta_m$ starts from the bottom parts, we apply linear regression to approximate the beginning part of each `log-log' plot, and estimate the spectral dimension, i.e. the number $d_S>0$ such that
\begin{equation*}
N_{N/D}(x)\sim x^{d_S/2}\qquad\text{as }x\to\infty.
\end{equation*}

\begin{figure}[H]
\begin{subfigure}{.45\textwidth}
\includegraphics[width=\linewidth]{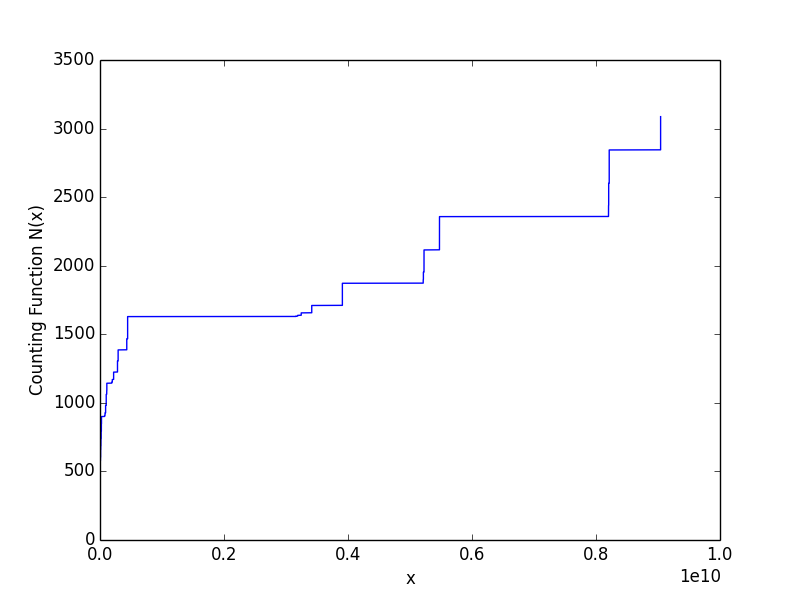}\par 
\caption{{\small $a=r=\frac{1}{6}$, level=6}}
\end{subfigure}
\begin{subfigure}{.45\textwidth}
\includegraphics[width=\linewidth]{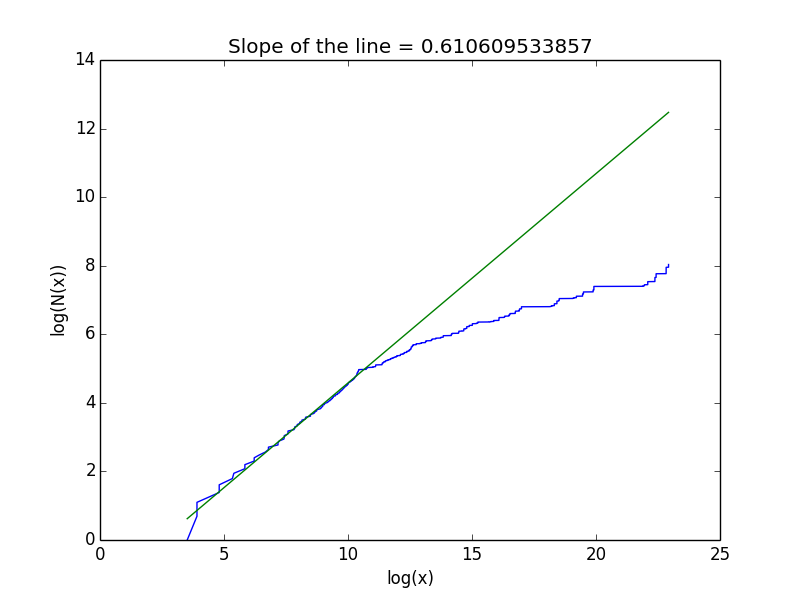}\par 
\caption{$N(x) \sim x^{0.61061}$}
\end{subfigure}
\par\bigskip
\begin{subfigure}[b]{0.45\textwidth}
\includegraphics[width=\linewidth]{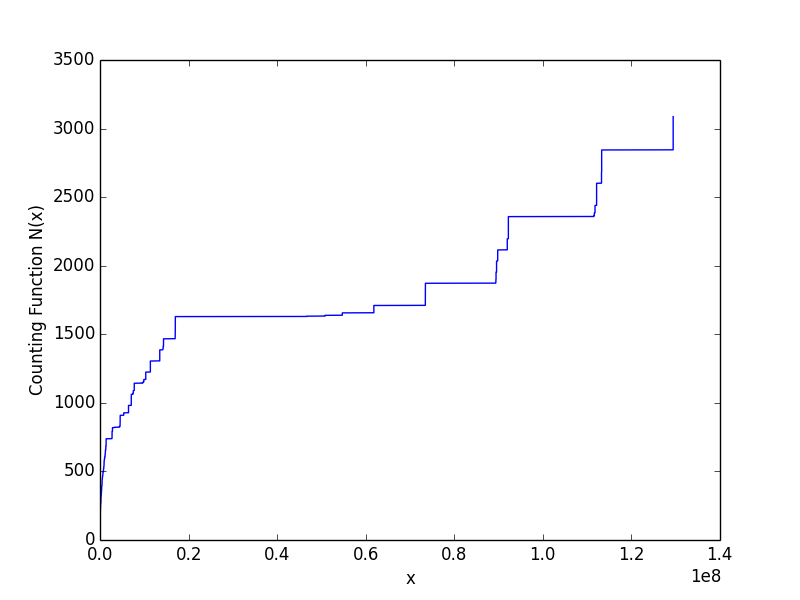}\par
\caption{{\small $a=r=\frac{1}{4}$, level=6}}
\end{subfigure}
\begin{subfigure}[b]{0.45\textwidth}
\includegraphics[width=\linewidth]{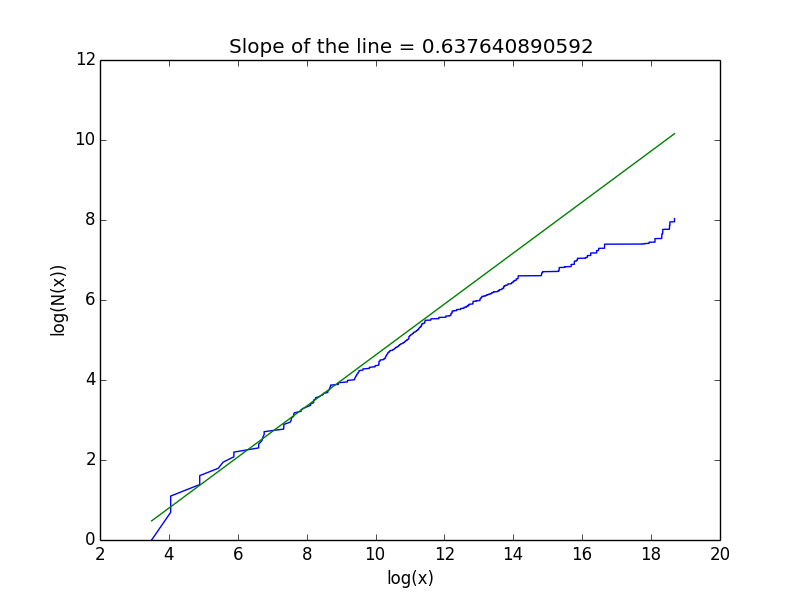}\par 
\caption{$N(x) \sim x^{0.63764}$}
\end{subfigure}
\par\bigskip
\begin{subfigure}[b]{0.45\textwidth}
\includegraphics[width=\linewidth]{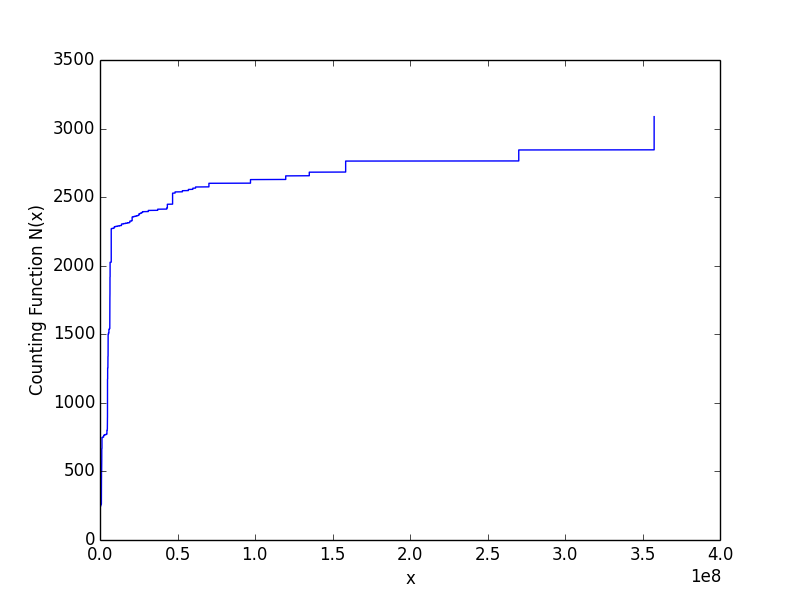}\par 
\caption{{\small $a=r=\frac{333}{1000}$, level=6}}
\end{subfigure}
\begin{subfigure}[b]{0.45\textwidth}
\includegraphics[width=\linewidth]{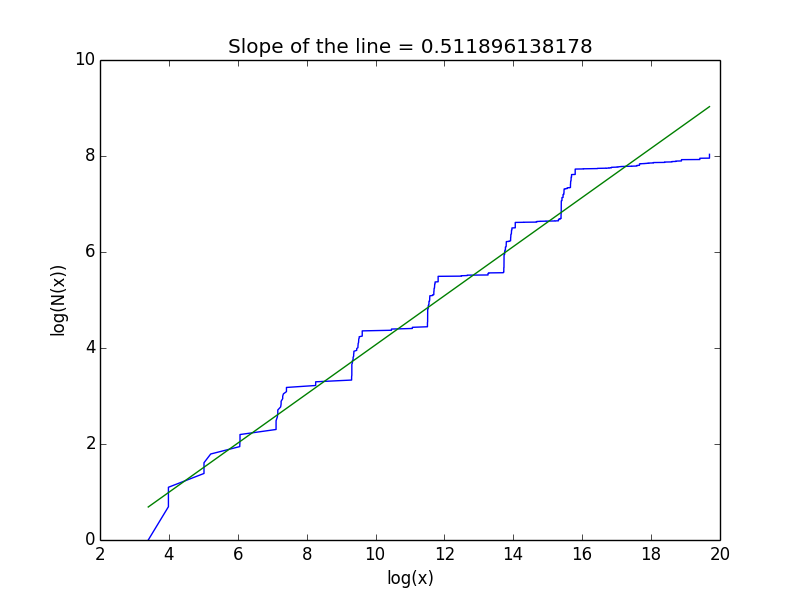}\par 
\caption{$N(x) \sim x^{0.51190}$}
\end{subfigure}
\caption{Eigenvalue counting functions of discrete Dirichlet Laplacians, see~\cite{CGSZ17} for further examples.}
\end{figure}
\begin{figure}[H]\ContinuedFloat
\par\bigskip
\begin{subfigure}[b]{0.45\textwidth}
\includegraphics[width=\linewidth]{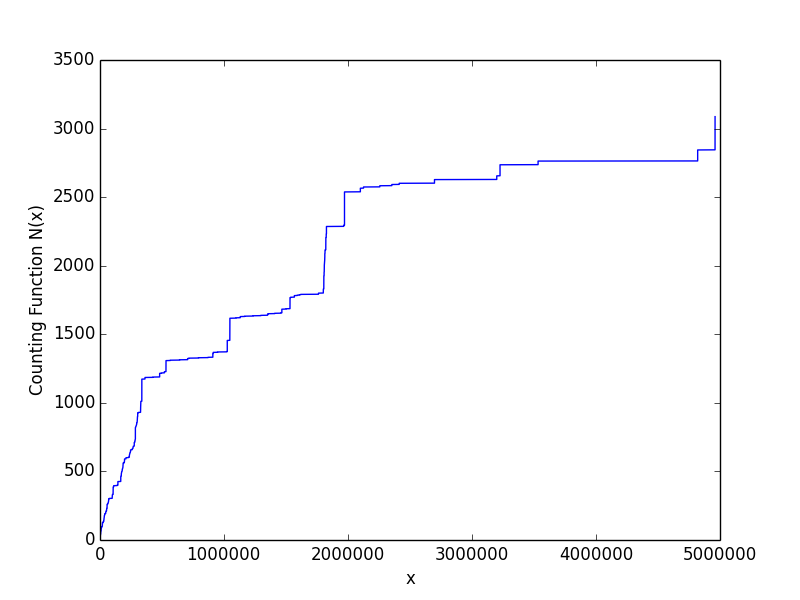}\par 
\caption{{\small $a=\frac{10}{33}, r=\frac{11}{20}$, level=6}}
\end{subfigure}
\begin{subfigure}[b]{0.45\textwidth}
\includegraphics[width=\linewidth]{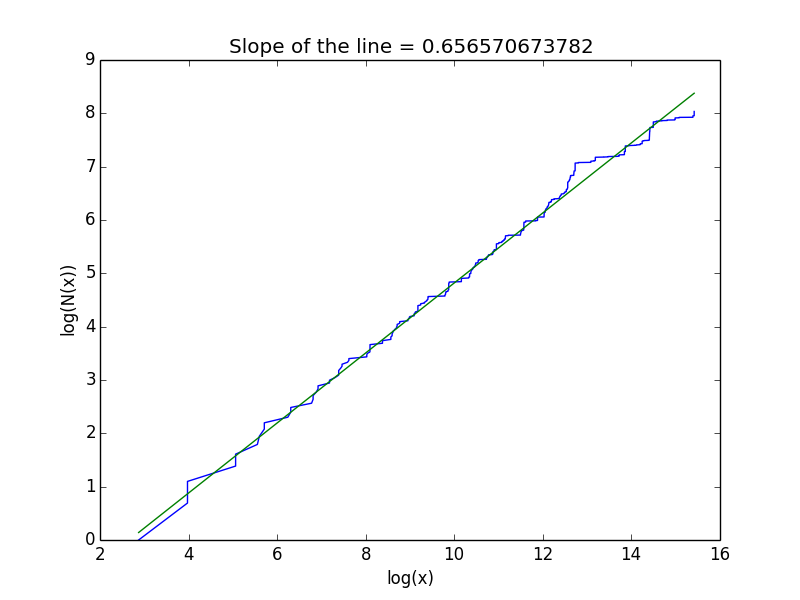}\par 
\caption{$N(x) \sim x^{0.65657}$}
\end{subfigure}
\caption{(cont.) Eigenvalue counting functions of discrete Dirichlet Laplacians, see~\cite{CGSZ17} for further examples.}
\label{F:Ecf_discrete_plot}
\end{figure}

The estimated spectral dimension allows us to plot an approximated Weyl ratio 
\begin{equation*}
\frac{N_{N/D}(x)}{x^{d_S/2}}.
\end{equation*}
The following graphs are the 'y-$\log(x)$' plots of Weyl ratios corresponding to different choices of $a$ and $r$ at level 6. We can observe that for small $x$, where the spectrum begins to converge, the graphs look like periodic functions.

\begin{figure}[H]
\begin{subfigure}{.45\textwidth}
\includegraphics[width=\linewidth]{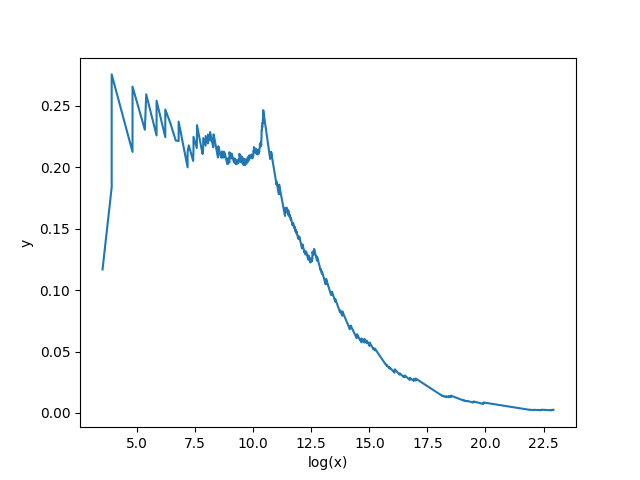}\par 
\caption{{\small $a=r=\frac{1}{6}$, level=6}}
\end{subfigure}
\begin{subfigure}{.45\textwidth}
\includegraphics[width=\linewidth]{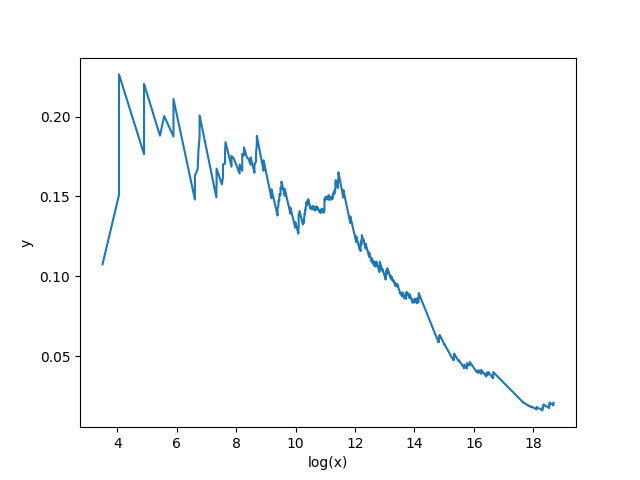}\par 
\caption{{\small $a=r=\frac{1}{4}$, level=6}}
\end{subfigure}
\caption{'y-$\log(x)$' plots of Weyl ratios corresponding to different choices of $a$ and $r$ at level 6, see~\cite{CGSZ17} for further examples.}
\end{figure}
\begin{figure}[H]\ContinuedFloat
\par\bigskip
\begin{subfigure}[b]{0.45\textwidth}
\includegraphics[width=\linewidth]{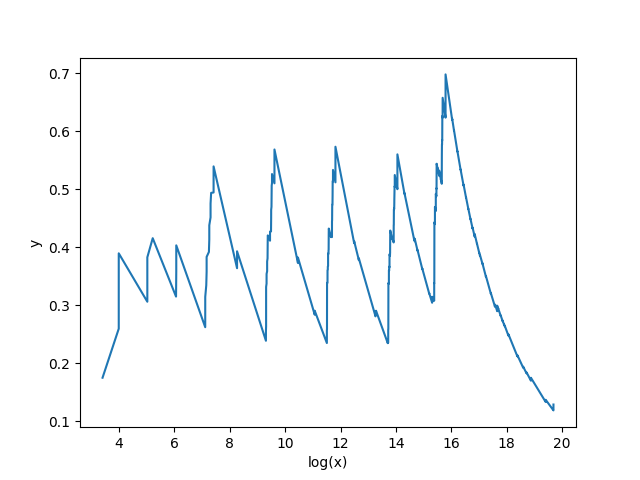}\par
\caption{{\small $a=r=\frac{333}{1000}$, level=6}}
\end{subfigure}
\begin{subfigure}[b]{0.45\textwidth}
\includegraphics[width=\linewidth]{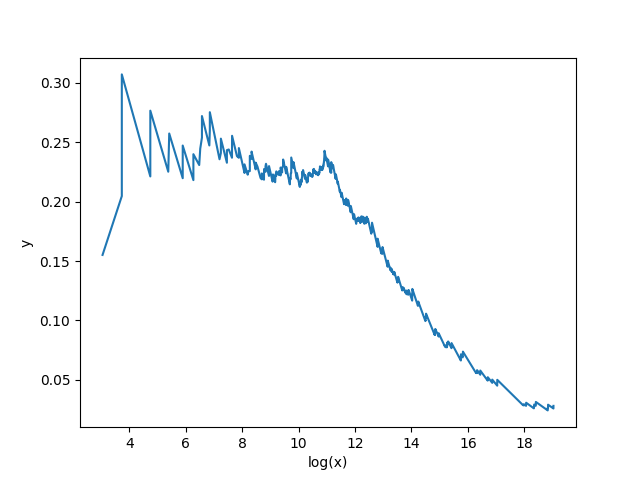}\par 
\caption{{\small $a=\frac{1}{6}$, $r=\frac{333}{1000}$, level=6}}
\end{subfigure}
\par\bigskip
\begin{subfigure}[b]{0.45\textwidth}
\includegraphics[width=\linewidth]{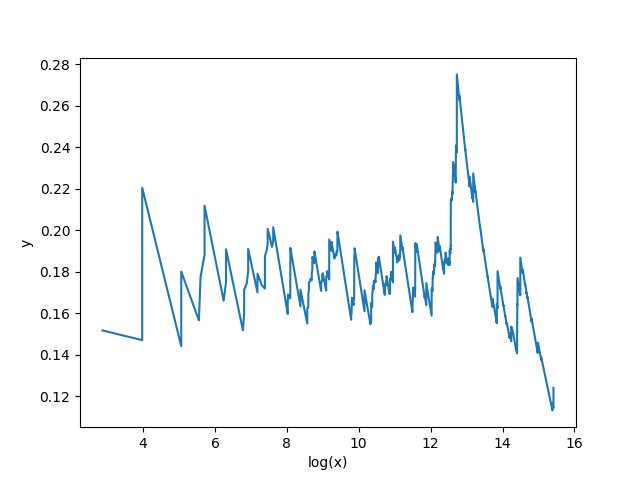}\par 
\caption{{\small $a=\frac{10}{33}$, $r=\frac{11}{20}$, level=6}}
\end{subfigure}
\begin{subfigure}[b]{0.45\textwidth}
\includegraphics[width=\linewidth]{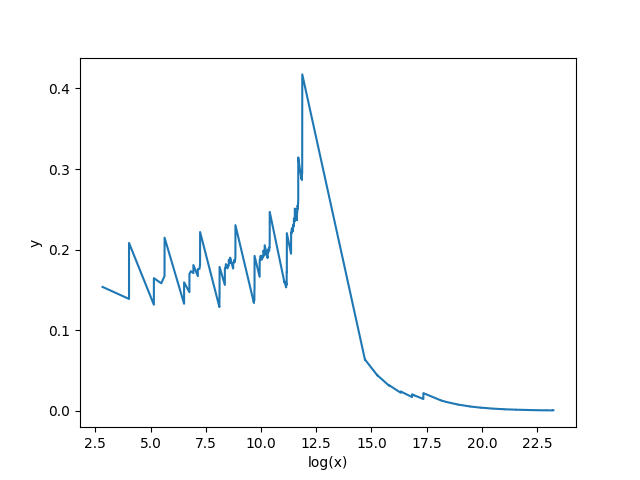}\par 
\caption{{\small $a=\frac{333}{1000}$, $r=\frac{599}{1000}$, level=6}}
\end{subfigure}
\caption{(cont.) 'y-$\log(x)$' plots of Weyl ratios corresponding to different choices of $a$ and $r$ at level 6, see~\cite{CGSZ17} for further examples.}
\label{F:Ecf_weylratio_plot}
\end{figure}

One can observe that this quantity approaches the Weyl ratio of the Sierpinski gasket as $a$ tends to $\frac{1}{3}$ and $r$ tends to $\frac{3}{5}$.


\subsubsection{Eigenfunctions on the Hanoi attractor}
The plots displayed in Figure~\ref{F:Efcts_discrete_Hanoi_1} and~\ref{F:Efcts_discrete_Hanoi_2} show that the 'spectral decimation' scheme valid in some fractal spaces, see e.g.~\cite{FS92,Str01} does not apply to the Hanoi attractor. In the computations presented, the parameters have been chosen to be $r=a=\frac{1}{6}$ and $r=a=\frac{1}{10}$. At each level, the eigenfunction corresponding to the smallest eigenvalue is restricted to one edge of $\Gamma_m$. The graphs of the functions show that the corresponding eigenfunction from level $m+1$ does not match the eigenfunction from level $m$ at some vertices of $\Gamma_m$.

\begin{figure}[H]
\centering
\includegraphics[scale=.37]{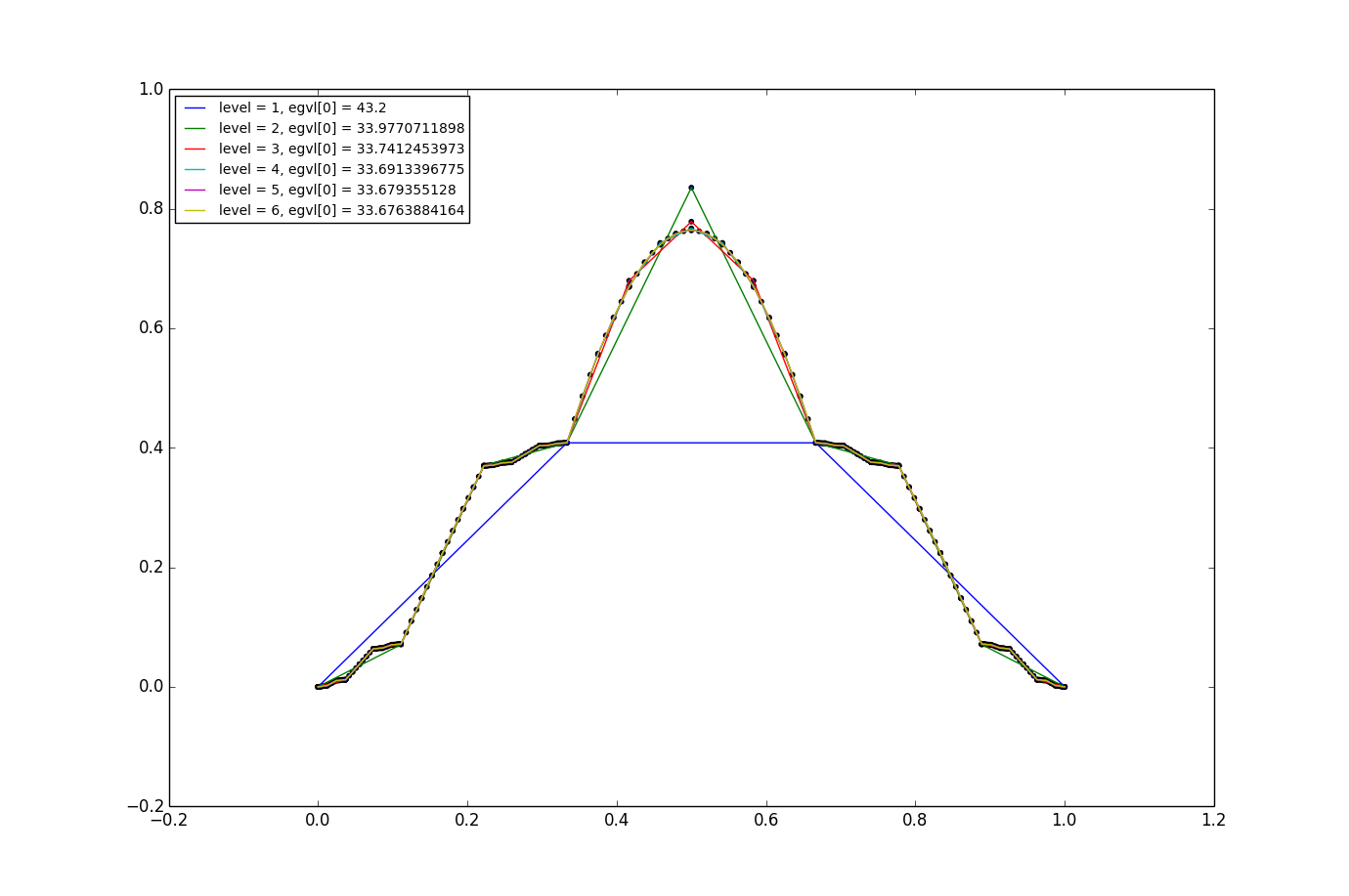}\par
\caption{{Restriction of the eigenfunction corresponding to the lowest eigenvalue to one edge, $r=a=\frac{1}{6}$, level $1-6$}}
\label{F:Efcts_discrete_Hanoi_1}
\end{figure}
\begin{figure}[H]
\includegraphics[scale=.38]{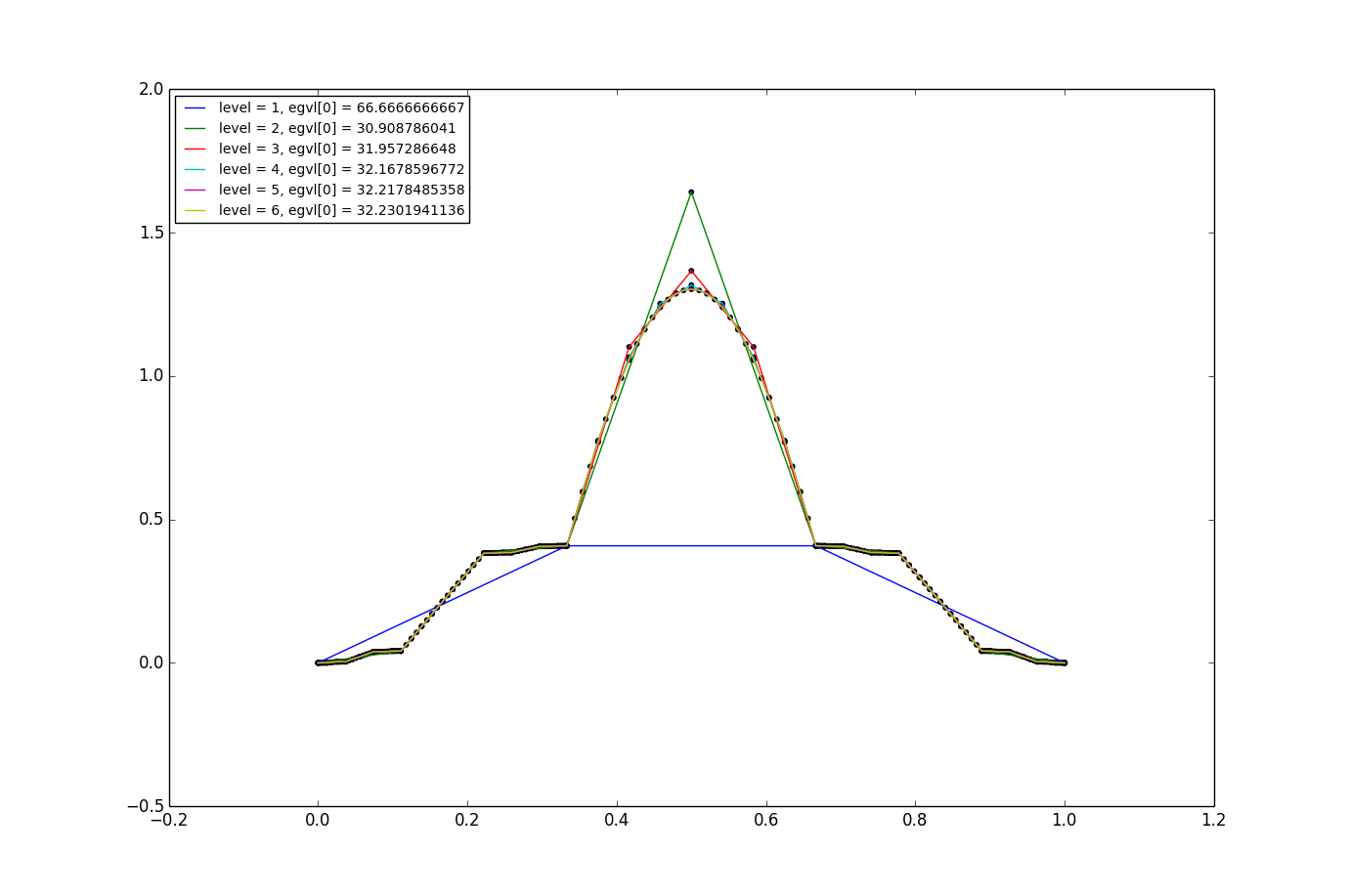}\par 
\caption{Restriction of the eigenfunction corresponding to the lowest eigenvalue to one edge, $r=a=\frac{1}{10}$, level $1-6$}
\label{F:Efcts_discrete_Hanoi_2}
\end{figure}

Although spectral decimation is not applicable in this case, we can still see that the eigenfunctions corresponding to the $k$th lowest eigenvalues of different levels share the same pattern. When the measure parameter $a$ tends to $\frac{1}{3}$ and the resistance parameter $r$ approaches $\frac{3}{5}$, eigenfunctions become more similar to eigenfunctions on $\SG$.

\begin{figure}[H]
\begin{subfigure}[b]{0.45\textwidth}
\includegraphics[width=\linewidth]{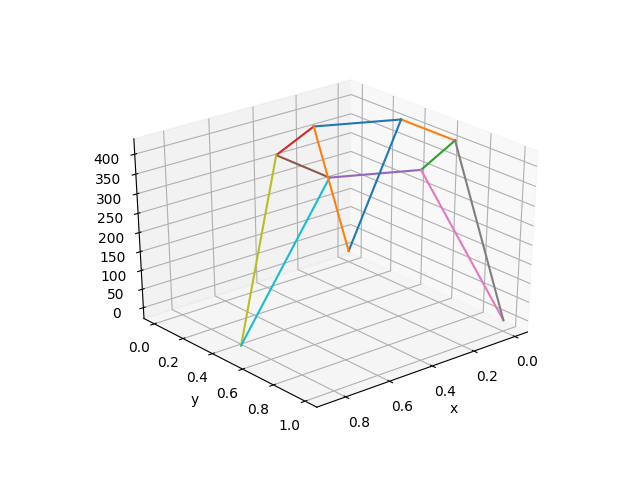}\par
\caption{{\small level=1, eigenvalue=43.2}}
\end{subfigure}
\begin{subfigure}[b]{0.45\textwidth}
\includegraphics[width=\linewidth]{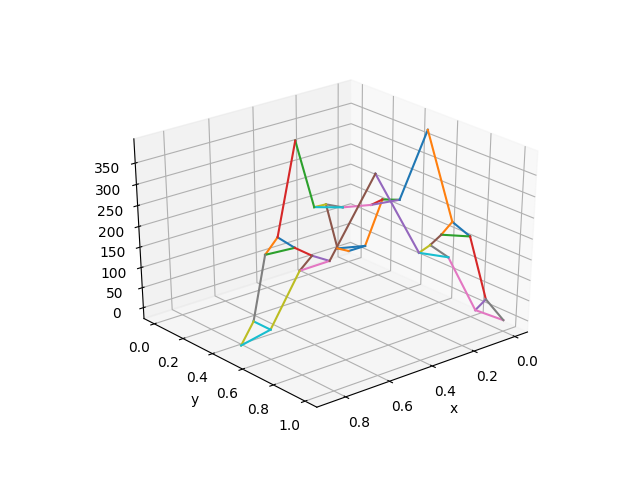}\par
\caption{{\small level=3, eigenvalue=33.977}}
\end{subfigure}
\par\bigskip
\begin{subfigure}[b]{0.45\textwidth}
\includegraphics[width=\linewidth]{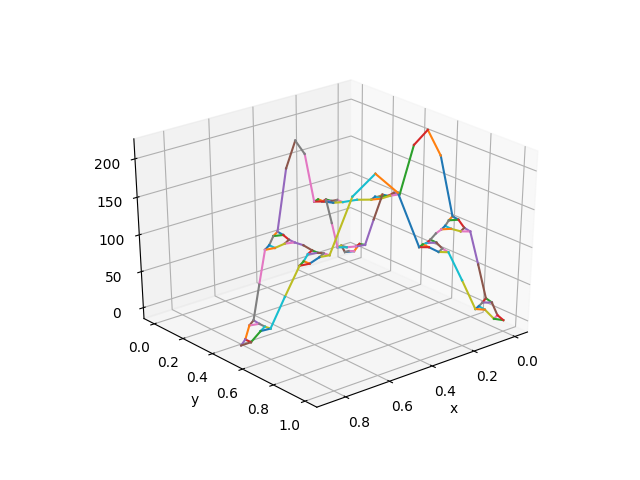}\par
\caption{{\small level=2, eigenvalue=33.741}}
\end{subfigure}
\begin{subfigure}[b]{0.45\textwidth}
\includegraphics[width=\linewidth]{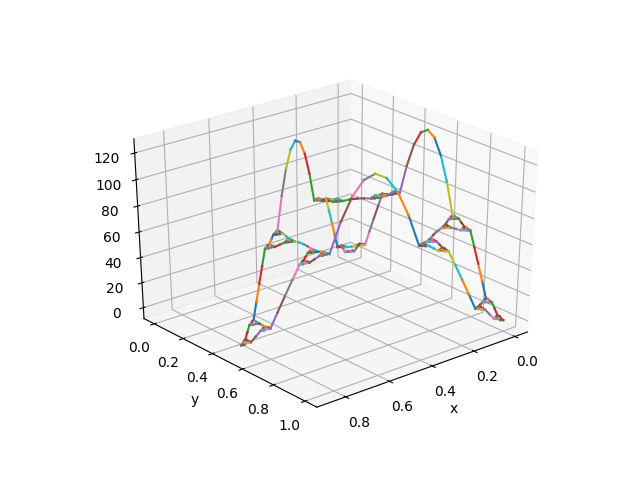}\par
\caption{{\small level=4, eigenvalue=33.691}}
\end{subfigure}
\par\bigskip
\begin{subfigure}[b]{0.45\textwidth}
\includegraphics[width=\linewidth]{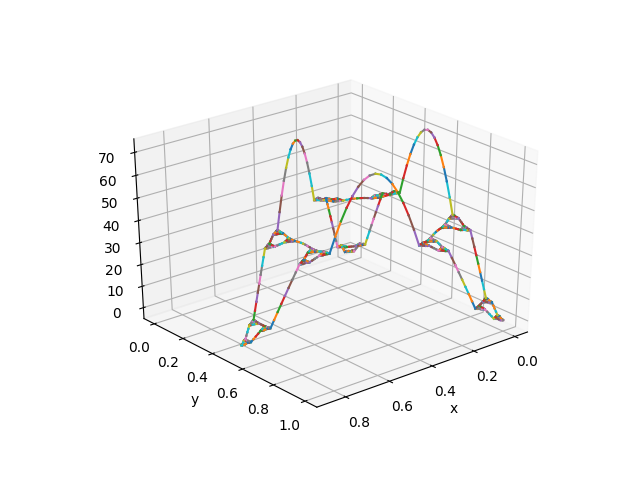}\par
\caption{{\small level=5, eigenvalue=33.679}}
\end{subfigure}
\begin{subfigure}[b]{0.45\textwidth}
\includegraphics[width=\linewidth]{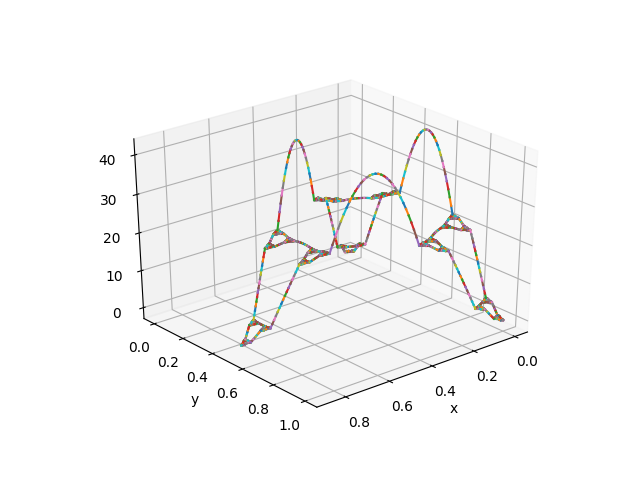}\par
\caption{{\small level=6, eigenvalue=33.676}}
\end{subfigure}
\caption{Eigenfunction corresponding to the lowest eigenvalue, $r=a=\frac{1}{6}$, level $1-6$, see~\cite{CGSZ17} for further examples.}
\label{F:Hanoi_discrete_EF}
\end{figure}


\begin{figure}[H]
\begin{subfigure}[b]{0.475\textwidth}
\includegraphics[width=\linewidth]{a_0_17_r_0_17_level_6_index_0.png}\par
\caption{{\small $a=\frac{1}{6}$, $r=\frac{1}{6}$, level=6, eigenvalue=33.676}}
\end{subfigure}
\begin{subfigure}[b]{0.475\textwidth}
\includegraphics[width=\linewidth]{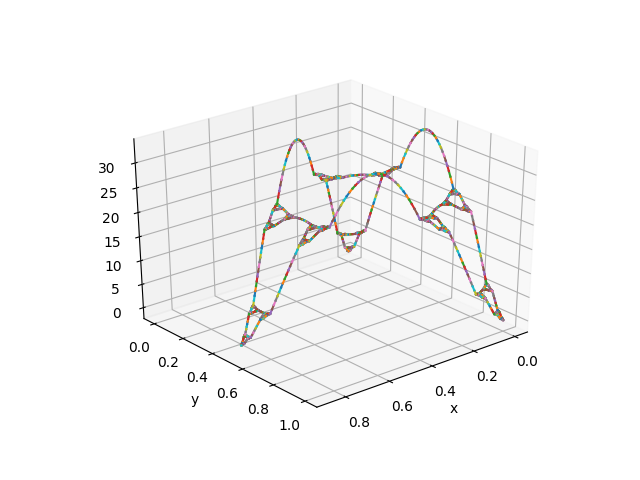}\par
\caption{{\small $a=\frac{1}{6}$, $r=\frac{1}{4}$, level=6, eigenvalue=26.142}}
\end{subfigure}
\par\bigskip
\begin{subfigure}[b]{0.475\textwidth}
\includegraphics[width=\linewidth]{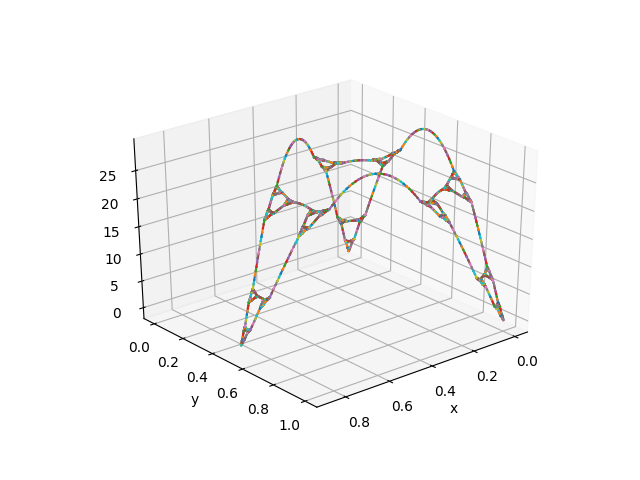}\par
\caption{{\small $a=\frac{1}{6}$, $r=\frac{1}{3}$, level=6, eigenvalue=21.245}}
\end{subfigure}
\begin{subfigure}[b]{0.475\textwidth}
\includegraphics[width=\linewidth]{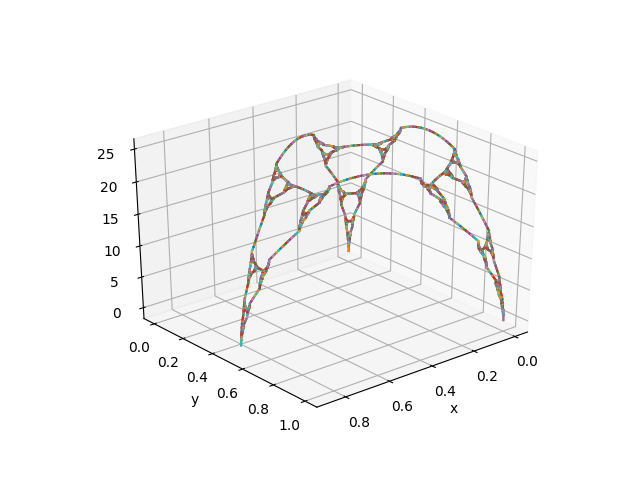}\par
\caption{{\small $a=\frac{1}{6}$, $r=\frac{1}{2}$, level=6, eigenvalue=15.395}}
\end{subfigure}
\caption{Eigenfunction corresponding to the lowest eigenvalue, increasing $r$.}
\end{figure}
\begin{figure}[H]
\begin{subfigure}[b]{0.475\textwidth}
\includegraphics[width=\linewidth]{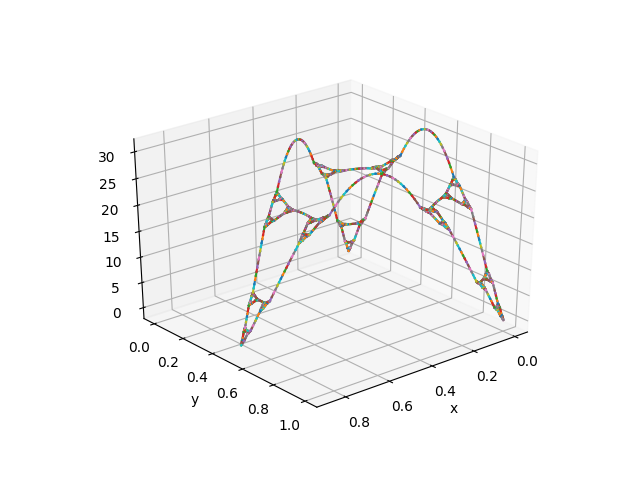}\par
\caption{{\small $a=\frac{1}{10}$, $r=\frac{1}{3}$, level=6, eigenvalue=18.181}}
\end{subfigure}
\begin{subfigure}[b]{0.475\textwidth}
\centering
\includegraphics[width=\linewidth]{a_0_17_r_0_33_level_6_index_0.png}\par
\caption{{\small $a=\frac{1}{6}$, $r=\frac{1}{3}$, level=6, eigenvalue=21.245}}
\end{subfigure}
\caption{Eigenfunction corresponding to the lowest eigenvalue, increasing $a$, see~\cite{CGSZ17} for further examples.}
\end{figure}
\begin{figure}[H]\ContinuedFloat
\par\bigskip
\begin{subfigure}[b]{0.475\textwidth}
\centering
\includegraphics[width=\linewidth]{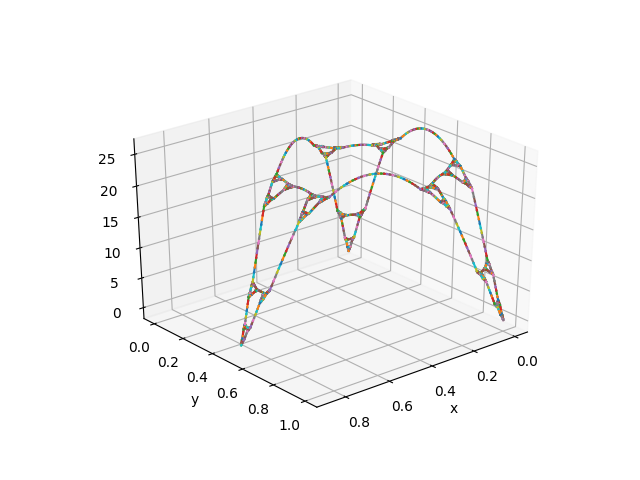}\par
\caption{{\small $a=\frac{1}{4}$, $r=\frac{1}{3}$, level=6, eigenvalue=25.746}}
\end{subfigure}
\begin{subfigure}[b]{0.475\textwidth}
\centering
\includegraphics[width=\linewidth]{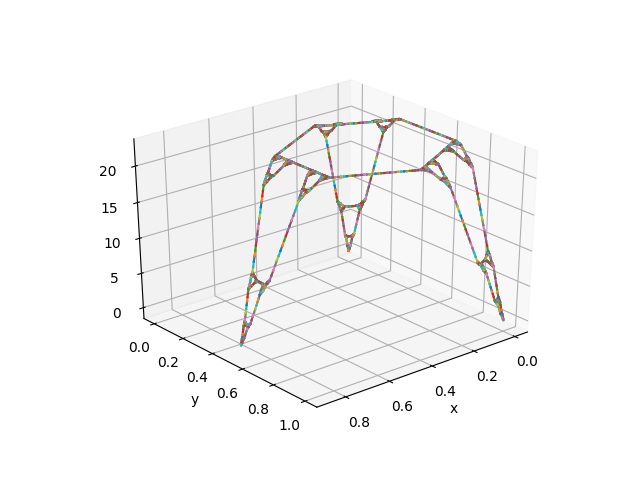}\par
\caption{{\small $a=\frac{333}{1000}$, $r=\frac{1}{3}$, level=6, eigenvalue=29.835}}
\end{subfigure}
\caption{(cont.) Eigenfunction corresponding to the lowest eigenvalue, increasing $r$, see~\cite{CGSZ17} for further examples.}
\end{figure}

\subsection{Quantum graph approach}
The Hanoi attractor $\rm H$ is also a \textit{fractal quantum graph}, a concept introduced in~\cite{ARKT16}, where the Laplacian $\Delta_\mu$ from~\eqref{E:DefDeltaMu} was approximated by quantum graphs. We refer to the appendix and~\cite{BK13} for basic notation and background. In this case, we consider approximating graphs as in Figure~\ref{F:HanoiApprox} and treat each edge as a one-dimensional interval with the standard one-dimensional Laplacian and suitable boundary conditions on each edge.

%

\subsubsection{Computational Method}
We will focus on the numerical computation of the Dirichlet spectrum. Since an eigenfunction $u$ of $-\Delta_\mu$ with eigenvalue $\lambda^2$ restricted to each edge should give a trigonometric function of frequency $\lambda$, we parametrize the restriction of $u$ on an edge $e$ of an approximating graph by
\begin{equation*}\label{E:u_e}
u(x)|_e = a_e \sin (\lambda x)+b_e \cos(\lambda x),\qquad x\in [0, L_e],
\end{equation*}
where $L_e$ is the length of the edge $e$, and set $b_e = 0$ for edges adjacent to the boundary $V_0$. 

\medskip

For any fixed $\lambda>0$ and each edge $e\in E_m$ with vertex $v\in V_m$, $u(v)|_e$ and $(u|_e)'(v)$ are linear combinations of $a_e$ and $b_e$. From the matching conditions, each vertex contributes with $\text{degree}(v)-1$ independent linear equations for $u(v)|_e$ whereas the boundary conditions provide a further equation for $(u|_e)'(v)$. This yields a homogeneous system of linear equations that can be expressed as 
\begin{equation}\label{E:LinEqLambda}
M(\lambda)(a_{e_1},a_{e_2},\ldots a_{e_{2|E_m|}},b_{e_1},b_{e_2}\ldots,b_{e_{2|E_m|}})^T=(0,0,\ldots)^T,
\end{equation}
where $|E_m|$ denotes the number of edges in $E_m$. Notice that $M(\lambda)$ is a square matrix because the number of equations in the linear system is  $\sum_{v\in V_m} \text{degree}(v)=2|E_m|$. 
Since a number $\lambda^2$ is an eigenvalue if and only if there exists a function $u$ whose parameters $a_e$ and $b_e$ solve~\eqref{E:LinEqLambda}, we search for solutions of the latter system of equations. In the case of Dirichlet eigenfunctions, one can parametrize the function $u$ at one of the two adjacent edges to a boundary vertex by sine curves, so that the number of equations at level $m$ reduces to $2|E|-3 = 3^{m+2}-9$.

\medskip
A technical issue in the computation arises 
from the fact that it is only possible to obtain the almost-nullspace decomposition of $M(\lambda)$.  In actual computations, writing $M(\lambda)= S\Sigma V^T$, 
it happens that the first few diagonal entries of $\Sigma$ are almost zero and we therefore choose the corresponding columns in $V$ as the basis for the almost-nullspace.

\subsubsection{Eigenvalue counting function}
At each level $m\geq 1$, the Laplacian associated with the approximating quantum graph, $\Delta_{Q_m}$, is a self-adjoint operator and we can again consider its eigenvalue counting function, defined analogously as~\eqref{E:Def_Ecf}. In particular we have the following asymptotic behavior of the eigenvalue counting function at any approximation level.

\begin{proposition}
Let $N_{Q_m}(x)$ denote the eigenvalue counting function of $\Delta_{Q_m}$. For each (finite) level of approximation $m\geq 1$, $N_{Q_m}(x)\sim O(\sqrt{x})$. 
In particular, if the effective resistance scaling factor $r$ is rational, $N_{Q_m}(\sqrt{x})$ is periodic.
\end{proposition}

The first part of the proposition follows from the classical properties of finite quantum graph, see e.g.~\cite[Chapter 3]{BK13}. 
To obtain the second, notice that each edge in $E_m$ is parametrized as an interval $[0,r^k]$ for some $1\leq k\leq m$, and 
the entries of the matrix $\operatorname{det}(M(\lambda))=F(\sin(r\lambda), \cos(r\lambda),\ldots, \sin(r^m\lambda), \cos(r^m\lambda))$ are trigonometric polynomials. 
At the $m$th approximation level, the matching (continuity) conditions and boundary conditions for the derivative give rise to equations of the form 
$a_1\sin(\lambda x_1)+a_2\cos(\lambda x_2)=a_3\sin(\lambda x_3)+a_4\cos(\lambda x_4)$, where $x_i\in \{0, r, ..., r^m\}$, whose behavior is periodic if the coefficients are rational. 
This result is is supported when performing numerical root-search methods, see Figure~\ref{F:HanoiECF}.



\begin{figure}[H]
\includegraphics[width = \textwidth]{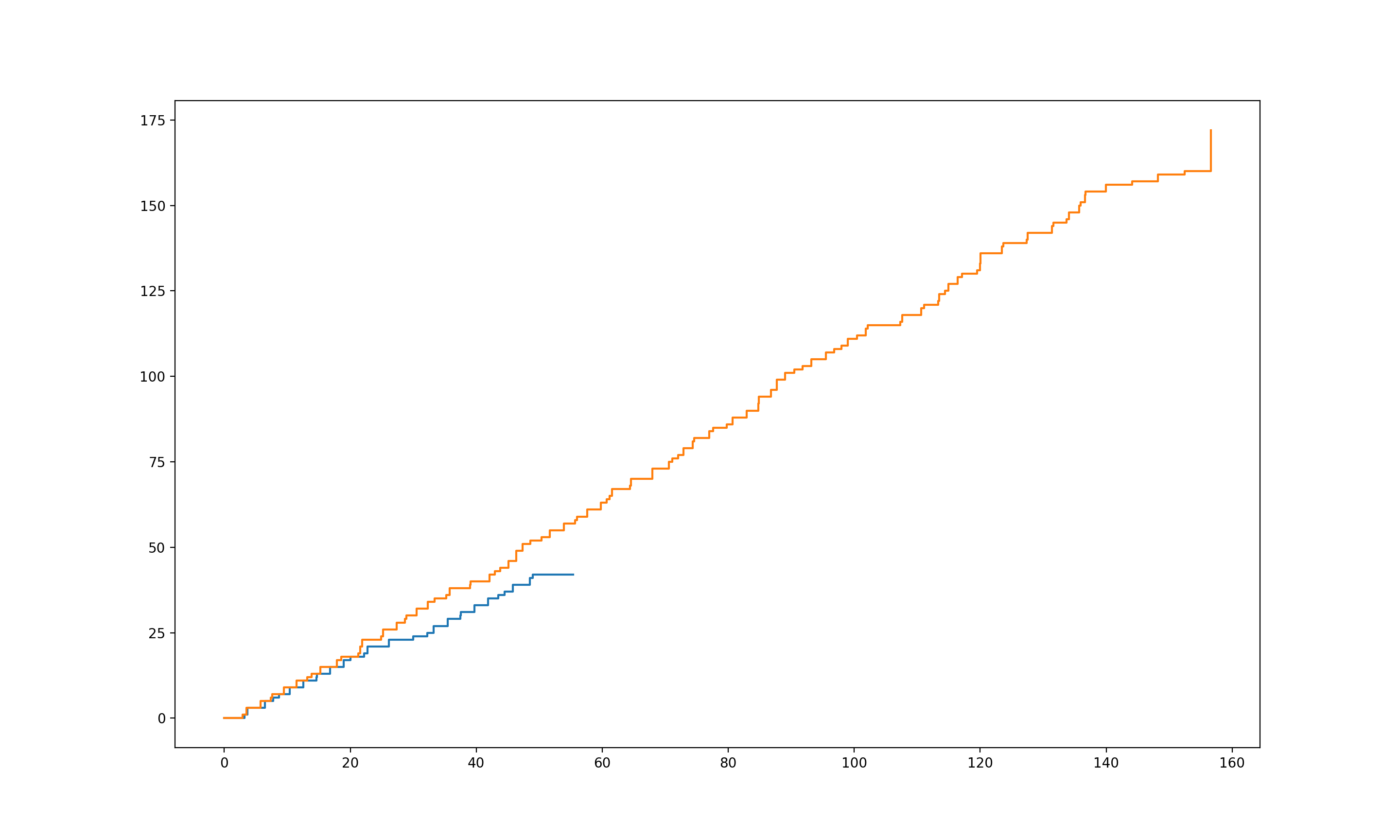}
\caption{Plot of $N_{Q_m}(\sqrt{x})$ for level 1 (blue) and level 2 (yellow) quantum hybrid SG-based graph with scaling factor $r=\frac{1}{6}$.}
\label{F:HanoiECF}
\end{figure}

\subsection{Spectrum of Hanoi attractor. Comparison of approaches}
We conclude this section with a comparison between results obtained following the discrete graph and the quantum graph approach. Figure~\ref{F:Hanoi_QG_EF} provides several eigenfunctions corresponding to the lowest eigenvalues of $\Delta_{Q_m}$ at different approximation levels. So far, only a qualitative comparison with the corresponding ones from Figure~\ref{F:Hanoi_discrete_EF} is possible. 

\begin{figure}[H]
\begin{subfigure}[b]{0.475\textwidth}
\includegraphics[width=\linewidth]{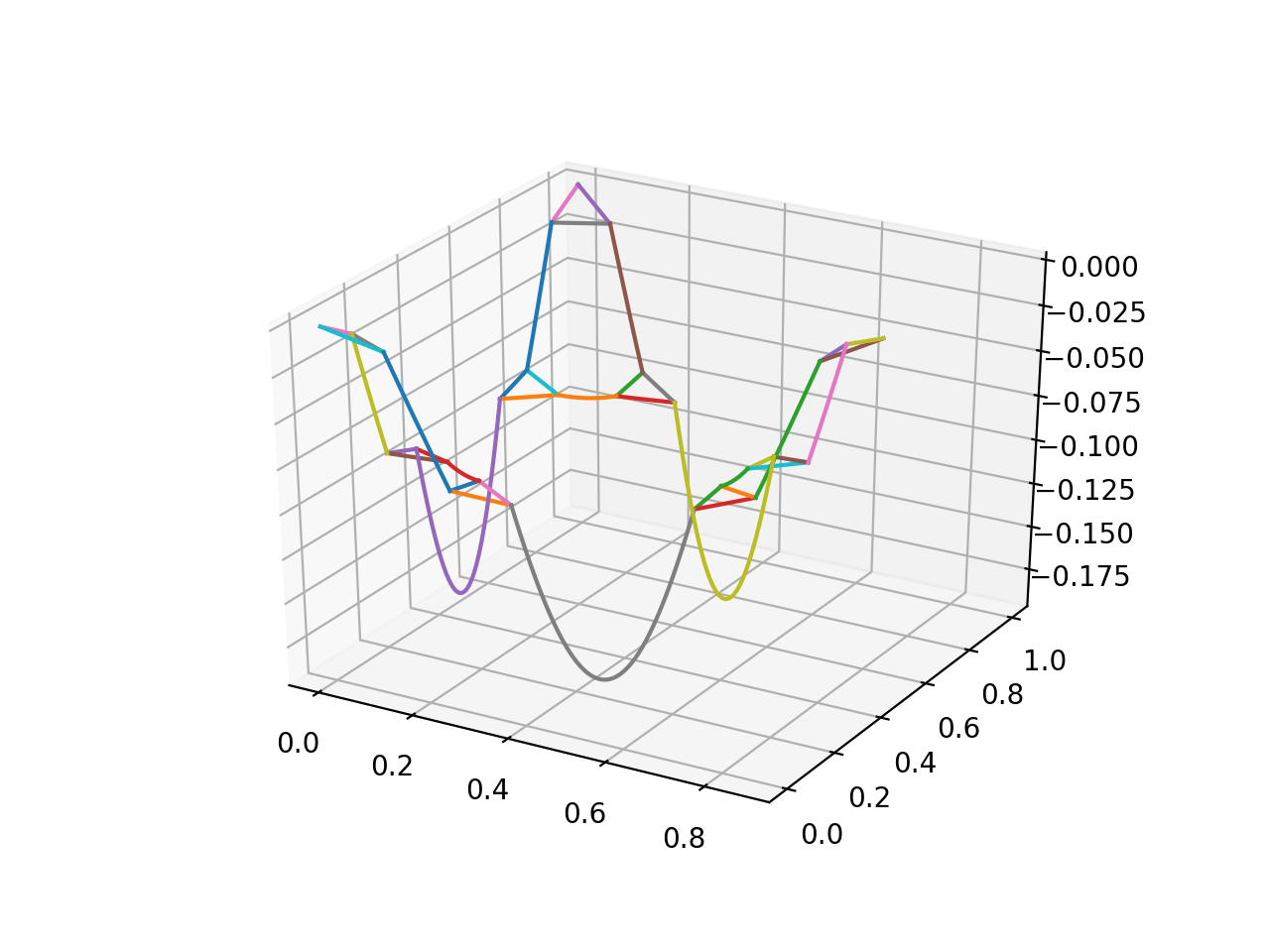}\par
\caption{{\small level=1, eigenvalue=8.58}}
\end{subfigure}
\begin{subfigure}[b]{0.475\textwidth}
\includegraphics[width=\linewidth]{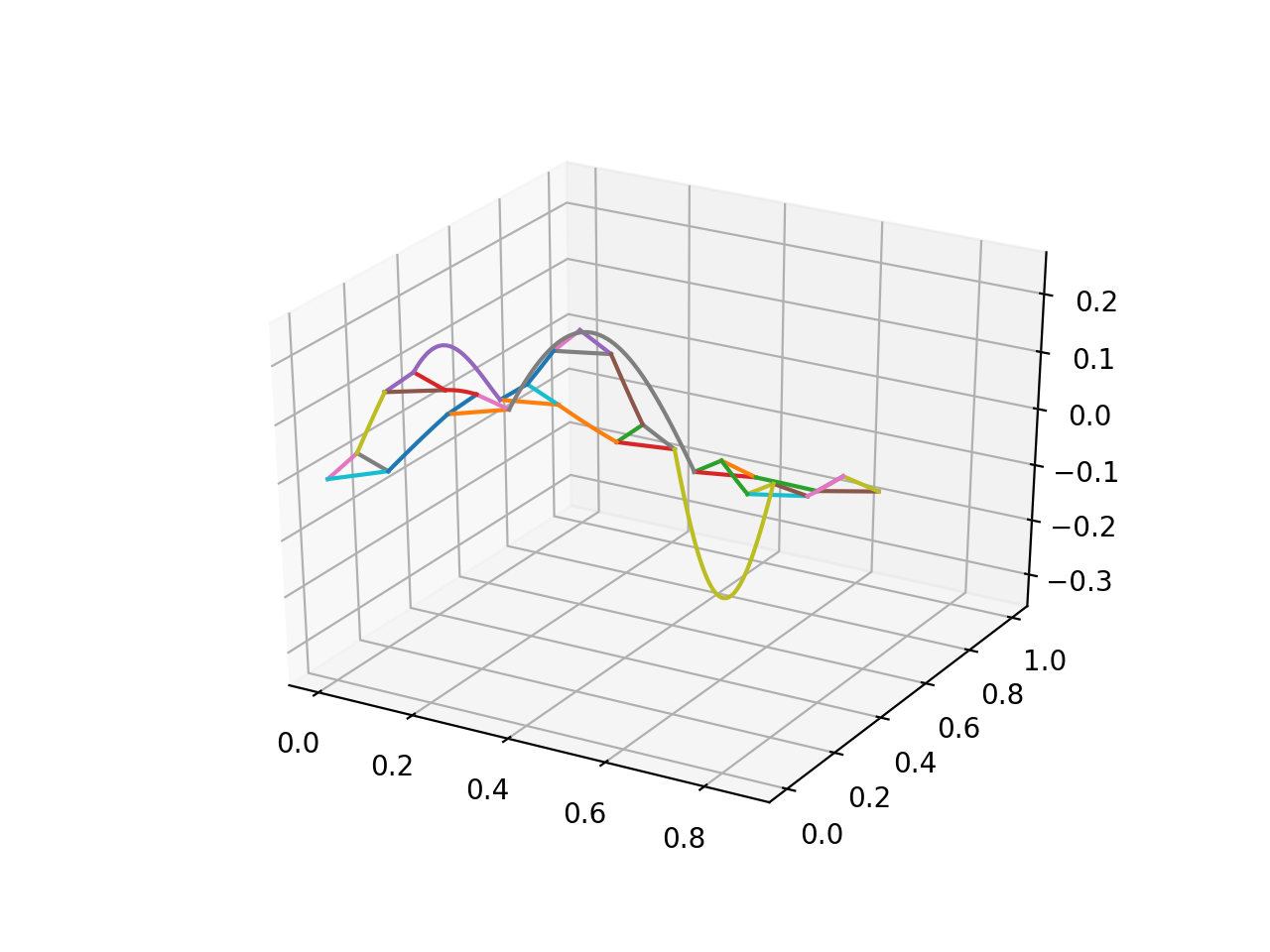}\par
\caption{{\small level=1, eigenvalue=12.27}}
\end{subfigure}
\par\bigskip
\begin{subfigure}[b]{0.475\textwidth}
\includegraphics[width=\linewidth]{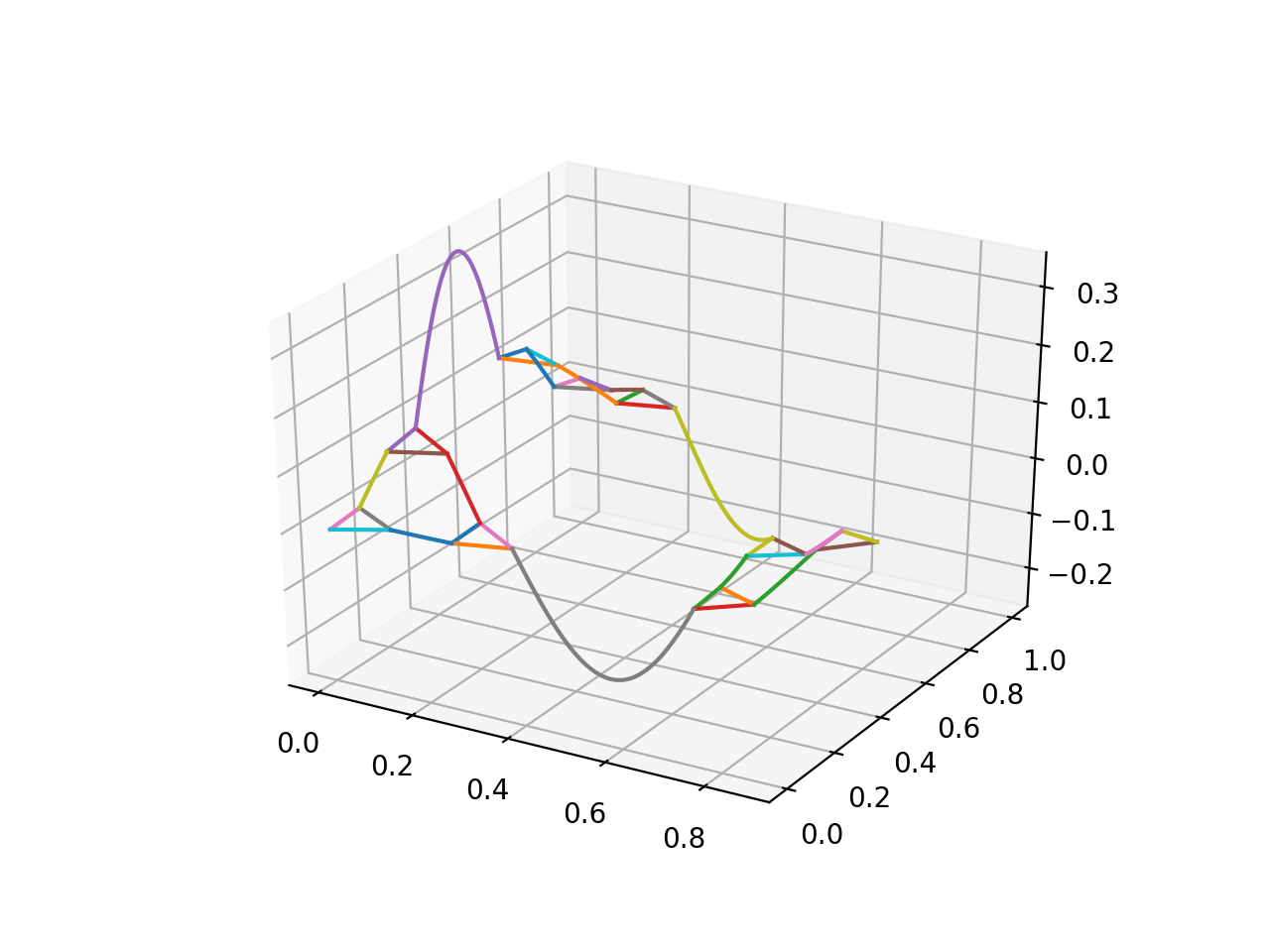}\par
\caption{{\small level=1, eigenvalue=12.27}}
\end{subfigure}
\begin{subfigure}[b]{0.475\textwidth}
\includegraphics[width=\linewidth]{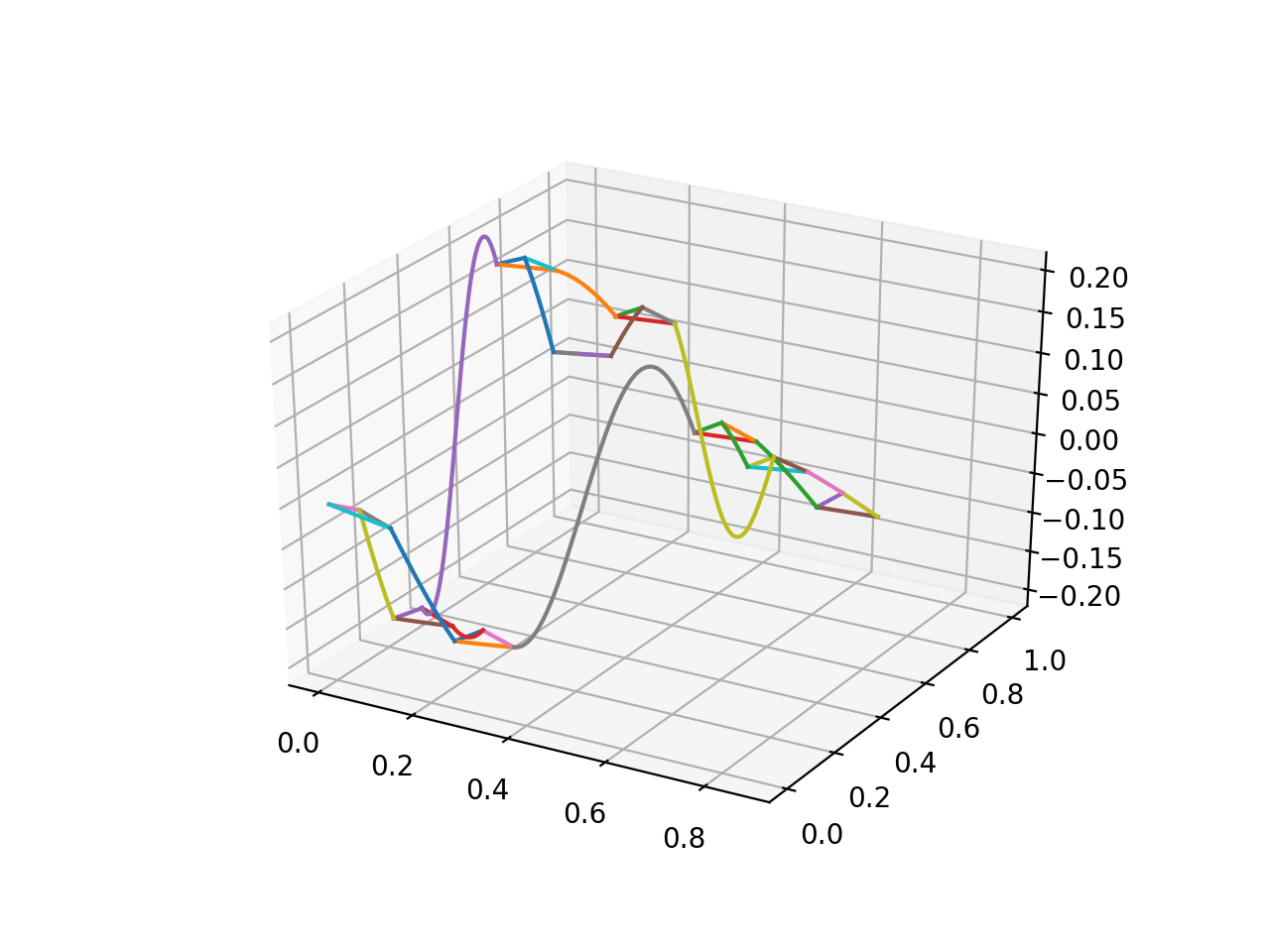}\par
\caption{{\small level=1, eigenvalue=32.95}}
\end{subfigure}
\par\bigskip
\begin{subfigure}[b]{0.475\textwidth}
\includegraphics[width=\linewidth]{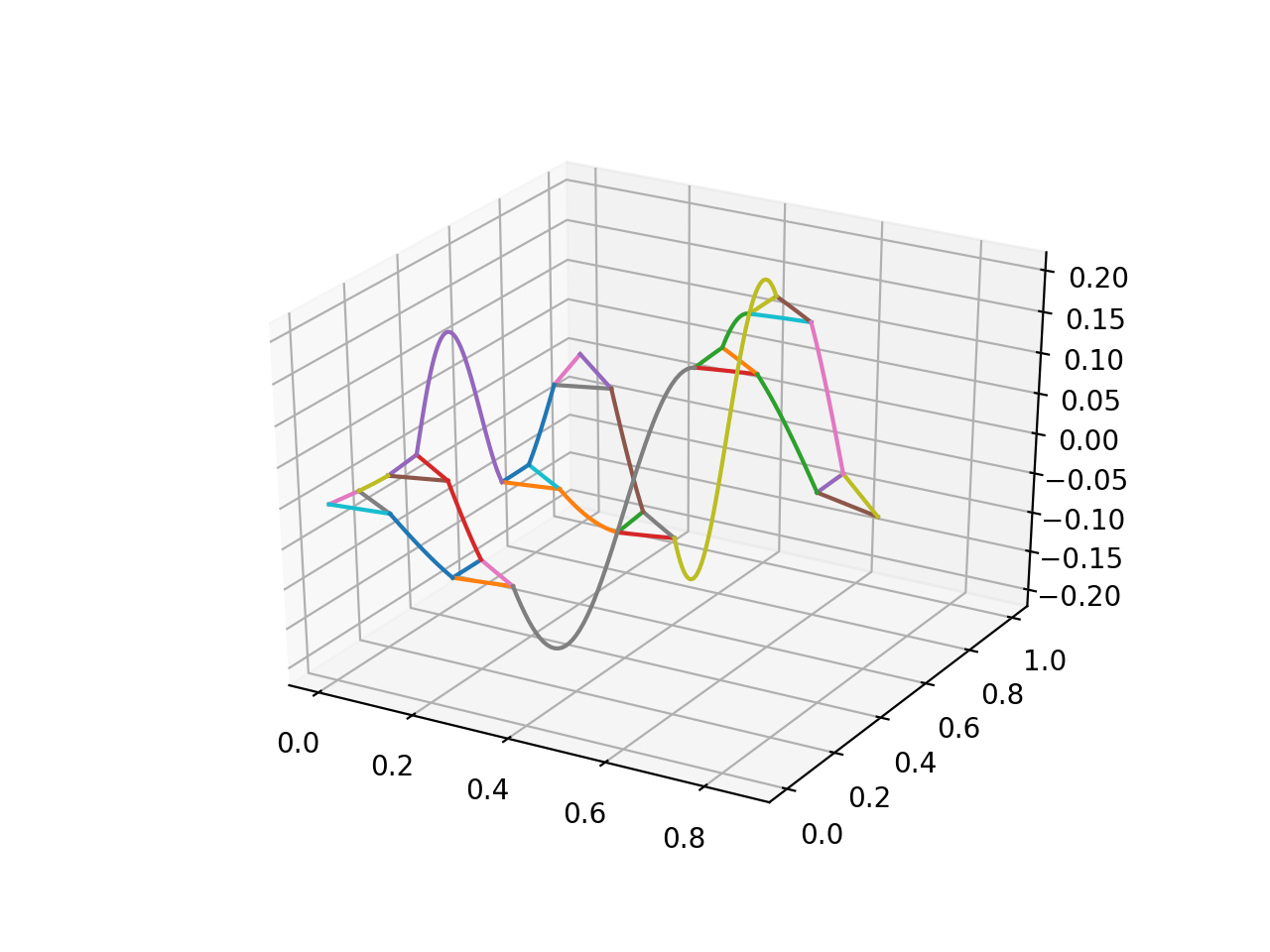}\par
\caption{{\small level=1, eigenvalue=32.95}}
\end{subfigure}
\begin{subfigure}[b]{0.475\textwidth}
\includegraphics[width=\linewidth]{q574.png}\par
\caption{{\small level=1, eigenvalue=32.95}}
\end{subfigure}
\caption{Eigenfunctions for level 0 and level 1 quantum graph approximation of $\rm H$, see~\cite{CGSZ17} for further examples.}
\label{F:Hanoi_QG_EF}
\end{figure}

In addition, the subsequent table summarizes the numerical results obtained for the bottom of the spectrum via quantum graph approximation and via discrete graph approximation for several approximation levels. 
Here, when we compare eigenvalues obtained from the two different approaches, we multiple each eigenvalue from the quantum graph approximation by a renormalization factor of $\frac{13}{3}$. This is due to the relation~\eqref{E:RelationLaplacian} between the usual second derivative on an interval and the Laplacian defined on the Hanoi attractor.
\begin{table}[H]
\centering
\renewcommand\arraystretch{1.1}
\renewcommand\tabcolsep{3pt}
\begin{tabular}{|c|c|c|c|c|c|c|c|}
\hline
\multicolumn{2}{|c|}{Level 0(Q)} & \multicolumn{2}{|c|}{Level 1(Q)} &\multicolumn{2}{|c|}{Level 2(Q)} & \multicolumn{2}{|c|}{Hanoi Attractor}\\ \hline
\small Ev. & {\small Renorm.\ ev.}&\small Ev. & {\small Renorm.\ ev.}&\small Ev. & {\small Renorm.\ ev.} & \small Ev. & \small Mult.\\ \hline
10.247   & 44.402 & 8.578 & 37.173  & 7.896  & 34.216  & \textbf{33.676}  & 1\\ \hline
13.627   & 59.051 & 12.266  & 53.153   & 11.424 & 49.506  &\textbf{49.906} & 2\\ \hline
41.306   & 178.992 & 32.951  & 142.786  & 30.030 & 130.132 & \textbf{122.399}  & 2  \\ \hline
59.750   & 258.918 & 54.613  & 236.657  & 51.955  & 225.139  &\textbf{208.308} & 1  \\ \hline 
75.686   & 327.975 & 57.438  & 248.897  & 52.592  & 227.897 &\textbf{220.949}  & 1  \\ \hline
107.259  & 464.788 & 89.685  & 388.635  & 83.999   & 363.995 &\textbf{345.050}  & 2  \\ \hline
156.406  & 677.761 & 132.033 & 572.143  & 122.324  & 530.069 & \textbf{502.813}  & 2  \\ \hline 
213.693  & 926.002 & 172.604 & 747.953  & 156.876  & 679.794 & \textbf{625.255}  & 1  \\ \hline
217.180  & 941.113 & 192.661 & 834.863  & 186.323  & 807.398 & \textbf{790.386}  & 1  \\ \hline 
280.562  & 1215.767 & 232.571 & 1007.807  & 218.448  & 946.610 & \textbf{895.853} & 2  \\ \hline 
358.903  & 1555.247 & 320.370 & 1388.268  & & & \textbf{1320.110} & 2 \\ \hline
400.372  & 1734.945 & 343.876 & 1490.131  & & & \textbf{1392.494} & 1  \\ \hline 
\end{tabular}
\caption{Bottom of spectrum for quantum graph compared to the spectrum of the discrete level 6 graph approximation of the Hanoi attractor.\\ Ev.= eigenvalue, Renorm. ev. = renormalized eigenvalue, Mult. = multiplicity.}
\end{table}
\begin{section}{Spectrum of the hybrid with base ${\rm SG}3$}\label{sec:SpcSG3}
In the present section, we carry out a similar spectral analysis on the hybrid fractal with base ${\rm SG}_3$ introduced in Definition~\ref{def:SG3.01} and denoted by ${\rm H}$. Again and for simplicity we consider a resistance form as in Section~\ref{section:SG3} whose resistance parameters are $R=1$, $r_{{\rm H}}=r$ and $r_{{\rm I}}=r_{{\rm SG}}=\rho$, c.f.\ Definition~\ref{def:SG3.02}. In particular, we know from Lemma~\ref{lemma:SG3.01} that the corresponding energy is graph-directed self-similar.

\subsection{Discrete graph approach}\label{sec:SpcSG3.01}
Following Subsection~\ref{section:SG3.1}, we consider the approximating graphs $\Gamma_m=(V_m,E_m,r_m)$ and define $V_*=\cup_{m\geq 0}V_m$.

\medskip

In view of Theorem~\ref{thm:SG3.01}, an associated resistance form $(\mathcal E,\mathcal F)$ exists given that the renormalization equation
\begin{equation*}
5\rho^2+5r^2+\frac{31}{3}\rho r-3\rho-\frac{7}{3}r=0
\end{equation*}
holds, see Figure~\ref{fig:res}.

\begin{figure}[H]
\centering
\begin{subfigure}{.45\textwidth}
\centering
\includegraphics[trim={0 0 0 -3em}, width=.8\linewidth]{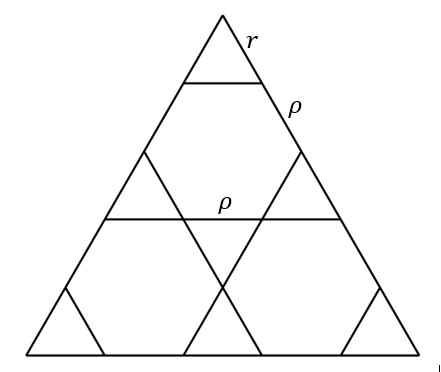}
\caption{Level 1 resistance}
\end{subfigure}%
\begin{subfigure}{.45\textwidth}
\centering
\includegraphics[trim={0 0 0 1em},clip,width=0.8\linewidth]{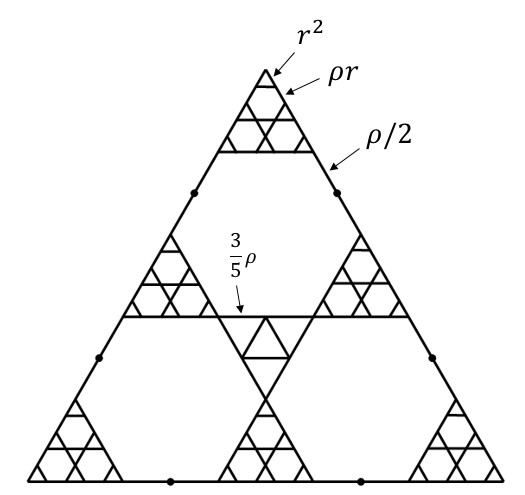}
\caption{Level 2 resistance}
\end{subfigure}
\caption{Contraction property of the resistance.}
\label{fig:res}
\end{figure}

Further, we equip ${\rm H}$ with a weakly self-similar measure measure with parameters $a$, $b$, $c$ as described in Figure~\ref{fig:ms}, so that the measure parameters satisfy
\begin{equation*}
6a+6b+c=1.
\end{equation*}

\begin{figure}[H]
\centering
\begin{subfigure}{.45\textwidth}
\centering
\includegraphics[trim={0 0 0 -.5em}, width=.8\linewidth]{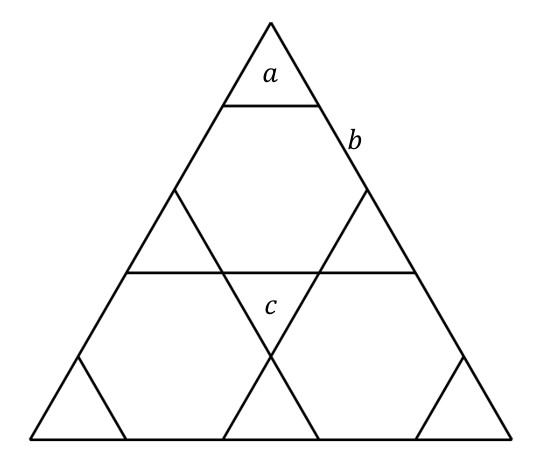}
\caption{Level 1 measure}
\end{subfigure}%
\begin{subfigure}{.45\textwidth}
\centering
\includegraphics[trim={0 0 0 .75em},clip,width=.8\linewidth]{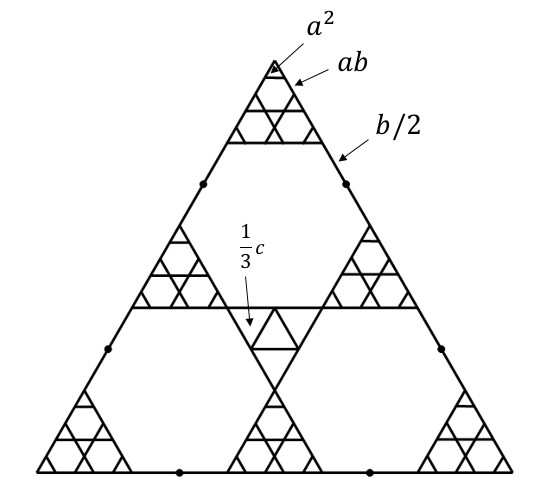}
\caption{Level 2 measure}
\end{subfigure}
\caption{Contraction property of measure}
\label{fig:ms}
\end{figure}



Our aim is to study the Laplace operator $\Delta_\mu$ defined through the weak formulation~\ref{E:DefDeltaMu} by means of the corresponding discrete Laplacian $\Delta_m$ from Definition~\eqref{E:DefLaplacian_m}. In this case, any point $x\in V_m \setminus V_0$, $x$ can be either an endpoint or an internal point of a unique line segment, first born in $V_k$ for some $k\leq m$, denoted by $I^{(m,k)}_{x}$, or a boundary point or an internal point of a unique reversed triangle, first born in $V_k$ for some $k\leq m$, denoted by $J^{(m,k)}_{x}$.
It can be shown that
\begin{equation*}
\int_K\psi_{x}^{(m)}=\begin{cases}
\frac{1}{3}a^{m}+(\frac{1}{2})^{m-k+1}a^{k-1}b&\text{if $x$ is an endpoint of } I^{(m,k)}_{x},\\
(\frac{1}{2})^{m-k}a^{k-1}b&\text{if $x$ is an internal point of } I^{(m,k)}_{x},\\
\frac{1}{3}a^{m}+(\frac{1}{3})^{m-k+1}a^{k-1}c&\text{if $x$ is a boundary point of $J^{(m,k)}_{x}$},\\
2(\frac{1}{3})^{m-k+1}a^{k-1}c&\text{if $x$ is an interior point of $J^{(m,k)}_{x}$}  
\end{cases}
\end{equation*}

One more observation can be made is that if $x$ is an internal point of $I^{(m,k)}_{x}$, then
\begin{equation*}
\Delta_m u(x)=\frac{u(y_0)+u(y_1)-2u(x)}{\left(a^{k-1}b(\frac{1}{2})^{m-k}\right)\left((\frac{1}{2})^{m-k}r^{k-1}\rho\right)}
\end{equation*}
where $y_1\,\ y_2$ are 2 adjacent points of $x$ in $I^{(m,k)}_{x}$.
Compared to the usual second derivative on the interval, we will get
\begin{equation*}
u^{\prime\prime}(x)=\lim_{m\to\infty}\frac{u(y_0)+u(y_1)-2u(x)}{((\frac{1}{2})^{m-k}r^{k-1}\rho)^2}=\Big(\frac{r}{a}\Big)^{k-1}\frac{b}{\rho}\Delta_{\mu}u(x)
\end{equation*}
Again, solving $-\Delta_{\mu}u=\lambda u$ on ${\rm H}$ will yield trigonometric functions on each interval.


\subsection{Eigenfunctions}
Following the same numerical computation method from the previous section, we can compute the eigenvalues and eigenfunctions of $\Delta_m$. Figure~\ref{fig:SG3egf} displays, on the left, the eigenfunction corresponding to the lowest (Dirichlet) eigenvalue on level 4, with parameters $a=r=\frac{1}{12}, b=\frac{1}{13}$. On the right side, the same function is plotted now restricted to the middle reversed triangle. Notice that the graph of the function resembles the eigenfunction corresponding to the lowest eigenvalue on the ordinary Sierpinski Gasket.

\begin{figure}[H]
\centering
\begin{subfigure}{.5\textwidth}
\centering  
\includegraphics[trim={2cm 1cm 2cm 2.2cm},clip,width=\linewidth]{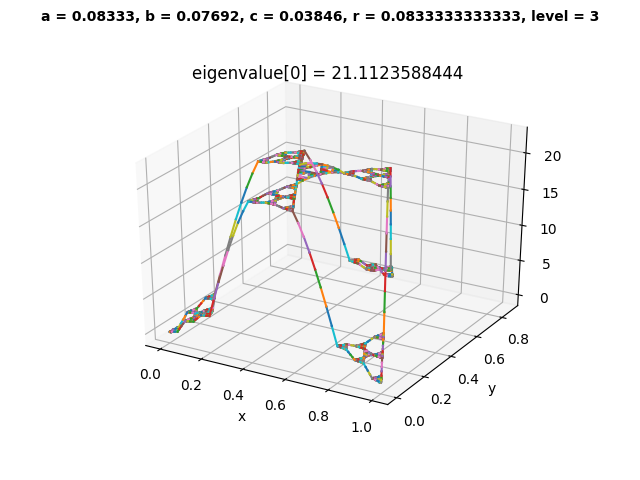}
\end{subfigure}%
\begin{subfigure}{.5\textwidth}
\centering
\includegraphics[trim={2cm 1cm 2cm 2.2cm},clip,width=\linewidth]{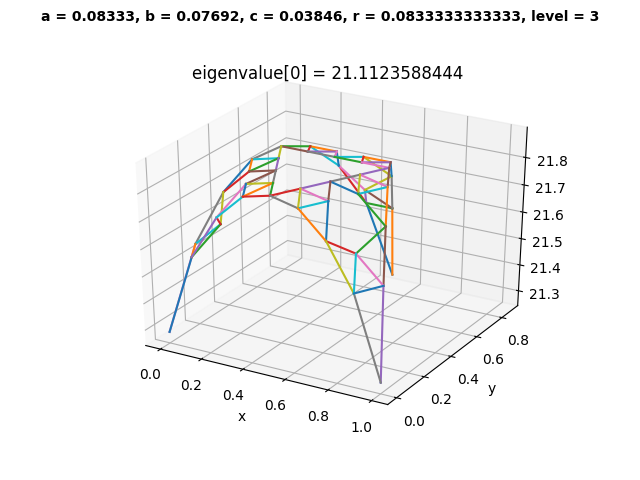}
\end{subfigure}
\caption{Level 4 eigenfunction, see~\cite{CGSZ17} for further examples.}
\label{fig:SG3egf}
\end{figure}

\subsection{Spectrum and eigenvalue counting function}
In this section, we plot the eigenvalue counting function and the corresponding log-log plot with respect to different choices of parameters $a, b$ and $r$. We refer to the reader to the website~\cite{CGSZ17} to generate more data.

\begin{figure}[H]
\centering
\begin{subfigure}{.4\textwidth}
\centering
\includegraphics[trim={0.5cm 0cm 1.58cm 1cm},clip,width=.9\linewidth]{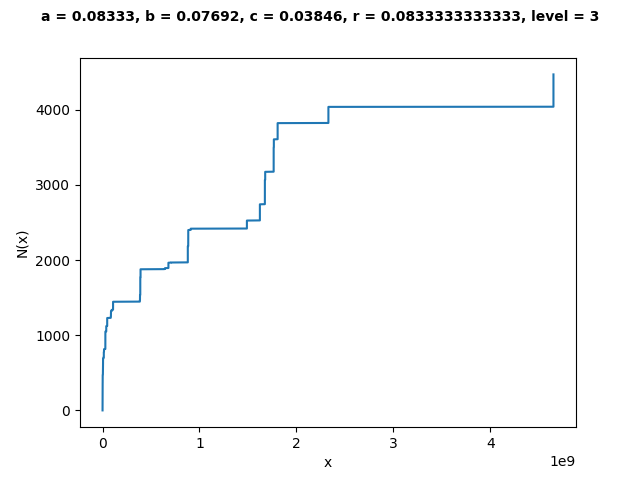}
\caption{{\small $a=r=\frac{1}{12},b=\frac{1}{13}$, level 3.}}
\end{subfigure}
\begin{subfigure}{.4\textwidth}
\centering
\includegraphics[trim={0.8cm 0cm 1.6cm 0.8cm},clip,width=.9\linewidth]{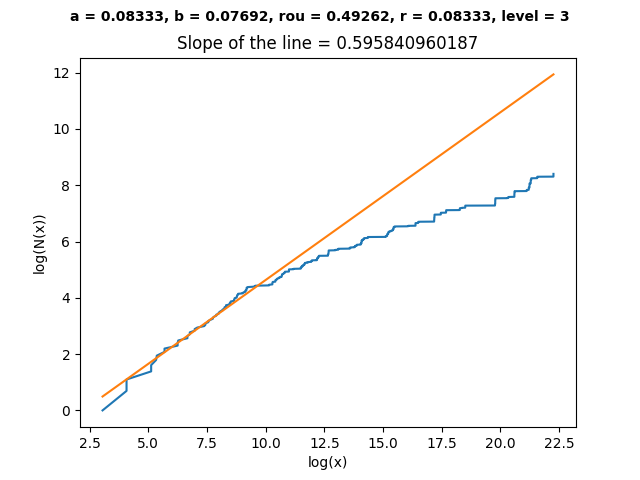}
\caption{$N(x) \sim x^{0.595841}$}
\end{subfigure}
\par\bigskip
\begin{subfigure}{.4\textwidth}
\centering
\includegraphics[trim={0.5cm 0cm 1.58cm 1cm},clip,width=0.9\linewidth]{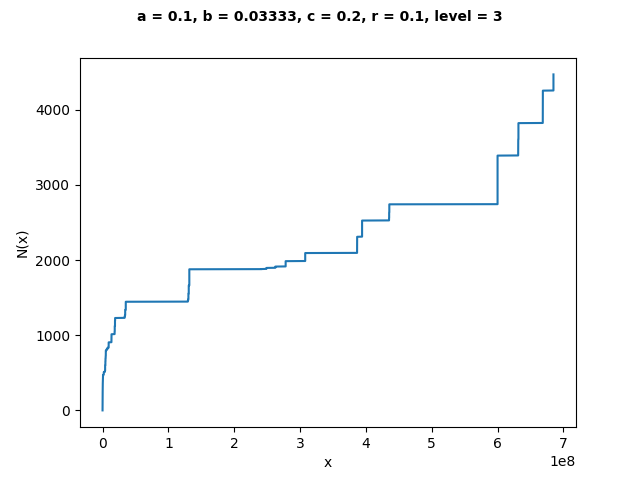}
\caption{{\small $a=r=\frac{1}{10},b=\frac{1}{30}$, level 3.}}
\end{subfigure}
\begin{subfigure}{.4\textwidth}
\centering
\includegraphics[trim={0.8cm 0cm 1.6cm 0.8cm},clip,width=0.9\linewidth]{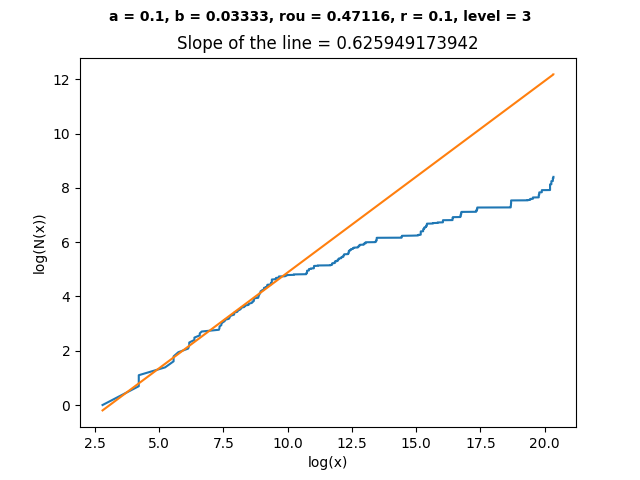}
\caption{$N(x) \sim x^{0.625949}$}
\end{subfigure}
\par\bigskip
\begin{subfigure}{.4\textwidth}
\centering
\includegraphics[trim={0.5cm 0cm 1.58cm 1cm},clip,width=.9\linewidth]{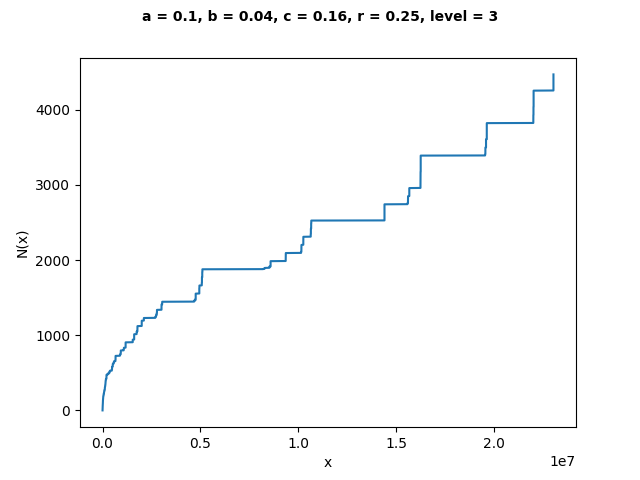}
\caption{{\small $a=\frac{1}{10},b=\frac{1}{25}, r=\frac{1}{4}$, level 3.}}
\end{subfigure}
\begin{subfigure}{.4\textwidth}
\centering
\includegraphics[trim={0.8cm 0cm 1.6cm 0.8cm},clip,width=.9\linewidth]{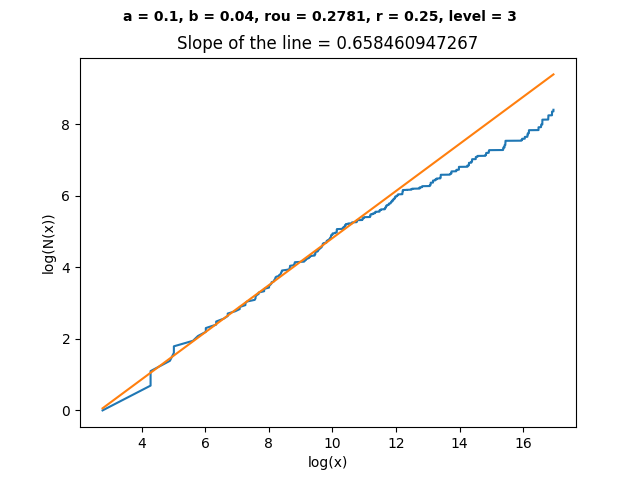}
\caption{$N(x) \sim x^{0.658461}$}
\end{subfigure}
\par\bigskip
\begin{subfigure}{.4\textwidth}
\centering
\includegraphics[trim={0.5cm 0cm 1.58cm 1cm},clip,width=.9\linewidth]{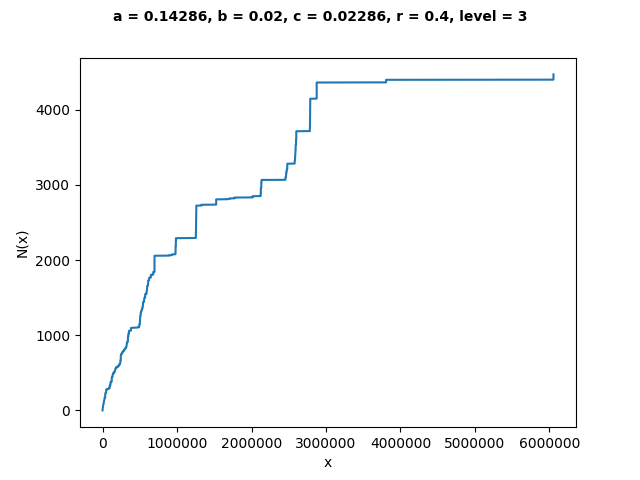}
\caption{{\small$a=\frac{1}{7},b=\frac{1}{50}，r=\frac{2}{5}$, level 3.}}
\end{subfigure}
\begin{subfigure}{.4\textwidth}
\centering
\includegraphics[trim={0.8cm 0cm 1.6cm 0.8cm},clip,width=.9\linewidth]{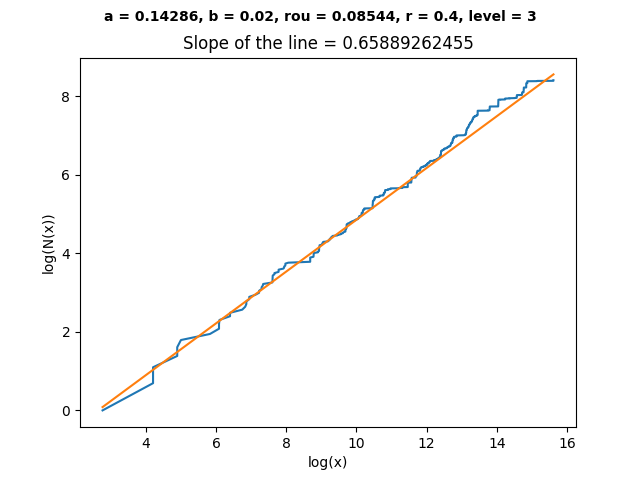}
\caption{$N(x) \sim x^{0.658893}$}
\end{subfigure}
\end{figure}
\end{section}

\subsection{Spectral asymptotics}
In this paragraph, we investigate the asymptotic behavior of the eigenvalue counting function of the Laplacian $\Delta_\mu$ and a related counting function for eigenvalues whose eigenfunctions are supported on the inverted $\SG$s. As mentioned at the beginning of the present section, $\Delta_\mu$ has an associated energy form that is graph-directed self-similar, c.f.~\ref{lemma:SG3.01}. Viewing the hybrid fractal $\rm{H}$ as the graph-directed fractal depicted in Figure~\ref{fig:SpcSG3.01} will allow us to apply the results in~\cite{HN03} in order to provide its spectral asymptotics.

\medskip

Following the notation in~\cite{HN03} and Figure~\ref{fig:HF02}, the directed graph $(\mathcal{S},E)$ that corresponds to ${\rm H}$ has vertices $\mathcal{S}=\{J_1,J_2,J_3\}$, where $J_1=\blacktriangle$ $J_2=-$, $J_3=\blacktriangledown$ and $18$ edges, of which $6$ are loops in $J_1$, $2$ loops in $J_2$ and $3$ loops in $J_3$, see Figure~\ref{fig:SpcSG3.01}. The graph $(\mathcal{S},E)$ is not connected and each vertex is a strongly connected component in the sense of~\cite[Section 3]{HN03}. There are no other strongly connected components.

\begin{figure}[H]
\centering
\begin{tikzpicture}
\draw ($(0:0)$) node (J1) {$J_1=\blacktriangle$};
\draw ($(45:3)$) node (J2) {$J_2=-$};
\draw ($(0:4)$) node (J3) {$J_3=\blacktriangledown$};
\draw[->] (J1.90) -- (J2) node[midway, left] {$6$ of this};
\draw[->] (J1) -- (J3) node[midway, above] {$1$ of this};
\path (J1) edge [loop left] node {$6$ of this} (J1);
\path (J2) edge [loop right] node {$2$ of this} (J2);
\path (J3) edge [loop right] node {$3$ of this} (J3);
\end{tikzpicture}
\caption{Directed graph associated with the hybrid $\rm H$ with base $\SG_3$ from Section~\ref{section:SG3}.}
\label{fig:SpcSG3.01}
\end{figure}
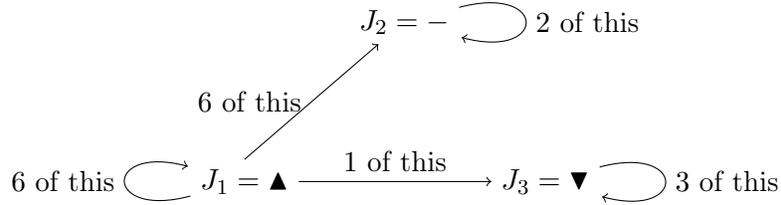

Setting $E_{ij}:=\{e\text{ edge from }J_i\text{ to }J_j\}$ and $E_i:=\{e\in E_{ij}, J_j\in \mathcal{S}\}$, the resistance parameter $r_e$ of an edge $e\in E$ is given by the corresponding resistance scaling factor described in Figure~\ref{fig:SG3.03}, i.e.
\begin{equation}\label{E:SpcSG3.01}
r_e=\begin{cases}
r&\text{if }e\in E_{11},\\
\rho&\text{if }e\in E_{12}\cup E_{13},\\
1/2&\text{if }e\in E_2,\\
3/5&\text{if }e\in E_3.
\end{cases}
\end{equation}
The weakly self-simliar measure introduced in Subsection~\ref{sec:SpcSG3.01} provides the measure parameter $\mu_e'$ of each edge $e\in E$,
\begin{equation}\label{E:SpcSG3.02}
\mu'_e=\begin{cases}
a&\text{if }e\in E_{11},\\
b&\text{if }e\in E_{12},\\
c&\text{if }e\in E_{13},\\
1/2&\text{if }e\in E_2,\\
1/3&\text{if }e\in E_3.
\end{cases}
\end{equation}
In view of the above, by~\cite[Theorem 5.6]{HN03} the spectral dimension of ${\rm H}$ is given by
\begin{equation*}
d_S^\mu=\max_{J\in\mathcal{S}} d_S^\mu(J)=\max\Big\{\frac{2\log 6}{-\log(ra)},1,\frac{2\log 3}{\log 5}\Big\}.
\end{equation*}


More precisely,~\cite[Theorem 5.3]{HN03} yields the following spectral asymptotics for the hybrid ${\rm H}$.

\begin{theorem}\label{T:SpcSG3}
Let $N(x)$ denote the (Neumann) eigenvalue counting function of the Laplacian $\Delta_\mu$ on the hybrid ${\rm H}$ with base $\SG_3$, resistance parameters $r,\rho$ and measure parameters $a,b,c$. For $x$ large we have
\begin{enumerate}[leftmargin=.3in]
\item if $0<ra<\frac{1}{36}$, then 
\begin{equation*}
N(x)\sim x^{-\frac{\log 3}{\log 5}}G(\log x),
\end{equation*}
\item if $ra=\frac{1}{36}$, then
\begin{equation*}
N(x)\sim x^{\frac{\log 3}{\log 5}}\log x,
\end{equation*}
\item if $\frac{1}{36}<ra<\frac{1}{9}$, then
\begin{equation*}
N(x)\sim x^{\frac{\log 6}{-\log (ra)}}G(\log x),
\end{equation*}
where $G$ is a periodic function.
\end{enumerate}
\end{theorem}

For an eigenfunction $\lambda>0$, let $u_\lambda$ denote its corresponding eigenfunction. We finish this section by analyzing the relation between $N(x)$ and the counting function that considers only eigenfunctions whose associated eigenfunction is supported in one of the inverted $\SG$s of the hybrid ${\rm H}$. With the notation from Section~\ref{section:SG3}, define for each $x>0$ the function
\begin{multline*}
N'_{\SG}(x):=\#\{\lambda~\text{(D/N)-eigenvalue of }\Delta_\mu\text{ with }\supp u_\lambda\subset\SG_\alpha\\
\text{for some }\alpha\in\bigcup_{k\geq 1}\mcA_k\text{ and }~\lambda\leq x\}.
\end{multline*}
The choice of the resistance and measure parameters given in~\eqref{E:SpcSG3.01} and~\eqref{E:SpcSG3.02} and Proposition~\ref{prop:SG3.01} imply that $\lambda$ is an eigenvalue of the latter kind if and only if for some $k\geq 1$, that is the level where the copy $\SG_\alpha$ lives, $\lambda a^{-1} \big(\frac{ra}{3}\big)^k$ is an eigenvalue of the Laplacian on the usual $\SG$. The asymptotic behavior of this function is the same as the eigenvalue counting function of the usual $\SG$.

\begin{proposition}\label{P:SpcSG3.01}
For $x$ large,
\begin{equation*}
N'_{\SG}(x)\sim x^{\frac{\log 3}{\log 5}}G'(x),
\end{equation*}
where $G'$ is a periodic function.
\end{proposition}
\begin{proof}
Let $N_{\SG}(x)$ denote the eigenvalue counting function of the regular $\SG$. From~\cite[Theorem 2.4]{KL93} we know that $N_{\SG}(x)\sim x^{-\log 3/\log 5}G_{\SG}(\log x/2)$, where $G_{\SG}$ is a $\log 5/2$-periodic positive function. In view of the previous characterization of the eigenvalues counted by $N'_{\SG}$ we have
\begin{align*}
N'_{\SG}(x)&=\sum_{k\geq 1}^\infty\sum_{\alpha\in\mcA_k} N_{\SG}\Big(\Big(\frac{ra}{3}\Big)^k\!\frac{x}{a}\Big)=\sum_{k\geq 1}^\infty 3^kN_{\SG}\Big(\Big(\frac{ra}{3}\Big)^k\!\frac{x}{a}\Big)\nonumber\\
&\sim x^{\frac{\log 3}{\log 5}}\sum_{k\geq 1}^\infty \big((ra)^{\frac{\log 3}{\log 5}} 3^{1-\frac{\log 3}{\log 5}}\big)^kG_{\SG}\Big((k\log (ra/3)+\log x-\log a)/2\Big).
\end{align*}
Since $G_{\SG}$ is periodic and $0<ra<1/9$, we have that $(ra)^{\frac{\log 3}{\log 5}} 3^{1-\frac{\log 3}{\log 5}}<3^{1-3\frac{\log 3}{\log 5}}<1$ so that the series on the right hand side above is always convergent and a periodic function.
\end{proof}
By considering their difference, Theorem~\ref{T:SpcSG3} and Proposition~\ref{P:SpcSG3.01} allow us to deduce the asymptotic behavior of the function counting those eigenvalues supported also away from the inverted $\SG$s. 
\begin{corollary}
For $x$ large we have
\begin{equation*}
N(x)-N'_{\SG}(x)\sim x^{\frac{\log 3}{\log 5}}\widetilde{G}(x)
\end{equation*}
for some periodic function $\widetilde{G}$. 
\end{corollary}
Thus, the eigenvalue counting functions $N(x),N'(x)$ and their difference behave asymptotically in the same way up to a periodic function.
\appendix
\section{Definitions and background}
\subsection{Resistance forms}
We refer to~\cite{Kig01,Kig12} for further details. 
\begin{definition}\label{def:DB.RF01}
Let $X$ be a set and $\ell(X):=\{u\colon X\to X\}$. A resistance form on $X$ is a pair $(\E,\F)$ that satisfies the following properties.
\begin{itemize}[leftmargin=.45in]
\item[(RF1)] $\F$ is a linear subspace of $\ell(X)$ that contains constants, $\E$ is a non-negative symmetric quadratic form on $\F$ and $\E(u,u)=0$ if and only if $u$ is constant.
\item[(RF2)] For any $u,v\in\F$, define the relation of equivalence $u\sim v$ if and only if $u-v$ is constant. Then, $(X/_{\sim},\E)$ is a Hilbert space.
\item[(RF3)] $\F$ separates points of $X$.
\item[(RF4)] For any $x,y\in X$,
\[
R_{(\E,\F)}(x,y)=\sup\Big\{\frac{|u(x)-u(y)|^2}{\E(u,u)}~|~u\in\F,\,\E(u,u)>0\big\}<\infty.
\]
\item[(RF5)] For any $u\in\F$, the function $\overline{u}:=\min\{\max\{0,u\},1\}\in\F$ and $\E(\overline{u},\overline{u})\leq\E(u,u)$.
\end{itemize}
\end{definition}

\begin{definition}\label{def:DB.RF02}
A resistance form $(\E,\F)$ on $X$ is said to be local if for any $u,v\in\F$ such that $\inf\{R_{(\E,\F)}(x,y)~|~x\in\supp u, y\in\supp v\}>0$ it holds that $\E(u,v)=0$.
\end{definition}
\begin{lemma}{\cite[Lemma 8.2]{Kig12}}\label{lemma:DB.RF01}
For any non-empty set $Y\subseteq X$ define $\F|_Y:=\{u_{|_{_Y}}~|~u\in\F\}$. Then, for any $u\in\F_Y$ there exists a unique function $h_Y(u)\in\F$ such that $h_Y(u)_{|_Y}=u_{|_Y}$ and 
\[
\E(h_Y(u),h_Y(u))=\min\{\E(v,v)~|~v\in\F,~v_{|_Y}=u_{|_Y}\}.
\]
\end{lemma}

\begin{definition}\label{def:DB.RF03}
The function $h_Y(u)$ is called the $Y$-harmonic function with the boundary value $u$ and the space of $Y$-harmonic functions is denoted by $\mcH_{(\E,\F)}(Y)$.
\end{definition}
\begin{definition}\label{def:DB.RF04}
The pair $(\E|_Y,\F|_Y)$, where $\E|_Y(u,u):=\E(h_Y(u),h_Y(u,u))$ for any $u\in\F|_Y$ is called the trace of the resistance form $(\E,\F)$ on $Y$.
\end{definition}
One can prove, see~\cite[Theorem 8.4]{Kig12} that $(\E|_Y,\F|_Y)$ is a resistance form on $Y$.

\subsection{Quantum graphs}
Here we refer to~\cite{BK13} for further details. 
\begin{definition}
A quantum graph is a triple $(G, \mathcal H, B)$ where $G$ is a metric graph, $\mathcal H$ is a Hamiltonian and $B$ the boundary condition.
\end{definition}
A \textit{metric graph} is a graph $G=(V, E)$ equipped with a metric $d$ that assigns a length to each edge. A metric graph becomes a quantum one after being equipped with a \textit{Hamiltonian}, i.e. a differential (or sometimes more general) operator $\mathcal{H}$ on $\Gamma$. Common choices are $-\frac{d^2}{dx^2}$, $f(x)\rightarrow -\frac{d^2f}{dx^2}+V(x)f(x)$, and other self-adjoint operators.

\begin{definition}
In general, a boundary condition in a graph $G=(V, E)$ is defined so that for each vertex $v$ with degree $d_v$, $\exists$ $d_v\times d_v$ matrices $A_v$ and $B_v$ that satisfy:
\begin{enumerate}[leftmargin=.45in]
\item For a continuous function $F$ defined on each edge of $G$, denote $F(v)=[f_1(v), ..., f_{dv}(v)]^T$, $F^{\prime}(v) = [f_1^{\prime}(v), ..., f_{dv}^{\prime}(v)]^T$, then $A_vF(v)+B_vF(v) = 0, \forall v\in G$.
\item Putting together $A_v$ and $B_v$ horizontally we yield a matrix $[A_v|B_v]$ of full rank.
\end{enumerate}
\end{definition} 

The reason for defining boundary conditions can be explained as follows: For a Hamiltonian, the safest underlying space to consider is $\bigoplus_{e\in G}C^{\infty}_0(e)$, the space of smooth functions vanishing on all vertices. However, to make $\mathcal H$ self-adjoint, it is natural to consider Sobolev spaces $H^1_0(e)$ with zeros on the boundary. However, we want to extend to functions which do not vanish on all vertices, because otherwise we are simply dealing with functions on line segments regardless of the graph structure. The boundary condition is needed to get a nice function space to work with. One of the most commonly adapted boundary conditions is the\textit{Neumann condition}, which requires that functions are continuous on the vertices, and all derivatives sum up to zero, i.e. 
\[
A_v = \begin{bmatrix}
1&-1&0&...&0&0\\0&1&-1&...&0&0\\0&0&1&...&0&0\\...\\...\\0&0&0&...&1&-1\\0&0&0&...&0&0
\end{bmatrix}
\qquad\text{and}\qquad
B_v = \begin{bmatrix}
1&0&0&...&0&0\\0&1&0&...&0&0\\0&0&1&...&0&0\\...\\...\\0&0&0&...&1&0\\0&0&0&...&0&1
\end{bmatrix}.
\]
\bibliographystyle{amsplain}
\bibliography{RefsHyb.bib}
\end{document}